\documentclass[a4paper, 12pt]{article}
\usepackage[utf8]{inputenc}
\usepackage{ marvosym }
\usepackage{cmap}
\usepackage[T2A]{fontenc}
\usepackage{amsmath,amssymb,amsthm}
\usepackage{comment}
\usepackage[english]{babel}
\usepackage{circuitikz}

\definecolor{green}{RGB}{10, 150, 0}
\usepackage[pdftex,colorlinks=true,linkcolor=blue,urlcolor=red,unicode=true,hyperfootnotes=false,bookmarksnumbered]{hyperref}
\usepackage{indentfirst}
\usepackage{graphicx}
\usepackage{ dsfont }
\usepackage{enumitem}
\usepackage{color}
\usepackage[ruled,vlined,linesnumbered]{algorithm2e}
\usepackage{geometry}
\usepackage{pgf,tikz}\usetikzlibrary{arrows}
\usepackage{amsfonts}

\newtheorem{theorem}{Theorem}[section]
\newtheorem{proposition}[theorem]{Proposition}
\newtheorem{lemma}[theorem]{Lemma}

\newtheorem{corollary}[theorem]{Corollary}

\theoremstyle{definition}
\newtheorem{definition}[theorem]{Definition}
\newtheorem{example}[theorem]{Example}
\newtheorem{problem}[theorem]{Problem}

\theoremstyle{remark}
\newtheorem{remark}[theorem]{Remark}

\newcommand{\mscomm}[1]{\begingroup\color{green}#1\endgroup}

\long\def\comment#1\endcomment{}

\newcommand{\newnew}[1]{#1}
\newcommand{\bluenew}[1]{#1}
\newcommand{\bluevarnew}[1]{#1}

\newcommand{\bluevar}[1]{#1}

\begin{document}

	\title{Incidences, tilings, and fields}
	\author{ P.~Pylyavskyy,  M.~Skopenkov}
	\date{}
	\maketitle
        \begin{abstract}
            The master theorem, introduced by Richter-Gebert and \bluenew{generalized} by Fomin and the first author, provides a method for proving incidence theorems of projective geometry using triangular tilings of surfaces. We investigate which incidence theorems over ${\mathbb{C}}$ and ${\mathbb{R}}$ can or cannot be proved via the master theorem. For this, we formalize the notion of a tiling proof. We introduce a hierarchy of classes of theorems based on the underlying topological spaces. A key tool is considering the same theorems over finite fields.
            %
        \end{abstract}
	\let\thefootnote\relax\footnotetext{This work is supported by   KAUST baseline funding.
    }


\section{Introduction}
\label{sec:intro}


Incidence theorems about points and lines in the plane are at the core of projective geometry. Their variety is boundless. 
They have been linked to basic algebraic \cite{Hartshorne}, rational \cite{Amitsur-66}, and determinantal \cite{Richter-Gebert-Li-19} identities. See \cite[Chapter~3, Section~9]{ZS21} for an elementary introduction to incidence theorems.

A more recent look at incidence theorems has originated from Coxeter--Greitzer’s proof of Pappus’ theorem by multiple applications of Menelaus's theorem. Richter-Gebert \cite{Richter-Gebert-06} has visualized such proofs as triangular tilings of surfaces, resembling proofs of identities in geometric group theory~\cite{Olshanskii-89}, \bluenew{in particular, Lambek's polyhedral conditions~\cite{Bush-18}, and proof nets~\cite{Lamarche-Retore-98}}. 
Fomin and the first author~\cite{FP22} introduced a similar approach based on quadrilateral tilings 
and obtained numerous classical and new theorems in this way, also in higher dimensions. In what follows, we concentrate fully on triangular tiles, although analogous results should 
hold for quadrilateral tiles. See a quick introduction in Section~\ref{ssec:quick}. 


This has led to the \bluenew{\emph{completeness}} question of whether all incidence theorems arise from tilings. 
This question for triangular tilings was addressed by Barali\'c et al. \cite{Baralic-etal-20} who introduced a formalization of tiling proofs (called the \emph{Menelaus system}) and provided examples of incidence theorems unprovable in their setup. Their formalization was rather restricted (for instance, it did not include the use of the incidence axiom in the proofs), leaving the question of whether \bluenew{completeness} holds in a more refined setup. This was the starting point of the present work; 
see also \cite{Izosimov-26,Izosimov-Pylyavskyy-25,KL} for recent further progress in this direction and~\cite{Richter-Gebert-26} for an overview.

In this paper, we show applications of tiling proofs far beyond the original scope: in addition to generating incidence theorems, we can now 
efficiently construct counterexamples to them, 
and study their dependence on the ground field. 
We give unexpected links of incidence theorems to commutative algebra, geometric group theory, piecewise-linear and algebraic topology, and even lattice gauge theory. One of the main tools borrowed from algebraic geometry is considering the same theorems over finite fields. Our main results are listed in Section~\ref{ssec:contributions}. 

\subsection{Quick Start} \label{ssec:quick}

The concept of a tiling proof of an incidence theorem should be clear from the following two examples, which are restatements of Desargues' and Pappus' theorems.

\begin{example}[Desargues' theorem] \label{ex-desargues} (See Figure~\ref{fig-desargues} to the left.)
    Let $P_1,\dots,P_4\bluevarnew{\in\mathbb{R}^2}$ be four points in the plane such that no three of them are collinear. On each line $P_iP_j$, where $i\ne j$, pick a point $P_{ij}$ distinct from $P_i$ and $P_j$. If the triple of points on the extensions of 
    the sides of each triangle $P_1P_2P_3$, $P_2P_3P_4$, $P_3P_4P_1$ is collinear, then the same holds for 
    $P_1P_2P_4$, i.e. $P_{12}\in P_{14}P_{24}$.
\end{example}

\begin{proof} \cite[Example~8.4]{FP22}
For collinear points $A,B,C$, 
denote by \bluenew{$[AB/AC]$} the unique $k\in\mathbb{R}$ such that \bluenew{$\overrightarrow{AB}=k\cdot\overrightarrow{AC}$}.
Applying Menelaus's theorem three times, we can write 
$$\left[\frac{P_{12}P_1}{P_{12}P_2}\right] \cdot \left[\frac{P_{23}P_2}{P_{23}P_3}\right] \cdot \left[\frac{P_{13}P_3}{P_{13}P_1}\right] = 1,$$
$$\left[\frac{P_{13}P_1}{P_{13}P_3}\right] \cdot \left[\frac{P_{34}P_3}{P_{34}P_4}\right] \cdot \left[\frac{P_{14}P_4}{P_{14}P_1}\right] = 1,$$
$$\left[\frac{P_{23}P_3}{P_{23}P_2}\right] \cdot \left[\frac{P_{24}P_2}{P_{24}P_4}\right] \cdot \left[\frac{P_{34}P_4}{P_{34}P_3}\right] = 1.$$
Multiplying the three equalities, we get 
$$\left[\frac{P_{12}P_1}{P_{12}P_2}\right] \cdot \left[\frac{P_{24}P_2}{P_{24}P_4}\right] \cdot \left[\frac{P_{14}P_4}{P_{14}P_1}\right] = 1,$$
which again by Menelaus's theorem means that the points $P_{12}, P_{14}, P_{24}$ are collinear.
\end{proof}

Note now that this proof can be conveniently visualized with the help of a tetrahedron on the right of Figure~\ref{fig-desargues}. Specifically, let us put a $+$ next to a \bluenew{segment} if the corresponding length appears in the numerator of one of the original three equalities, $-$ if in the denominator. Then the \bluenew{segments} that get both a $+$ and a $-$ cancel out, and we are left with the desired equality associated with face $P_1P_2P_4$. This suggests that one may be able to obtain other theorems by considering other triangulations of surfaces. This is a variation of the main idea of \cite{Richter-Gebert-06, FP22}.



\begin{figure}[htbp]
    \centering
    \includegraphics[width=0.85\textwidth]{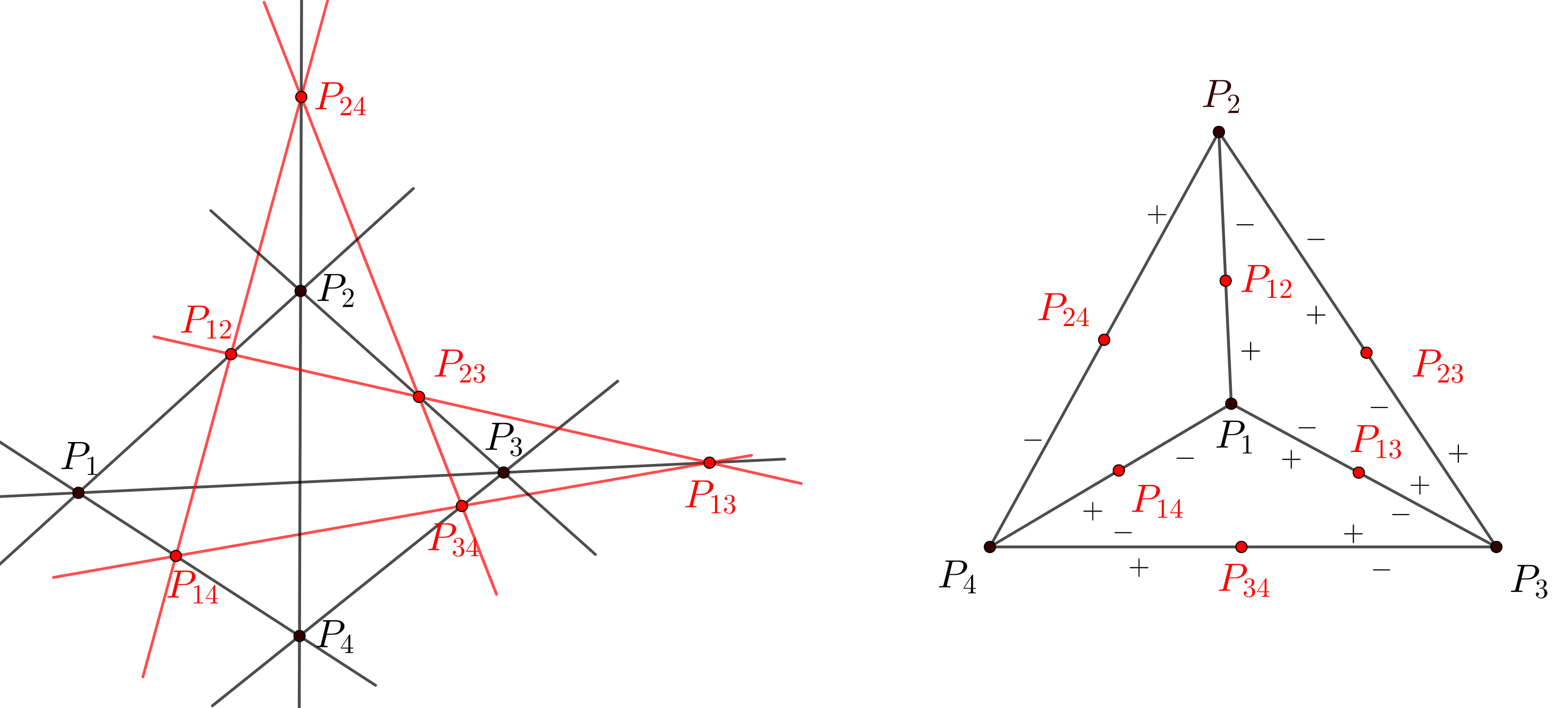}
    \caption{Desargues' configuration (left) and a tiling of a sphere (right). See Example~\ref{ex-desargues}.}
    \label{fig-desargues}
\end{figure}


\begin{example}[Pappus' theorem] \label{ex-pappus} (See Figure~\ref{fig-pappus-intro}.)
    Let two lines $L_2$ and $L_3$ contain distinct points $P_2,P_3,P_4$ and $P_5,P_6,P_7$ respectively, not contained in $L_2\cap L_3$. Let \bluenew{the} line $L_4$ pass through $P_8= P_2P_6\cap P_3P_5$ and $P_9= P_3P_7\cap P_4P_6$. Then the line $L_1=P_2P_7$ passes through $P_1=L_4\cap P_4P_5$. 
\end{example}

\begin{proof}[Proof in the case when $L_2,L_3,L_4$ are non-concurrent]
(Coxeter--Greitzer; cf.~\cite[Example~8.5]{FP22}.)
Denote $P_{10}:= L_2\cap L_3$, $P_{11}:= L_3\cap L_4$, $P_{12}:= L_4\cap L_2$; see Figure~\ref{fig-pappus-intro} to the left.

Applying Menelaus's theorem five times we can write 
$$\left[\frac{P_{1}P_{12}}{P_{1}P_{11}}\right] \cdot \left[\frac{P_{5}P_{11}}{P_{5}P_{10}}\right] \cdot \left[\frac{P_{4}P_{10}}{P_{4}P_{12}}\right] = 1,$$
$$\left[\frac{P_{8}P_{12}}{P_{8}P_{11}}\right] \cdot \left[\frac{P_{6}P_{11}}{P_{6}P_{10}}\right] \cdot \left[\frac{P_{2}P_{10}}{P_{2}P_{12}}\right] = 1,$$
$$\left[\frac{P_{9}P_{12}}{P_{9}P_{11}}\right] \cdot \left[\frac{P_{7}P_{11}}{P_{7}P_{10}}\right] \cdot \left[\frac{P_{3}P_{10}}{P_{3}P_{12}}\right] = 1,$$
$$\left[\frac{P_{9}P_{11}}{P_{9}P_{12}}\right] \cdot \left[\frac{P_{6}P_{10}}{P_{6}P_{11}}\right] \cdot \left[\frac{P_{4}P_{12}}{P_{4}P_{10}}\right] = 1,$$
$$\left[\frac{P_{8}P_{11}}{P_{8}P_{12}}\right] \cdot \left[\frac{P_{5}P_{10}}{P_{5}P_{11}}\right] \cdot \left[\frac{P_{3}P_{12}}{P_{3}P_{10}}\right] = 1.$$

Multiplying the five equalities, we get 
$$\left[\frac{P_{1}P_{12}}{P_{1}P_{11}}\right] \cdot \left[\frac{P_{7}P_{11}}{P_{7}P_{10}}\right] \cdot \left[\frac{P_{2}P_{10}}{P_{2}P_{12}}\right] = 1,$$
which again by Menelaus's theorem means that the points $P_{1}, P_{2}, P_{7}$ are collinear.
\end{proof}

We consider the case when $L_2,L_3,L_4$ are concurrent and complete the proof in Section~\ref{ssec:examples} \newnew{(this is the first tiling proof of Pappus' theorem in full generality).}


\begin{figure}[htb]
    \centering
\includegraphics[width=0.6\textwidth]{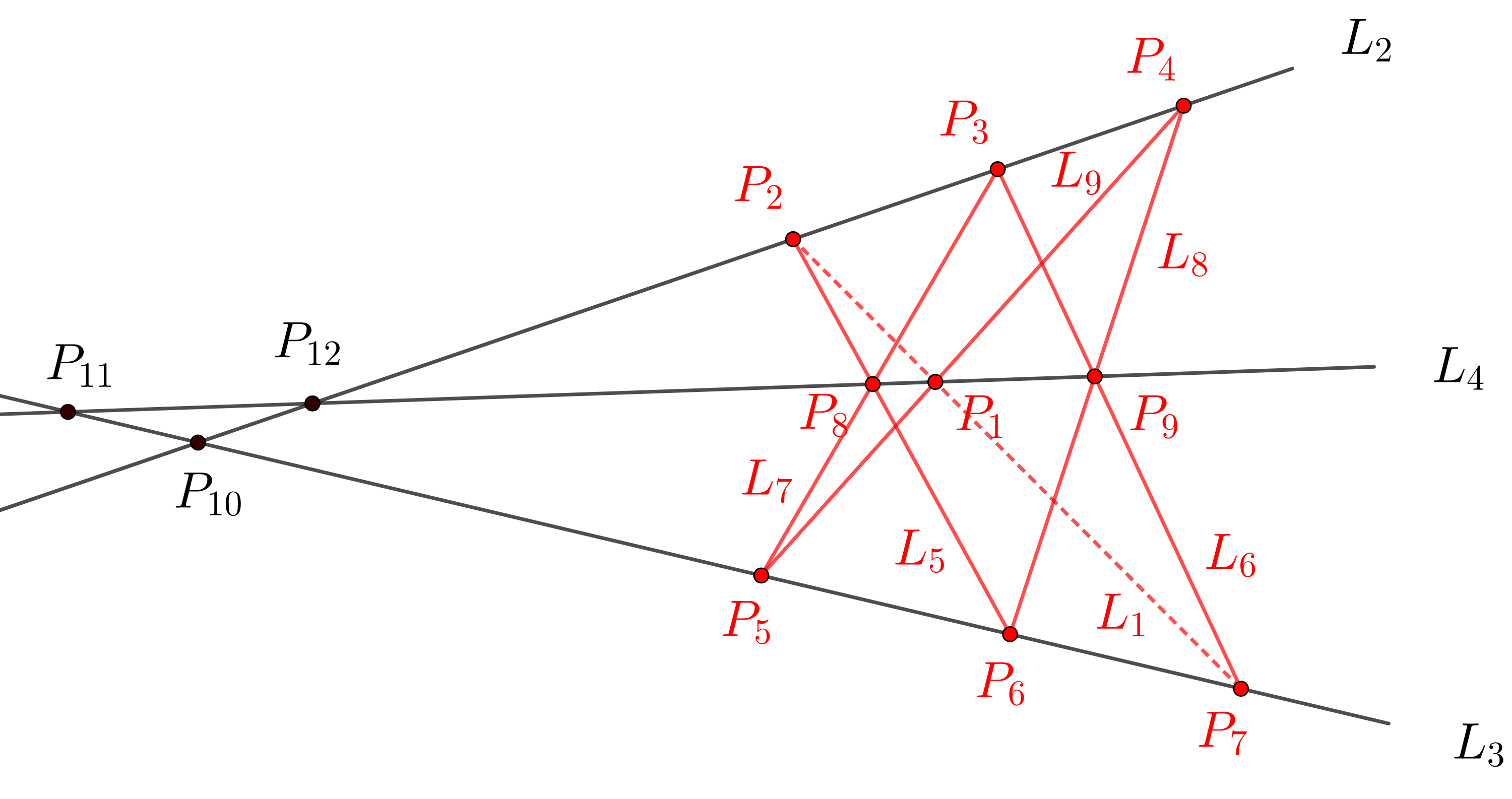}
\includegraphics[width=0.35\textwidth]{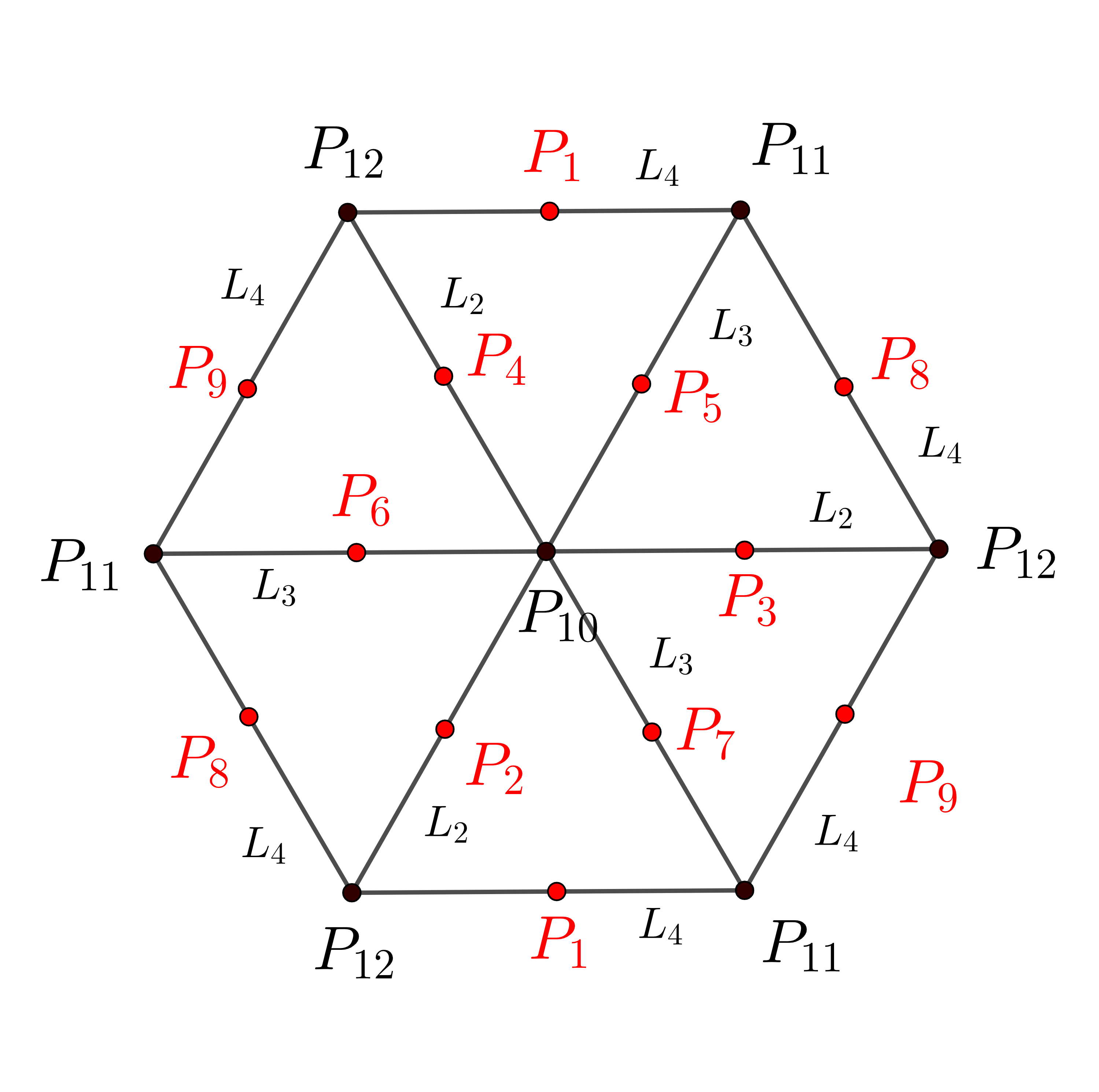}
    \caption{Pappus' configuration (left) and a tiling of a torus (right). 
    The opposite sides of the hexagon to the right are identified.
    See Example~\ref{ex-pappus}.
    }
    \label{fig-pappus-intro}
\end{figure}

Again, this proof can be conveniently visualized with the help of a torus glued out of six triangles on the right of Figure~\ref{fig-pappus-intro}. Each instance of the Menelaus theorem corresponds to one of the six triangles, and the fact that the cancellation works out the way it does is seen from each edge of this tiling occurring in exactly two triangles. For example, the Menelaus theorem for the top triangle has a term $\left[{P_{5}P_{11}}/{P_{5}P_{10}}\right]$ that corresponds to one of the three sides of this triangle. This term cancels out with the term $\left[{P_{5}P_{10}}/{P_{5}P_{11}}\right]$ associated with the same side of the neighboring triangle.


\subsection{Contributions}
\label{ssec:contributions}

We introduce a 
Master theorem  
(Theorem~\ref{th-master-theorem-general}), generating more incidence theorems 
over a given field, such as ${\mathbb{R}}$, using simplicial complexes rather than surfaces. Depending on the shape of the simplicial complex, 
we introduce the following classes of incidence theorems over~${\mathbb{R}}$:
$$
\text{
\begin{tabular}{|c|}
\hline
     sphere-tiling   \\
     provable \\
\hline     
\end{tabular}
$\subsetneqq$
\begin{tabular}{|c|}
\hline
     surface-tiling   \\
     provable \\
\hline     
\end{tabular}
$\subsetneqq$
\begin{tabular}{|c|}
\hline
     grope-tiling   \\
     provable \\
\hline     
\end{tabular}
$\subsetneqq$
\begin{tabular}{|c|}
\hline
     simplicial-complex    \\
     provable \\
\hline     
\end{tabular}
$\subsetneqq$
\begin{tabular}{|c|}
\hline
     all true    \\
     theorems \\
\hline     
\end{tabular}.
}
$$
Over $\mathbb{C}$,  
only the first and last inclusions are proper, and 
the rest are conjectured to be equalities.
Here, the (generalized) gropes are defined in Section~\ref{ssec:gropes}; see \cite{Teichner} for a concise introduction.

In particular, we 
give examples of the following incidence theorems (in parentheses, the key idea of the proof is presented):
\begin{description}
\item[Example~\ref{ex-pappus}:]
a theorem over $\mathbb{C}$ provable by a tiled torus but not a 
sphere (because it does not hold over a skew field).
\item[Example~\ref{ex-Fano}:] a theorem over $\mathbb{C}$ unprovable by a tiled 
surface nor by a simplicial complex
(because it does not hold over the field with $2$ elements);
\item[Example~\ref{ex-9-gon}:] a theorem over $\mathbb{R}$ that is provable by a simplicial complex (actually, a generalized grope) but not a tiled 
surface 
(because it does not hold over $\mathbb{C}$);
\item[Example~\ref{ex-Hesse}:] a theorem true over $\mathbb{R}$, but not $\mathbb{C}$, that is unprovable by any simplicial complex 
(because it does not hold over the field with $3$ elements);
\item[Example~\ref{ex-F4}:]
a theorem true over both $\mathbb{R}$ and $\mathbb{C}$, provable by a simplicial complex over $\mathbb{R}$ 
but not $\mathbb{C}$ (because it does not hold over the field with $4$ elements);
\item[Example~\ref{ex-non-grope}:] a theorem over $\mathbb{R}$ that is provable by a simplicial complex but not a grope (because it does not hold over the field with $5$ elements);
\end{description}



\subsection{Organization of the paper}

In Section~\ref{sec:preliminaries}, we define incidence theorems and tiling proofs. Surprisingly, we did not find this definition in the literature.
Although a few examples 
are enough to provide insight into what a tiling proof is, a precise definition 
is vital to show that some theorem has \emph{no} tiling proof. For this purpose, we need a mathematical-logic level of rigor  throughout 
(while keeping 
geometric language). 
However, 
our 
results do not rely on a particular definition of tiling proofs and hold for any definition such that the Master Theorem is true (see Theorems~\ref{th-master-theorem}, \ref{th-master-theorem-general}, \ref{th-non-commutative-master-theorem}). 
A reader 
ready to accept the truth of the Master Theorem(s) 
can skip Section~\ref{sec:preliminaries} entirely. 

In Section~\ref{sec:complex}, we present a few warm-up results on complex geometry.

In Section~\ref{sec:real}, we present our main results on geometry over real numbers and general fields. In particular, we introduce a new notion of simplicial-complex proofs 
(see Definition~\ref{def-simplicial-complex-proof}). 
A reader 
ready to accept the truth of the Master Theorem~\ref{th-master-theorem-general} 
can skip this technical definition.

In Section~\ref{sec:skew}, we present a few variations concerning skew fields.

Although our results have connections to 
topology and 
algebra, we do not assume much knowledge of those subjects. We are going to use only the following basic facts about fields. The residues modulo 
a prime number $p$ form the field $\mathbb{F}_p=\mathbb{Z}/p\mathbb{Z}$ with $p$ elements.
There exists also a field $\mathbb{F}_4$ with four elements.
The nonzero elements of any field $\mathbb{F}$ form a group $\mathbb{F}^*$ with respect to multiplication. For any field $\mathbb{F}$, there is a field of rational functions $\mathbb{F}(X)$ and the ring of polynomials $\mathbb{F}[X]$ with coefficients in $\mathbb{F}$. The latter ring is a unique factorization domain. 
All necessary results about skew fields are recalled in Appendix~\ref{sec:auxiliary}.

\section{Foundations}
\label{sec:preliminaries}



\subsection{Incidence theorems}

An incidence theorem asserts that a collection of incidences and non-incidences between several lines and points in the plane implies another incidence. Let us make this notion 
precise. 

Throughout, 
we consider the projective plane $P^2$ over a field $\mathbb{F}$ (usually $\mathbb{R}$ or $\mathbb{C}$). Denote by $P^{2*}$ the set of 
\bluenew{lines} on 
$P^2$. 
A point $P\in P^2$ is \emph{incident} to a line $L\in P^{2*}$ if $P\in L$.

\begin{definition}[Incidence Matrix] \label{def-incidence-matrix}
    Let $M$ be an $m\times n$ matrix with the entries $\pm 1,0$. 
    A finite sequence of points $P_1,\dots,P_m\in P^2$ and lines $L_1,\dots,L_n\in P^{2*}$ has \emph{incidence matrix}~$M$, if $M_{ij}=1$ implies $P_i\in L_j$ and $M_{ij}=-1$ implies $P_i\notin L_j$ for each 
    $1\le i\le m$ and $1\le j\le n$. 
\end{definition}

Notice that there is no condition on the incidence between $P_i$ and $L_j$ if $M_{ij}=0$. (This is somewhat similar to three-valued logic.) In particular, the zero matrix is always an incidence matrix, and the incidence matrix is not uniquely determined by the sequences $P_1,\dots,P_m$ and $L_1,\dots,L_n$. Also, we allow repeating points or lines in the sequences. 

We view the incidences and non-incidences between all pairs $P_i$ and $L_j$ such that $M_{ij}\ne 0$ as given and ask if they imply the 
incidence $P_1\in L_1$; cf.~\cite[\S3.1.1]{Richter-Gebert-Li-19}.

\begin{definition}[Incidence Theorem] \label{def-incidence-theorem}
The \emph{incidence theorem with the matrix $M$}, \bluenew{or the \emph{incidence predicate with the matrix $M$}}, is the predicate 
\begin{multline*}
\hspace{-0.7cm}\forall P_1,\dots,P_m\in P^2 \quad \forall L_1,\dots,L_n\in P^{2*}: P_1,\dots,P_m, L_1,\dots,L_n\text{ has incidence matrix $M$} 
\implies P_1\in L_1
\end{multline*}
\end{definition}

\emph{No} other variables (besides $P_1,\dots,P_m,L_1,\dots,L_n$), relations (besides `$\in$' and `has incidence matrix'), logical operators (besides `$\Rightarrow$') and quantifiers (besides `$\forall$') are allowed in our definition. 
\bluenew{Predicates of this form are known as \emph{quasi-identities}. 
Remarkably, quasi-identities also appear in the tiling approach to geometric group theory, e.g., \emph{Lambek polyhedral conditions} \cite{Bush-18}.}

\bluenew{Beware that} an incidence theorem is not necessarily a true predicate and its truth may depend on the ground field $\mathbb{F}$. For instance, consider two incidence theorems with the matrices
\begin{equation}\label{eq-M_1}
M=\left(\begin{matrix}
    0 & 1  \\
    1 & 1  \\
    1 & 1
\end{matrix}\right)
\qquad\text{and}\qquad
M'=\left(\begin{matrix}
    0 &  1 & 0  \\
    1 &  1 & 1  \\
    1 &  1 & -1
\end{matrix}\right).
\end{equation}
Geometrically, the former means ``There is a unique line through two given points'' and the latter means ``There is a unique line through two given \emph{distinct} points'' (because adding the column $(0,1,-1)^{\mathrm{T}}$ is equivalent to the requirement $P_2\ne P_3$). The former incidence theorem is false, and the latter is true (over any field). 
The latter is called 
the \emph{incidence axiom}. This is a slight abuse of terminology because we work with a projective plane over a field, not an axiomatically defined projective plane; hence, this assertion 
is a theorem rather than an axiom. 

More generally, any incidence theorem with a matrix having no $-1$ entries (and $M_{11}=0$) 
is false: a counterexample is a collection $P_1\ne P_2=\dots =P_{m}, L_1\ne L_2=\dots=L_{n}$ such that $P_1,P_2\in L_2$, $P_2\in L_1$, $P_1\notin L_1$. So, one needs at least one non-incidence 
for a new incidence.

An incidence theorem is tautologically true if $M_{11}=1$ but need not be false
if $M_{11}=-1$.

An incidence theorem can be \emph{vacuous} in the sense that no sequence of points and lines has incidence matrix $M$. A vacuous theorem is tautologically true. For instance, replacing the entry $M’_{11}=0$ with $M’_{11}=-1$ in \eqref{eq-M_1} leads to a vacuous theorem, which is true despite $M’_{11}=-1$.

In what follows, we state theorems in a human-readable form, without introducing the matrix $M$ explicitly. The understood matrix $M$ can always be easily reconstructed, up to slight ambiguity. For instance, the assumption that the line $L_k$ \emph{passes through} the points $P_i$ and $P_j$ just means that the $k$-th column of $M$ has ones at positions $i$ and $j$ and zeroes at all the other positions. Analogously, one encodes 
that two lines $L_i$ and $L_j$ \emph{intersect} at a point $P_k$. The assumption that the points $P_i$ and $P_j$ are \emph{distinct} is encoded by appending a column with $1$ at position $i$, $-1$ at position $j$, and zeroes at the other positions (that is, adding an auxiliary line incident with $P_i$ and non-incident with $P_j$). Analogously, one encodes \emph{distinct lines}, \emph{collinear points}, \emph{triangles}, etc. 
We often apply trivial logical implications, such as using the incidence axiom to make the matrix $M$ more compact. As a rule, we use notation different from $P_i$ and $L_j$ for the points and lines.
With these conventions, Desargues' and Pappus' theorems in Examples~\ref{ex-desargues} and~\ref{ex-pappus} are examples of incidence theorems. Let us give a simpler example.


\begin{example}\label{ex-warmup} (See Figure \ref{fig:line})
    Let \bluenew{the} points $P_1,P_2,P_3,P_4$ lie on a line $L_2$. Let \bluenew{the} line $L_1$ pass through $P_2$ and $P_3$, and \bluenew{the} line $L_3$ pass through $P_3$ and $P_4$ but not $P_1$. If $P_2\ne P_4$ then $L_1$ passes through~$P_1$.
\end{example}

\begin{figure}[htbp]
    \centering
\includegraphics[width=0.6\textwidth]{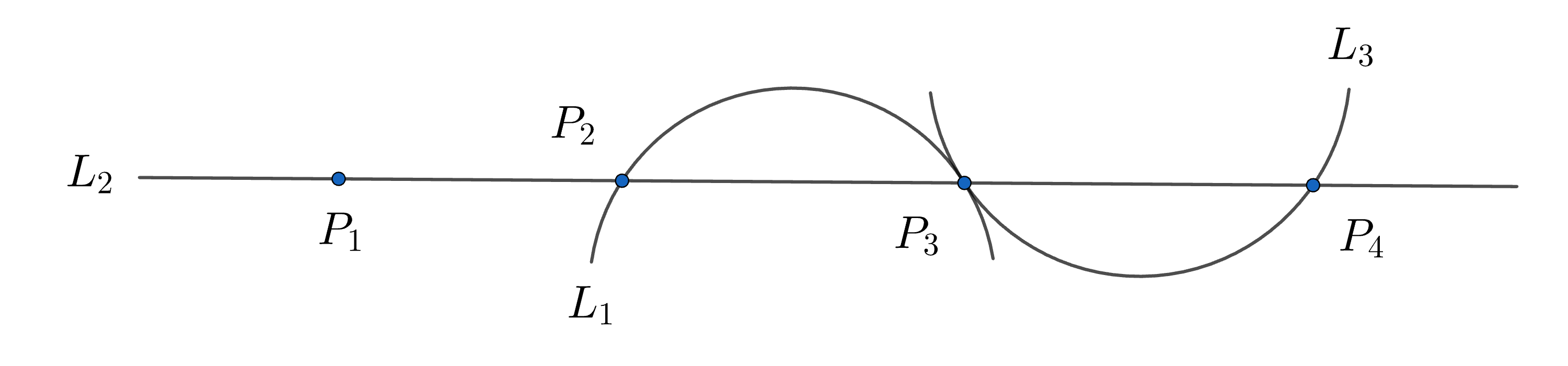}
    \caption{A simple theorem on a line illustrated by Example \ref{ex-warmup} }
    \label{fig:line}
\end{figure}

Here (up to slight ambiguity)
$$
M=
\left(
\begin{smallmatrix}
     0 & 1 & -1 &  0 \\
     1 & 1 &  0 & -1 \\
     1 & 1 &  1 &  0 \\
     0 & 1 &  1 &  1 
\end{smallmatrix}
\right).
$$
Here points correspond to rows and lines to columns. The fact that $P_2 \not = P_4$ is included via an axillary line passing through $P_4$ but not $P_2$; this line corresponds to the fourth column in~$M$. 

Example~\ref{ex-warmup} is true over any field, and we illustrate how to show this using the matrix form.
Indeed, observe that the submatrix of $M$ obtained by removing row 2 and column 4 is the same as $M'$ in~\eqref{eq-M_1} up to taking the transpose and permutation of rows. By the incidence axiom, we get $P_4\in L_1$. Then put $1$ into the entry $(4,1)$ of $M$. After that, 
removing row 3 and column 3 gives $M'$ up to permutation of rows. By the incidence axiom again, we get $P_1\in L_1$. Such arguments are easy to make automated, and we do it in 
Section~\ref{ssec:examples}.

Let us give an example of an incidence theorem true over the field with $q$ elements and false over any field with more elements:

\begin{example}\label{ex-q-points} (Line over the field with $q$ elements) Let $P_1,\dots,P_{q+1}$ be distinct points on a line~$L_2$. If a line $L_1$ does not pass through $P_2,\dots,P_{q+1}$ then it passes through $P_1$.    
\end{example}

Here $M$ is the $(q+1)\times (q+2)$ matrix (determined by the 
statement up to a slight ambiguity)
$$
M=
\left(
\begin{smallmatrix}
     0 & 1 & -1 & -1 & \hdots & -1 \\
    -1 & 1 & 1  & -1 & \hdots & -1 \\
    -1 & 1 & -1 &  1 & \hdots & -1 \\
    \vdots & \vdots & \vdots & \vdots & \ddots & \vdots \\
    -1 & 1 & -1 & -1 & \hdots &  1  
\end{smallmatrix}
\right);
$$
we encode the condition that $P_1,\dots,P_{q+1}$ are distinct by 
auxiliary lines $L_3,\dots,L_{q+2}\ne L_2$ passing through $P_2,\dots,P_{q+1}$ respectively. 
\bluenew{From the geometric meaning, }
we see that the truth of the incidence theorem with the matrix $M$ depends on the field. 

\begin{remark}\label{rem-extension}
    If an incidence theorem 
    is true over a field 
    then it is true over any sub-field.
\end{remark}


\subsection{Tiling proofs}

%
%
%
%
%
%
%

Now we formalize a method to generate (and prove) incidence theorems, discovered in~\cite{FP22,Richter-Gebert-06}.

Recall the proof of Desargues' theorem in Section~\ref{ssec:quick}. It had the following key ingredients. The triangles to which we applied Menelaus's theorem matched to form a triangulated surface. The lines in the theorem were their sides plus one additional (red) line per triangle. The points were their vertices plus a (red) point on each side. The resulting correspondence between vertices/edges/faces and points/lines preserved incidences except for the red line corresponding to a face and the point corresponding to its vertex. We summarize this construction as follows. 

\begin{definition}[Elementary surface-tiling proof] 
    \label{def-elementary-surface-tiling-proof}
    Consider an incidence theorem with an $m\times n$ matrix $M$. 
    An \emph{elementary surface-tiling proof} of the theorem is 
    a triangulated 
    closed orientable surface 
    equipped with
    two maps
    $$
    p\colon V\sqcup E\to \{1,\dots,m\}
    \qquad\text{and}\qquad
    l\colon F\sqcup E\to \{1,\dots,n\},
    $$
    where $V$, $E$, $F$ are the sets of vertices, edges, faces respectively,
    satisfying the 
    properties:
    \begin{enumerate}
        \item[(0)] there is a unique pair $i\in E,j\in F$ such that $i\subset j$ and $p(i)=l(j)=1$;
        \item[($+1$)] for any other pair $i\!\in\! E, j\!\in\! F$ or $i\!\in\! V\sqcup E,j\!\in\! E$
        such that $i\subset j$ we have $M_{p(i)l(j)}=1$;
        \item[($-1$)] for any $i\!\in\! V\sqcup E,j\!\in\! E$  contained in one face and such that $i\not\subset j$, we have $M_{p(i)l(j)}=-1$.
    \end{enumerate}
    
    The face $j$ from property~(0) is called the \emph{marked face}.
\end{definition}

Notice that there are no restrictions on $M_{p(i)l(j)}$ for $i$ and $j$ not contained in one face, and no restrictions on $M_{11}$ (although the most interesting case is $M_{11}=0$).  

The maps $p$ and $l$ assign (the indices of) \emph{p}oints and \emph{l}ines, respectively, to the vertices/ edges/faces of the triangulation. 
In the figures, we depict them 
by labeling vertices/edges/faces. A label $P_k$ at a vertex or edge $i$ means that $p(i)=k$, and a label $L_k$ at a face or edge $j$ means that $l(j)=k$. See Figure~\ref{fig:delta-triangulation} to the left. \bluenew{Beware that the red points depict edge midpoints; they are not vertices of the triangulation.} Further, the labels~$L_k$ are usually omitted when the map~$l$ is reconstructed from property~($+1$), and the points are usually denoted by other letters. 
An elementary surface-tiling proof of Desargues' theorem 
is shown in Figure~\ref{fig-desargues} to the right. 

The maps $p$ and $l$ need not be injective or surjective. However, in an elementary surface-tiling proof of an incidence theorem, one can always replace them with bijections. This leads to an elementary surface-tiling proof of a more general theorem: The former theorem is a particular case of the latter, where certain points coincide. 


\begin{lemma}[Elementary Lemma; see {\cite[Corollary~8.3]{FP22} and~\cite[p.~9]{Richter-Gebert-06}}] 
\label{l-elementary-lemma}
If an incidence theorem with some matrix 
has an elementary surface-tiling proof, then it is true (over any field). 
\end{lemma}


The lemma holds even for $M_{11}=-1$, when it means that the incidence theorem is vacuous.

The point of Lemma~\ref{l-elementary-lemma} is a systematic \emph{generation} of incidence theorems rather than their effective proof. Given an arbitrary surface tiling and \emph{bijections} $p$ and $l$ satisfying property~(0), one generates a true incidence theorem with the matrix $M$ determined by properties~($+1$) and~($-1$), and zeroes at all the other entries. 

The lemma follows from \cite[Corollary~8.3]{FP22} but we present a direct elementary proof.

 \begin{figure}[htbp]
    \centering
 \includegraphics[width=0.8\textwidth]{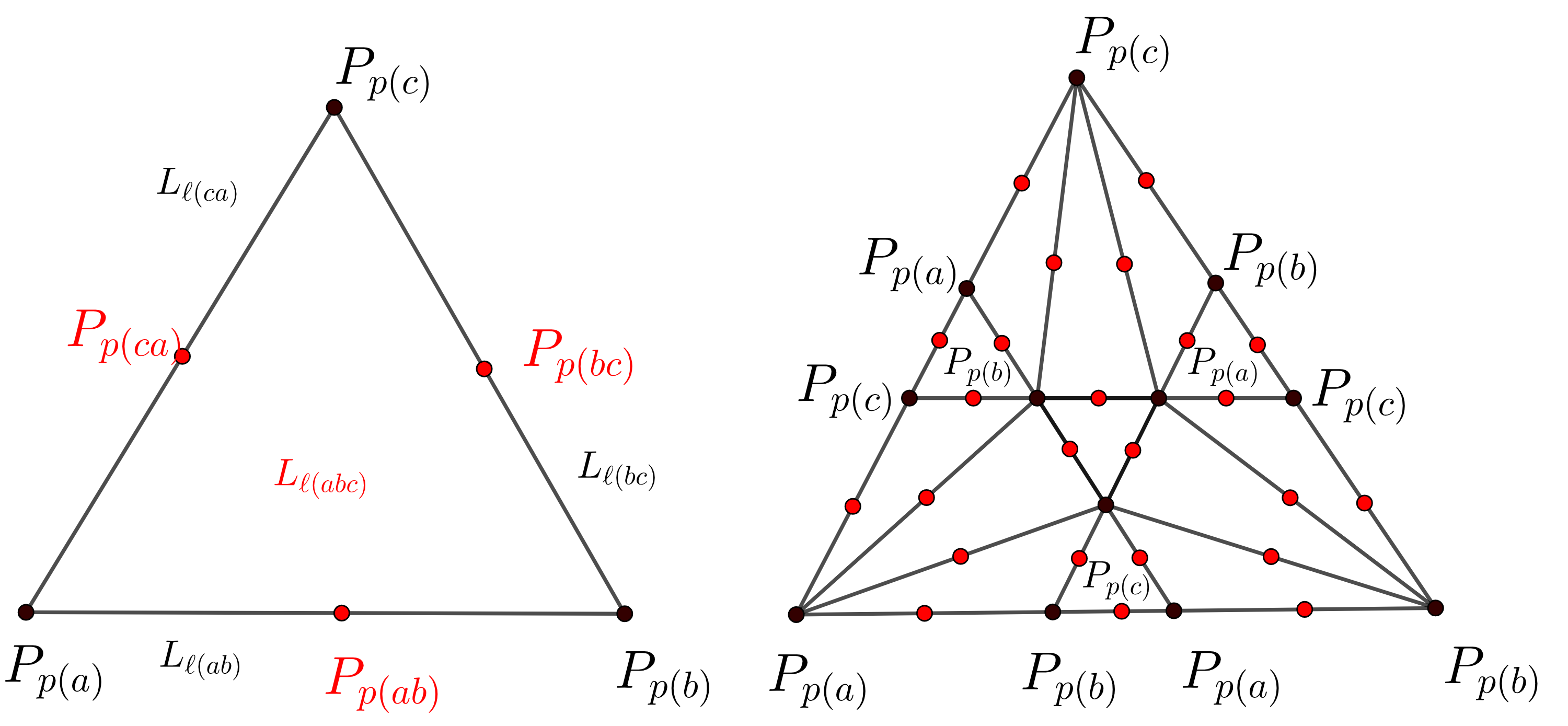}
    \caption{A face labeled by points and lines (left) and its octahedral subdivision (right). The latter has the combinatorics of the seven faces of an octahedron, three of which are subdivided further unless they have a common edge with the marked face.  \bluevar{Beware that the red points depict edge midpoints; they are not vertices of the triangulation.}
        See Definition~\ref{def-elementary-surface-tiling-proof} and  Remark~\ref{rem-triangulation}.
    }
    \label{fig:delta-triangulation}
\end{figure}

\begin{proof} (Cf.~\cite[Section~2.2]{Richter-Gebert-06}.)
    Let an incidence theorem with an $m\times n$ matrix $M$ have an elementary surface-tiling proof. Take an arbitrary sequence of points $P_1,\dots,P_m\in P^2$ and lines $L_1,\dots,L_n\in P^{2*}$ having incidence matrix $M$. Let us prove that $P_1\in L_1$.

    Assume without loss of generality that the ground field $\mathbb{F}$ is infinite. Otherwise, take an infinite extension of $\mathbb{F}$ and apply Remark~\ref{rem-extension}.
    %
    Over an infinite field, there always exists a line not passing through the points $P_1,\dots,P_m$. Taking the line to infinity by 
    a projective transformation, we may assume that $P_1,\dots,P_m$ lie in the affine plane $\mathbb{F}^2$.

    Take the triangulation of the closed orientable surface from Definition~\ref{def-elementary-surface-tiling-proof}. Fix an orientation of the surface; this specifies a counterclockwise cyclic ordering of the vertices of each face.

    Take an arbitrary face $abc\in F$ with the vertices listed counterclockwise; see Figure~\ref{fig:delta-triangulation} to the left. Let us show that $P_{p(a)},P_{p(b)},P_{p(c)}$ are vertices of a triangle and the line $L_{l(abc)}$ does not pass through them. By property ($+1$), the line $L_{l(ab)}$ passes through $P_{p(a)}$ and $P_{p(b)}$ because the sequence of points and lines has incidence matrix $M$. Analogously, $P_{p(b)},P_{p(c)}\in L_{l(bc)}$ and $P_{p(c)},P_{p(a)}\in L_{l(ca)}$. By property ($-1$), the line $L_{l(ab)}$ does not pass through $P_{p(c)}$.
    Analogously, $P_{p(a)}\notin L_{l(bc)}$ and $P_{p(b)}\notin L_{l(ca)}$. Hence, $P_{p(a)},P_{p(b)},P_{p(c)}$ are distinct and form a triangle. By properties ($+1$) and ($-1$), points 
    $P_{p(ab)}$, $P_{p(bc)}$, $P_{p(ca)}$ lie on (the extensions of) the sides of the triangle and are distinct from the vertices. 
    By properties (0) and ($+1$), there is at most one pair $i\in E,j\in F$ such that $i\subset j$ and $M_{p(i)l(j)}\ne 1$. Hence the line $L_{l(abc)}$
    contains at least two of the points $P_{p(ab)}$, $P_{p(bc)}$, $P_{p(ca)}$. Hence, it does not pass through the vertices (otherwise we get two distinct lines through two distinct points).
    
    Then we can apply Menelaus's theorem and conclude that
    \begin{equation*}
        \left[\frac{P_{p(a)}P_{p(ab)}}{P_{p(b)}P_{p(ab)}}\right]\cdot
        \left[\frac{P_{p(b)}P_{p(bc)}}{P_{p(c)}P_{p(bc)}}\right]\cdot
        \left[\frac{P_{p(c)}P_{p(ca)}}{P_{p(a)}P_{p(ca)}}\right]=1
    \end{equation*}
    for all faces $abc$ except for the 
    face $j$ appearing in the unique pair $(i,j)$ in property~(0). Here for collinear $A,B,C\in\mathbb{F}^2$, we denote by $[AB/CB]$ the unique $k\in\mathbb{F}$ such that $A-B=k(C-B)$.

    Multiplying such equations over all faces but $j$, we get the same equation for the face~$j$, because we have a triangulation of a closed orientable surface. For $abc=j$, by Menelaus's theorem, the points 
    $P_{p(ab)}$, $P_{p(bc)}$, $P_{p(ca)}$ 
    are collinear. One of them is $P_1$ and the other two are distinct from $P_1$ and each other. 
    The latter two lie on the line $L_{l(abc)}=L_1$, hence $P_1$ does. 
\end{proof}

In this proof, it is crucial that the pair $(i,j)$ in property~(0) exists and is unique: otherwise multiplying the equations over all faces but $j$ would not lead to the equation for
~$j$. 
However, the face $j$ with 
$l(j)=1$ need \emph{not} be unique; only the uniqueness of the \emph{pair} $(i,j)$ is required. 

Recall that the notion of a triangulation requires 
that the endpoints of each edge are distinct and the intersection of two distinct faces is either empty, or a single vertex, or a single edge \bluenew{\cite{Hatcher}.} Clearly, the lemma and its proof remain true without these requirements.
This generalization of triangulations is called $\Delta$-complexes or $\Delta$-triangulations \cite[Section~2.1]{Hatcher}. We avoid them in our definitions just because this notion is less well-known. Define an \emph{elementary surface-$\Delta$-tiling proof} by replacing the word ``triangulated'' with ``$\Delta$-triangulated'' in Definition~\ref{def-elementary-surface-tiling-proof}. In the examples below, we usually present elementary surface-$\Delta$-tiling proofs to minimize the number of tiles. Those are easily transformed into genuine triangulations:

\begin{remark}[Octahedral subdivision] \label{rem-triangulation}
    If an incidence theorem with some matrix has an elementary surface-$\Delta$-tiling proof, then it has an elementary surface-tiling proof.
\end{remark}

\begin{proof} 
Notice that the three vertices of each face in the surface-$\Delta$-tiling proof \bluenew{must be} distinct, otherwise, we get a contradiction to property~($-1$). \bluenew{It remains to make the intersection of distinct faces either empty or a single vertex or edge.} Let $j$ be the marked face. 
Subdivide each face $k$ having no common edges with $j$ into thirteen copies of $k$ as shown in Figure~\ref{fig:delta-triangulation} 
and extend the maps $p$ and $l$ to the resulting vertices, edges, and faces in an obvious way. 
Subdivide each face $k\ne j$ having a common edge with $j$ analogously, only do not add new vertices on the common edge(s). We get a genuine triangulation, still satisfying 
properties~(0), ($+1$), ($-1$).
\end{proof}

\subsection{Master Theorem}

Now we are going to state a Master Theorem producing even more incidence theorems. 

We can achieve much more with tiling proofs than just with elementary ones. What can help are proofs by contradiction, using the incidence axiom, 
case distinctions, and auxiliary constructions. In particular, this allows us to finish the tiling proof of Example~\ref{ex-pappus}. 

Let us formalize those notions one by one. 


Proof by contradiction means proving $P_1\notin L_1\Rightarrow P_i\in L_j$ instead of $P_i\notin L_j\Rightarrow P_1\in L_1$. This is applicable when 
$M_{ij}=-1$ and realized by setting also $M_{11}:=-1$. 


\begin{definition}[Surface-tiling proof by contradiction] 
    \label{def-surface-tiling-proof-with-relabeling}
Consider an incidence theorem with an $m\times n$ matrix $M$ such that $M_{ij}=-1$ for some $i$ and $j$. 
    A \emph{surface-tiling proof by contradiction} is an 
    elementary surface-tiling proof of the incidence theorem with the matrix obtained from $M$ by setting $M_{11}:=-1$, swapping the rows $1$ and $i$, and swapping the columns $1$ and $j$.
\end{definition}

The incidence theorem with the resulting matrix is vacuous, 
which means a contradiction.
We swap those rows and columns 
because the conclusions of our incidence theorems always concern $M_{11}$ but not $M_{ij}$.

Next, some incidence theorems are too tiny for a tiling proof, like the incidence axiom with the matrix $M'$ given by~\eqref{eq-M_1}. 
We introduce the following tool to deal with them.

\begin{definition}[Proof by contradiction to the incidence axiom] 
    \label{def-proof-by-contradiction-to-the-incidence-axiom}  
    The incidence theorem with a matrix $M$ 
    \emph{contradicts the incidence axiom} if
    $M$ has a sub-matrix 
    $
\left(
\begin{smallmatrix}
    -1 & 1 & *\\
     1 & 1 & 1\\
     1 & 1 &-1 
\end{smallmatrix}
\right)
    $
    up to permutation of rows and columns, where $*$ is any element of $\{-1,0,1\}$.
\end{definition}

Our next tool is case distinction.

\begin{definition}[Surface-tiling proof with case distinction] 
    \label{def-surface-tiling-proof-with-case-distinction}
    Consider an incidence theorem with an $m\times n$ matrix $M$. 
    Let $M_1,\dots,M_{2^k}$ be all possible matrices obtained from $M$ by replacing all zero entries with $\pm 1$. The incidence theorem with a matrix $M_l$ is a \emph{tautology} if $(M_l)_{11}=1$.
    A \emph{surface-tiling proof with case distinction} is a collection of 
    surface-tiling proofs by contradiction for all incidence theorems with matrices $M_1,\dots,M_{2^k}$ that are not tautologies and do not contradict the incidence axiom.
\end{definition}    

As a dummy example, the incidence axiom has a surface-tiling proof with case distinction, because all incidence theorems with the matrices $M_1,\dots,M_{4}$ are either tautologies or contradict the incidence axiom, so that no tilings are required.

In what follows we do case distinction in a human-readable form, grouping the 
$2^k$ cases.

Our last tool is auxiliary constructions: to the sequences $P_1,\dots,P_m$ and $L_1,\dots,L_n$, one can iteratively add 
the intersection point of two lines or 
the line through two of the points. 
Another auxiliary construction is adding a point not on the given lines or a line not passing through the given points. The latter construction is possible for an infinite ground field $\mathbb{F}$. Thus, in what follows, we assume that $\mathbb{F}$ is infinite unless otherwise explicitly indicated.

\begin{definition}[Surface-tiling proof with auxiliary constructions] 
    \label{def-surface-tiling-proof-with-auxiliary-constructions}
    Let the field $\mathbb{F}$ be infinite. Consider an incidence theorem with an $m\times n$ matrix $M_0$. For $i=1,\dots,k$, an \emph{auxiliary construction} $M_i$ is a matrix obtained from $M_{i-1}$ by appending a row or a column having either
    \begin{itemize}
        \item
        two entries $1$ and all the other entries $0$; or
        \item
        all the entries $-1$.
    \end{itemize}
    A \emph{surface-tiling proof with auxiliary constructions}, or just a \emph{surface-tiling proof}, is a finite sequence of auxiliary constructions $M_1, \dots, M_k$ and a surface-tiling proof with case distinctions for the incidence theorem with the matrix~$M_k$.  
\end{definition}

Notice that appending the rows as in Definition~\ref{def-surface-tiling-proof-with-auxiliary-constructions} is the only allowed operation; one is not allowed to fill in the entries of $M$ using incidence theorems or any kind of logical implications. A tiling proof is very different from deducing incidence theorems from axioms or each other.

The structure of Definitions~\ref{def-surface-tiling-proof-with-case-distinction}--\ref{def-surface-tiling-proof-with-auxiliary-constructions} permits just one possible order of modifications of the incidence matrix: first, all auxiliary constructions, then, case distinction. In our exposition of tiling proofs, we sometimes change this order to improve readability; it is understood that all auxiliary constructions are moved to the beginning.  

In what follows, we describe auxiliary constructions in a human-readable form, without writing the sequence $M_0,\dots,M_k$ explicitly. For instance, 
\emph{drawing a line not passing through given points} means appending a column with all entries $-1$. \emph{Drawing a line through two given points} means appending a column with two entries $1$ and the other entries $0$. This is possible even when the two points coincide (recall that repeating points are allowed in the sequence). However, \emph{drawing a line passing through 
$P_1$ 
and not passing through $P_2,\dots,P_m$}
is not an allowed auxiliary construction because it becomes impossible when $P_i=P_1$ for some $i\ne 1$, which a priori can happen. Instead, one 
can do several steps: draw lines through $P_1$ and $P_i$ for all $i\ne 1$ (even if $P_i=P_1$), pick a point $P_{m+1}$ not on those lines, and draw a line through $P_1$ and $P_{m+1}$.



The following version of \cite[Corollary~8.3]{FP22} 
follows 
from the definitions and Lemma~\ref{l-elementary-lemma}.

\begin{theorem}[Master theorem; {cf.~\cite[Corollary~8.3]{FP22} and~\cite[p.~9]{Richter-Gebert-06}}] \label{th-master-theorem}
If an incidence theorem with some matrix 
has a surface-tiling proof, then it is true over any infinite field. 
\end{theorem}

Here the field is assumed to be infinite for 
auxiliary constructions; recall that Definition~\ref{def-surface-tiling-proof-with-auxiliary-constructions} was given under such an assumption. Theorem~\ref{th-master-theorem} generates a lot of incidence theorems.

This section does not pretend to exhaust all possible types of tiling proofs one can invent. 
We aimed at the minimal definition covering 
the 
key examples in \cite{FP22}. 
The results of the next sections do not rely on the particular definition; they hold for any definition such that the Master Theorem is true. We conjecture that an analog of Church's thesis applies: \emph{any ``reasonable'' 
general definition of a tiling proof leads to the same set of incidence theorems provable by tilings} (depending only on the topological space used in the tiling proof).


\subsection{Basic examples}
\label{ssec:examples}

Let us give two examples of incidence theorems with surface-tiling proofs.

Our first example is 
Pappus' theorem (see Example~\ref{ex-pappus} and Figure~\ref{fig-pappus} to the top left). We formalize and complete the proof given in Section~\ref{ssec:quick}. \newnew{This is going to be the first tiling proof of Pappus' theorem in full generality; the arguments in \cite{Baralic-etal-20,FP22,Richter-Gebert-06} could only prove it under certain ``general position'' assumptions.}
%
%
%
%
%
Here (up to slight ambiguity)
\begin{equation}\label{eq-one-line-theorem}
M=
\left(\begin{smallmatrix}
 0 & 0 & 0 & 1 & 0 & 0 & 0 & 0 & 1 \\
 1 & 1 & -1 & 0 & 1 & -1 & 0 & 0 & 0 \\
 0 & 1 & -1 & 0 & -1 & 1 & 1 & 0 & 0 \\
 0 & 1 & -1 & 0 & -1 & -1 & 0 & 1 & 1 \\
 0 & -1 & 1 & 0 & -1 & -1 & 1 & 0 & 1 \\
 0 & -1 & 1 & 0 & 1 & -1 & 0 & 1 & 0 \\
 1 & -1 & 1 & 0 & -1 & 1 & 0 & 0 & 0 \\
 0 & 0 & 0 & 1 & 1 & 0 & 1 & 0 & 0 \\
 0 & 0 & 0 & 1 & 0 & 1 & 0 & 1 & 0 \\
\end{smallmatrix}
\right).
\end{equation}
Recall that the rows and columns of $M$ are labeled with points $P_1,\dots,P_{m}$ and lines $L_1,\dots,L_{n}$ in order, respectively, where we have introduced the lines (see Figure~\ref{fig-pappus} to the top left)
$$
L_5:=P_2P_6, \quad L_6:=P_3P_7, \quad L_7:=P_3P_5, \quad L_8:=P_4P_6, \quad L_9:=P_4P_5.
$$ 
The entries $-1$ encode the assumption that $P_2,\dots,P_7$ are distinct and not contained in $L_2\cap L_3$.

%

\begin{figure}[htb]
    \centering
\includegraphics[width=0.6\textwidth]{pappus.png}
\includegraphics[width=0.35\textwidth]{pappustile.png}
\includegraphics[width=0.6\textwidth]{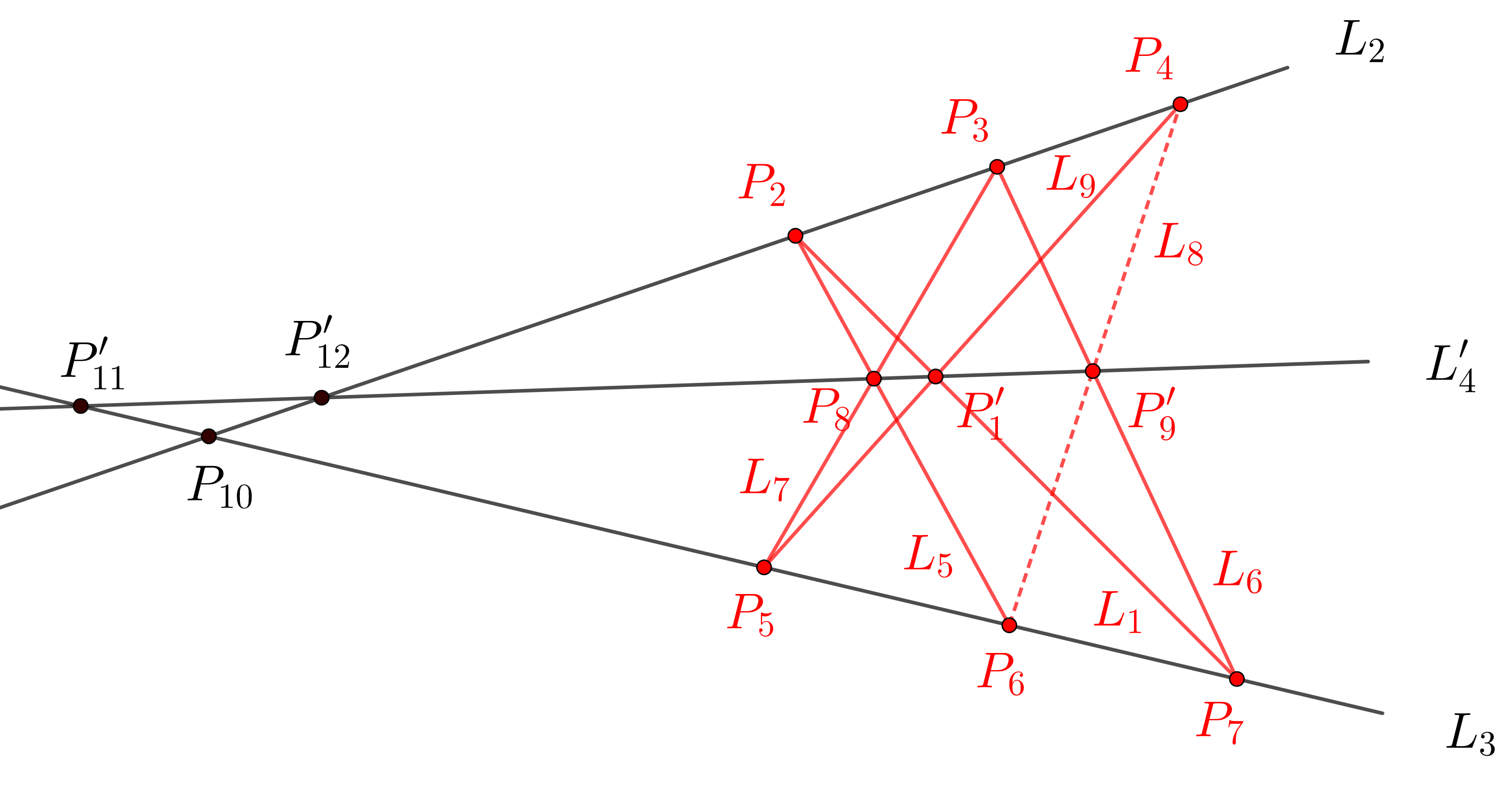}
\includegraphics[width=0.35\textwidth]{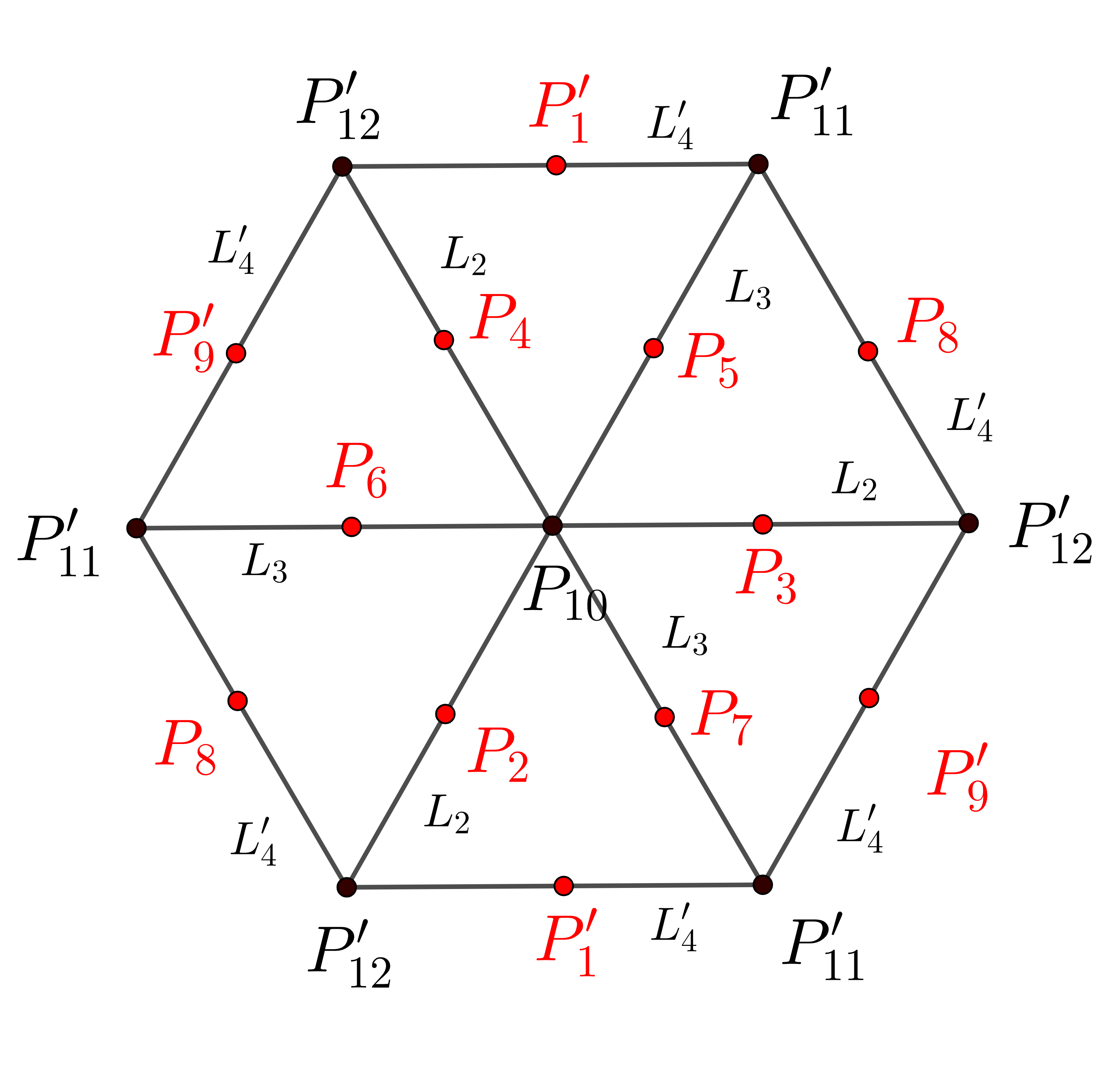}
    \caption{Pappus' configuration and auxiliary points $P_{10},P_{11},P_{12}$ (top left) and a tiling of a torus (top right) used in Case 1 of the tiling proof. Auxiliary points $P_{1}',P_9',P_{11}',P_{12}'$ and line $L_4'$ (bottom left) and a tiling of a torus (bottom right) used in Case 2 of the proof. 
    The opposite sides of each hexagon are identified.
    See the tiling proof of Example~\ref{ex-pappus}.
    }
    \label{fig-pappus}
\end{figure}

\begin{proof}[Tiling proof of Example~\ref{ex-pappus}] 
    Construct auxiliary 
    points $P_{10}:= L_2\cap L_3$, $P_{11}:= L_3\cap L_4$, $P_{12}:= L_4\cap L_2$;
    see Figure~\ref{fig-pappus} to the top left. This means adding three bottom rows to~$M$: 
    \begin{equation}\label{eq2-one-line-theorem}
\left(\begin{smallmatrix}
 0 & 0 & 0 & 1 & 0 & 0 & 0 & 0 & 1 \\
 1 & 1 & -1 & 0 & 1 & -1 & 0 & 0 & 0 \\
 0 & 1 & -1 & 0 & -1 & 1 & 1 & 0 & 0 \\
 0 & 1 & -1 & 0 & -1 & -1 & 0 & 1 & 1 \\
 0 & -1 & 1 & 0 & -1 & -1 & 1 & 0 & 1 \\
 0 & -1 & 1 & 0 & 1 & -1 & 0 & 1 & 0 \\
 1 & -1 & 1 & 0 & -1 & 1 & 0 & 0 & 0 \\
 0 & 0 & 0 & 1 & 1 & 0 & 1 & 0 & 0 \\
 0 & 0 & 0 & 1 & 0 & 1 & 0 & 1 & 0 \\
 0 & 1 & 1 & 0 & 0 & 0 & 0 & 0 & 0 \\
 0 & 0 & 1 & 1 & 0 & 0 & 0 & 0 & 0 \\
 0 & 1 & 0 & 1 & 0 & 0 & 0 & 0 & 0 \\
\end{smallmatrix}
\right).
\end{equation}
    
    Consider the following three cases:
    
    Case 1: $P_{10}\not\in L_4$. Consider the tiling shown in Figure~\ref{fig-pappus} to the top right. The maps $p$ and $l$ are depicted using the convention after Definition~\ref{def-elementary-surface-tiling-proof}. These maps satisfy properties (0), ($+1$), ($-1$) in Definition~\ref{def-elementary-surface-tiling-proof}; otherwise, we get a contradiction to the incidence axiom. 
    
    Indeed, we can identify an essentially unique way to fill the zero entries of~\eqref{eq2-one-line-theorem} with $\pm 1$ without getting a contradiction to the incidence axiom as follows. First, we put $-1$ in the intersection of row $10$ and column $4$ because $P_{10}\not\in L_4$ in Case 1. We fill in the other zero entries one by one. We try to put $1$ in the current zero entry. If this leads to a $3\times 3$ submatrix as in Definition~\ref{def-proof-by-contradiction-to-the-incidence-axiom} then we get a contradiction to the incidence axiom; hence, we put $-1$ instead. 
    Otherwise, we keep $0$ in the entry. Repeating this process (see an automated checking in \cite[Section~1]{github}, where we pass through the entries $3$ times), we bring matrix~\eqref{eq2-one-line-theorem} to the form 
    $$
    \left(\begin{smallmatrix}
 0 & -1 & -1 & 1 & -1 & -1 & -1 & -1 & 1 \\
 1 & 1 & -1 & -1 & 1 & -1 & -1 & -1 & -1 \\
 -1 & 1 & -1 & -1 & -1 & 1 & 1 & -1 & -1 \\
 -1 & 1 & -1 & -1 & -1 & -1 & -1 & 1 & 1 \\
 -1 & -1 & 1 & -1 & -1 & -1 & 1 & -1 & 1 \\
 -1 & -1 & 1 & -1 & 1 & -1 & -1 & 1 & -1 \\
 1 & -1 & 1 & -1 & -1 & 1 & -1 & -1 & -1 \\
 -1 & -1 & -1 & 1 & 1 & -1 & 1 & -1 & -1 \\
 -1 & -1 & -1 & 1 & -1 & 1 & -1 & 1 & -1 \\
 -1 & 1 & 1 & -1 & -1 & -1 & -1 & -1 & -1 \\
 -1 & -1 & 1 & 1 & -1 & -1 & -1 & -1 & -1 \\
 -1 & 1 & -1 & 1 & -1 & -1 & -1 & -1 & -1 \\
\end{smallmatrix}
\right).
$$
The resulting matrix has the properties from Definition~\ref{def-elementary-surface-tiling-proof}: 
32 ones are in the entries prescribed by property~($+1$), and the remaining entries but one are $-1$ so that property~($-1$) is automatic.

By Remark~\ref{rem-triangulation}, the incidence theorem with the resulting matrix has an elementary surface-tiling proof. This concludes the tiling proof that $P_1\in L_1$ in Case~1. (To be precise, we should also have replaced the remaining zero entry with $\pm1$, the value $+1$ leading to a tautology, and $-1$ 
to a surface-tiling proof by contradiction with $i=j=1$; see Definitions~\ref{def-surface-tiling-proof-with-case-distinction} and~\ref{def-surface-tiling-proof-with-relabeling}.)
    
    For Cases 2--3, we 
    construct auxiliary point $P_1':=P_2P_7\cap P_4P_5$, line $L_4':=P_1'P_8$, and points $P_{9}':= P_3P_7\cap L_4'$, $P_{11}':= L_3\cap L_4'$, $P_{12}':= L_4'\cap L_2$. See Figure~\ref{fig-pappus} to the bottom left. This means appending rows and columns to $M$ as follows (where $P_{13}:=P_1'$, $L_{10}:=L_4'$, $P_{14}:=P_9'$ etc.):
    \begin{equation}\label{eq4-one-line-theorem}
\left(\begin{smallmatrix}
 0 & 0 & 0 & 1 & 0 & 0 & 0 & 0 & 1 & 0 \\
 1 & 1 & -1 & 0 & 1 & -1 & 0 & 0 & 0 & 0 \\
 0 & 1 & -1 & 0 & -1 & 1 & 1 & 0 & 0 & 0 \\
 0 & 1 & -1 & 0 & -1 & -1 & 0 & 1 & 1 & 0 \\
 0 & -1 & 1 & 0 & -1 & -1 & 1 & 0 & 1 & 0 \\
 0 & -1 & 1 & 0 & 1 & -1 & 0 & 1 & 0 & 0 \\
 1 & -1 & 1 & 0 & -1 & 1 & 0 & 0 & 0 & 0 \\
 0 & 0 & 0 & 1 & 1 & 0 & 1 & 0 & 0 & 1 \\
 0 & 0 & 0 & 1 & 0 & 1 & 0 & 1 & 0 & 0 \\
 0 & 1 & 1 & 0 & 0 & 0 & 0 & 0 & 0 & 0 \\
 0 & 0 & 1 & 1 & 0 & 0 & 0 & 0 & 0 & 0 \\
 0 & 1 & 0 & 1 & 0 & 0 & 0 & 0 & 0 & 0 \\
 1 & 0 & 0 & 0 & 0 & 0 & 0 & 0 & 1 & 1 \\
 0 & 0 & 0 & 0 & 0 & 1 & 0 & 0 & 0 & 1 \\
 0 & 0 & 1 & 0 & 0 & 0 & 0 & 0 & 0 & 1 \\
 0 & 1 & 0 & 0 & 0 & 0 & 0 & 0 & 0 & 1 \\
\end{smallmatrix}
\right).
\end{equation}
(To be precise, we should have started with~\eqref{eq4-one-line-theorem} instead of~\eqref{eq2-one-line-theorem} even in Case~1, but this would not affect the above argument; see Definition~\ref{def-surface-tiling-proof-with-auxiliary-constructions}.)


    Case 2: $P_{10}\not\in L_4'$. This case is the same as the previous one, up to relabeling points/lines. 
    
    First, we explain the 
    argument informally and then rigorously justify it using matrices. By a similar tiling proof (see Figure~\ref{fig-pappus} to the bottom right), we get $P_{9}'\in L_8$.
    We consequently conclude  
    $$P_{9}'\in L_8\implies P_9'=P_9 \implies L_4'=L_4 \implies P_1'=P_1
    \implies P_1\in L_1.
    $$

    To be precise, we identify an essentially unique way to fill the zero entries of~\eqref{eq4-one-line-theorem} with $\pm 1$ without getting a tautology or a contradiction to the incidence axiom as follows. We put $-1$ in the entries $(1,1)$ and $(10,10)$ of matrix~\eqref{eq4-one-line-theorem}; meaning 
    $P_1\notin L_1$ (no tautology) and $P_{10}\not\in L_4'=:L_{10}$ (Case~2). 
    Analogously to Case~1 (see an automated checking in \cite[Section~2]{github}), we bring the matrix to the form 
\begin{equation*}
M'=\left(\begin{smallmatrix}
 -1 & -1 & -1 & 1 & -1 & -1 & -1 & -1 & 1 & -1 \\
 1 & 1 & -1 & -1 & 1 & -1 & -1 & -1 & -1 & -1 \\
 -1 & 1 & -1 & -1 & -1 & 1 & 1 & -1 & -1 & -1 \\
 -1 & 1 & -1 & -1 & -1 & -1 & -1 & 1 & 1 & -1 \\
 -1 & -1 & 1 & -1 & -1 & -1 & 1 & -1 & 1 & -1 \\
 -1 & -1 & 1 & -1 & 1 & -1 & -1 & 1 & -1 & -1 \\
 1 & -1 & 1 & -1 & -1 & 1 & -1 & -1 & -1 & -1 \\
 -1 & -1 & -1 & 1 & 1 & -1 & 1 & -1 & -1 & 1 \\
 -1 & -1 & -1 & 1 & -1 & 1 & -1 & 1 & -1 & -1 \\
 -1 & 1 & 1 & 0 & -1 & -1 & -1 & -1 & -1 & -1 \\
 -1 & 0 & 1 & 1 & -1 & -1 & -1 & -1 & -1 & -1 \\
 -1 & 1 & 0 & 1 & -1 & -1 & -1 & -1 & -1 & -1 \\
 1 & -1 & -1 & -1 & -1 & -1 & -1 & -1 & 1 & 1 \\
 -1 & -1 & -1 & -1 & -1 & 1 & -1 & -1 & -1 & 1 \\
 -1 & -1 & 1 & -1 & -1 & -1 & -1 & -1 & -1 & 1 \\
 -1 & 1 & -1 & -1 & -1 & -1 & -1 & -1 & -1 & 1 \\
\end{smallmatrix}
\right).
\end{equation*}
For the resulting matrix and the tiling shown in Figure~\ref{fig-pappus} to the bottom right, the properties in Definition~\ref{def-elementary-surface-tiling-proof} hold with the condition $p(i)=l(j)=1$ replaced with $p(i)=14$ and $l(j)=8$ in property~(0). Indeed, property~($+1$) holds by construction. Property~($-1$) holds automatically because $M'$ does not have new entries $1$ compared to~\eqref{eq4-one-line-theorem}, and zero entries are in the rows and columns corresponding to the points $P_{11},P_{12}$ and the line $L_4$ that do not appear in the tiling.

Since $M'_{14,8}=-1$, by Remark~\ref{rem-triangulation}, the incidence theorem with the matrix $M'$ has an elementary surface-tiling proof by contradiction. This concludes Case~2. 

   
    Case 3: $P_{10}\in L_4\cap L_4'$. This gives a tautology or a contradiction to the incidence axiom.
    
    Indeed, put $-1$ in the entry $(1,1)$  of matrix~\eqref{eq4-one-line-theorem} and $1$ in the entry $(10,10)$. Then we bring the matrix to the form (see an automated checking in \cite[Section~3]{github}) 
\begin{equation*}
\left(\begin{smallmatrix}
 -1 & -1 & -1 & 1 & -1 & -1 & -1 & -1 & 1 & -1 \\
 1 & 1 & -1 & -1 & 1 & -1 & -1 & -1 & -1 & -1 \\
 -1 & 1 & -1 & -1 & -1 & 1 & 1 & -1 & -1 & -1 \\
 -1 & 1 & -1 & -1 & -1 & -1 & -1 & 1 & 1 & -1 \\
 -1 & -1 & 1 & -1 & -1 & -1 & 1 & -1 & 1 & -1 \\
 -1 & -1 & 1 & -1 & 1 & -1 & -1 & 1 & -1 & -1 \\
 1 & -1 & 1 & -1 & -1 & 1 & -1 & -1 & -1 & -1 \\
 -1 & -1 & -1 & 1 & 1 & -1 & 1 & -1 & -1 & 1 \\
 -1 & -1 & -1 & 1 & -1 & 1 & -1 & 1 & -1 & -1 \\
 -1 & 1 & 1 & -1 & -1 & -1 & -1 & -1 & -1 & 1 \\
 -1 & -1 & 1 & 1 & -1 & -1 & -1 & -1 & -1 & -1 \\
 -1 & 1 & -1 & 1 & -1 & -1 & -1 & -1 & -1 & -1 \\
 1 & -1 & -1 & -1 & -1 & -1 & -1 & -1 & 1 & 1 \\
 -1 & -1 & -1 & -1 & -1 & 1 & -1 & -1 & -1 & 1 \\
 -1 & 0 & 1 & -1 & -1 & -1 & -1 & -1 & -1 & 1 \\
 -1 & 1 & 0 & -1 & -1 & -1 & -1 & -1 & -1 & 1 \\
\end{smallmatrix}
\right).
\end{equation*}  
Here the entry $(10,4)$ is $-1$ contradicting $P_{10}\in L_4$.  (No tilings are required in Case~3.)
\end{proof}

Сompared to the tiling proof in \cite[Example~8.5]{FP22}, our one 
has no general-position assumptions and, in particular, covers the case when the lines $L_2,L_3,L_4$ are concurrent. We have also corrected the labeling of edge midpoints by using 
$P_1=P_4P_5\cap L_4$ instead of $P_4P_5\cap P_2P_7$. 

To proceed, recall that the Master Theorem generates incidence theorems from tilings. 

Our second 
example is obtained from the simplest possible tiling: two triangles with glued boundaries; see Figure~\ref{fig:one-line-theorem} to the top right and Remark~\ref{rem-triangulation}.

\begin{example}[One-line theorem] \label{ex-one-line-theorem} (See Figure~\ref{fig:one-line-theorem} to the top left)
    Let points $P_1, P_2, P_3$ lie on the extensions of the sides $P_6P_5, P_5P_4, P_4P_6$ of a triangle $P_4P_5P_6$ but be distinct from the vertices. If $P_1, P_2, P_3$ lie on a line $L_2$ and $P_2, P_3$ lie on a line $L_1$ then $P_1$ lies on $L_1$.
\end{example}

\begin{figure}[htb]
    \centering
\includegraphics[width=0.85\textwidth]{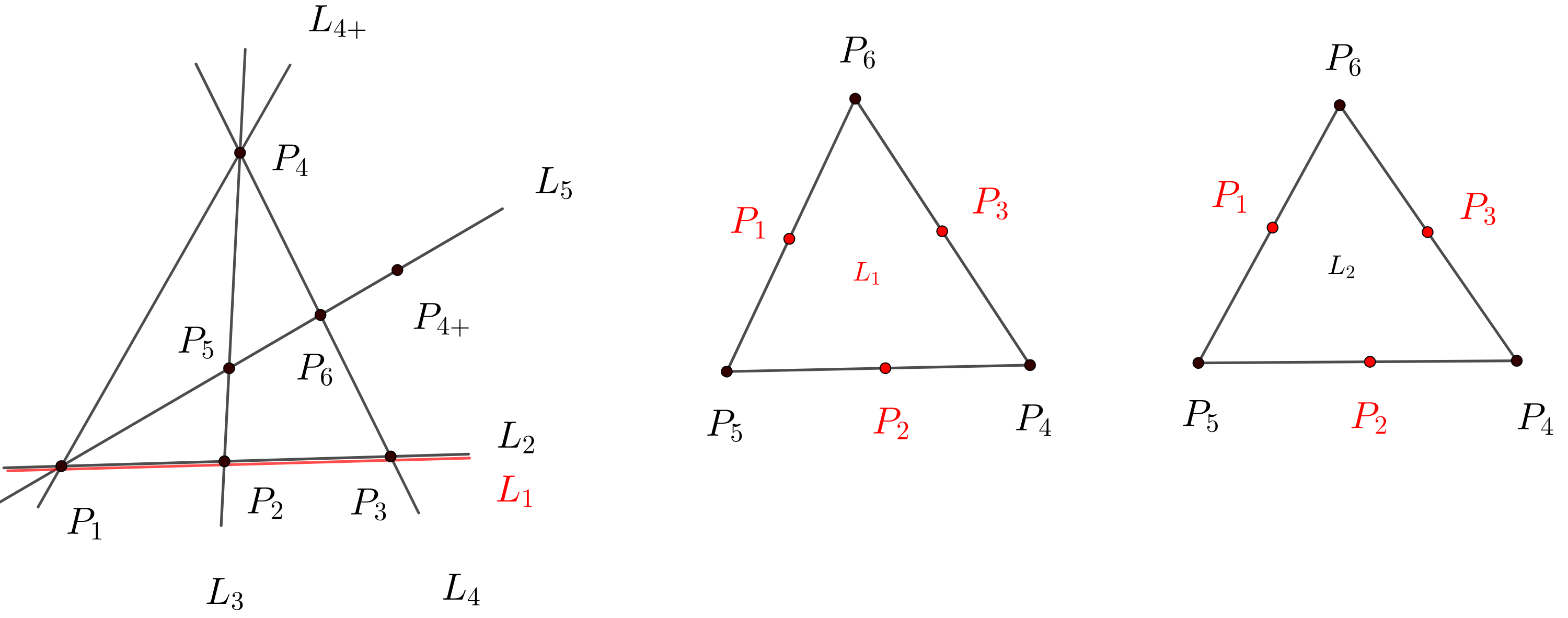}\\[-0.7cm]
\includegraphics[width=0.4\textwidth]{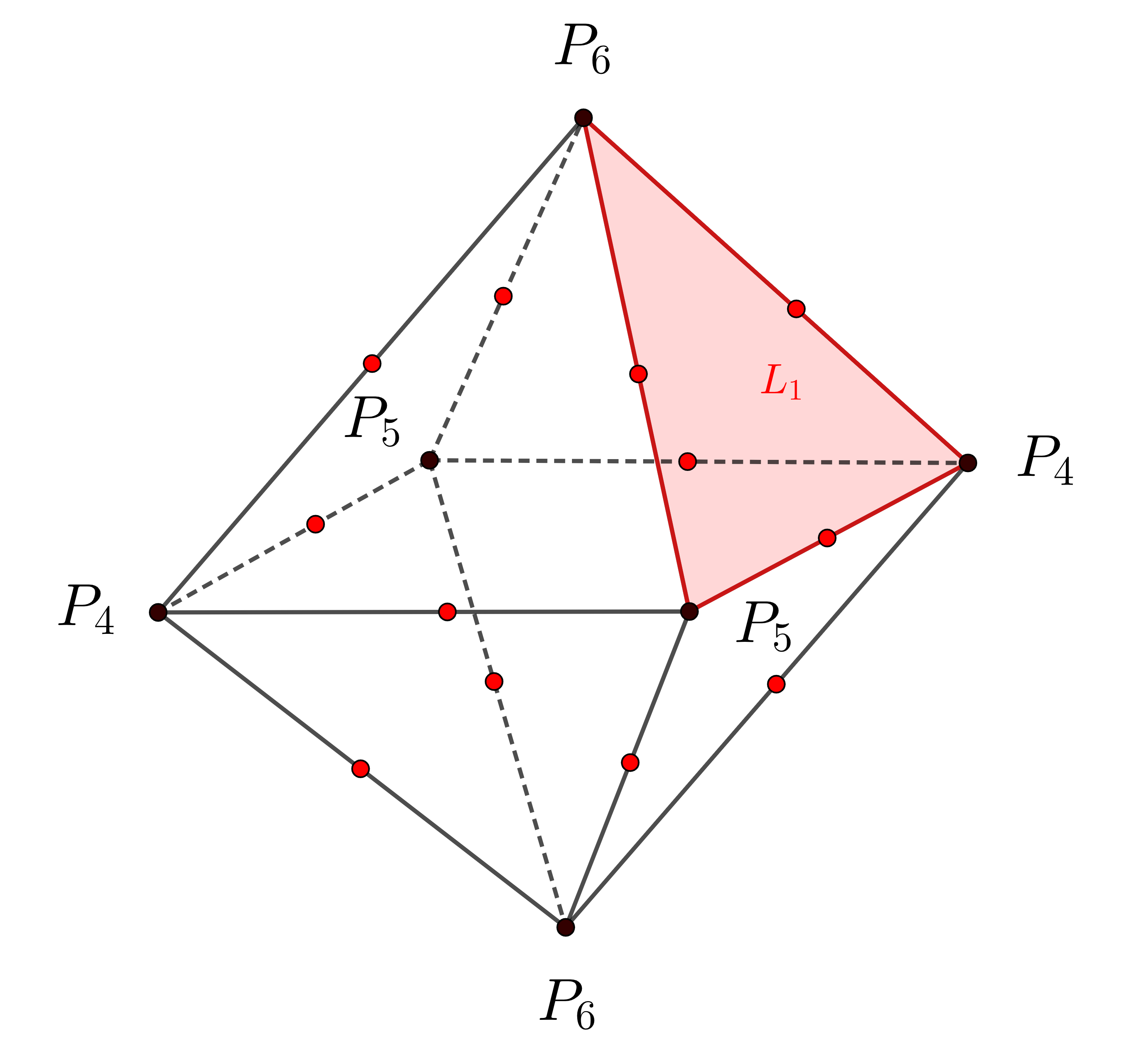}
    \caption{The one-line theorem (top left), a tiling (top right), and a triangulation of a sphere (bottom). 
    The auxiliary constructions in an attempt to prove the incidence axiom are also shown in the top left. The corresponding sides of the two triangles to the top right are identified. All faces of the octahedron other than the red one are labeled by the line $L_2$. All the edges with the endpoints labeled by $P_6$ and $P_5$, $P_5$ and $P_4$, $P_4$ and $P_6$ are labeled by $P_1$, $P_2$, $P_3$ respectively.
    See the tiling proof of Example~\ref{ex-one-line-theorem}.
    }
    \label{fig:one-line-theorem}
\end{figure}


\begin{proof}[Tiling proof] (Cf.~\cite[Example~2.12]{FP22}) See Figure~\ref{fig:one-line-theorem} to the bottom.
\end{proof}

This toy example is just a restatement of the incidence axiom. Here 
\begin{equation}\label{eq-one-line-theorem}
M=
\left(\begin{smallmatrix}
    0 &  1 & -1 & -1 &   1\\
    1 &  1 &  1 & -1 &  -1\\
    1 &  1 & -1 &  1 &  -1\\
    0 &  0 &  1 &  1 &  -1\\
    0 &  0 &  1 & -1 &   1\\
    0 &  0 & -1 &  1 &   1\\
\end{smallmatrix}
\right).
\end{equation}

\begin{remark} However, this does \emph{not} mean that one can prove the incidence axiom using this tiling. An attempt to do so results in a vicious circle, even if one assumes $P_1\ne P_2,P_3$ in addition. Indeed, if one starts with the matrix
$$
M=
\left(\begin{smallmatrix}
    0 & 1 &-1 & -1\\
    1 & 1 & 1 & -1\\
    1 & 1 &-1 &  1
\end{smallmatrix}
\right),
$$
then auxiliary constructions from Definition~\ref{def-surface-tiling-proof-with-auxiliary-constructions} lead to  
$$
M_k=
\left(\begin{smallmatrix}
    0 &  1 & -1 & -1 &  \mathbf{1} &  1 \\
    1 &  1 &  1 & -1 &  \mathbf{0} &  0 \\
    1 &  1 & -1 &  1 &  \mathbf{0} &  0 \\
    0 &  0 &  1 &  1 &  \mathbf{1} &  0 \\
    \mathbf{-1} & \mathbf{-1} & \mathbf{-1} & \mathbf{-1} &  \mathbf{-1}&  \mathbf{1} \\
    0 &  0 &  1 &  0 &  \mathbf{0} &  1 \\
    0 &  0 &  0 &  1 &  \mathbf{0} &  1 \\
\end{smallmatrix}
\right),
$$
where the highlighted column and row correspond to the line $L_{4+}:=P_1P_4$ and a point $P_{4+}\notin L_1,L_2,L_3,L_4,L_{4+}$ respectively (see Figure~\ref{fig:one-line-theorem}). 
Dropping the highlighted column and row gives 
$$
M'=
\left(\begin{smallmatrix}
    0 &  1 & -1 & -1 &  1\\
    1 &  1 &  1 & -1 &  0\\
    1 &  1 & -1 &  1 &  0\\
    0 &  0 &  1 &  1 &  0\\
    0 &  0 &  1 &  0 &  1\\
    0 &  0 &  0 &  1 &  1\\
\end{smallmatrix}
\right).
$$
The resulting matrix \bluenew{does not satisfy condition~($-1$) of Definition~\ref{def-elementary-surface-tiling-proof}: informally, it} does not have enough $(-1)$-s for an elementary surface-tiling proof compared to~\eqref{eq-one-line-theorem}. 
The missing $(-1)$-s can be obtained by case distinction and contradiction to the incidence axiom, but this means a vicious circle (using the axiom in the proof of itself).
\end{remark}

\section{Complex geometry}
\label{sec:complex}

\subsection{Untilable theorems}

We start with incidence geometry over complex numbers, having a particularly nice structure. Our first result is an incidence theorem over $\mathbb{C}$ (involving just $7$ points) that does not follow from the Master Theorem: 

\begin{example}[Fano axiom] \label{ex-Fano} (See Figure~\ref{fig-fano} to the left.)
     Let points $D\ne B$ and $E\ne B,C$ lie on 
     the \bluenew{sides} $AB$ and $BC$ of a triangle $ABC$. Take $F\in AC$ such that $BF$ passes through $CD\cap AE$. If $F\in DE$ then $D\in AC$.
\end{example}

The conclusion of the incidence theorem just means that $A=D$. 
Hereafter, \bluenew{by the \emph{sides} of a triangle $ABC$, we mean the lines $AB$, $BC$, and $CA$ (rather than the segments)}.

\begin{figure}[htbp]
    \centering
    \includegraphics[width=0.37\textwidth]{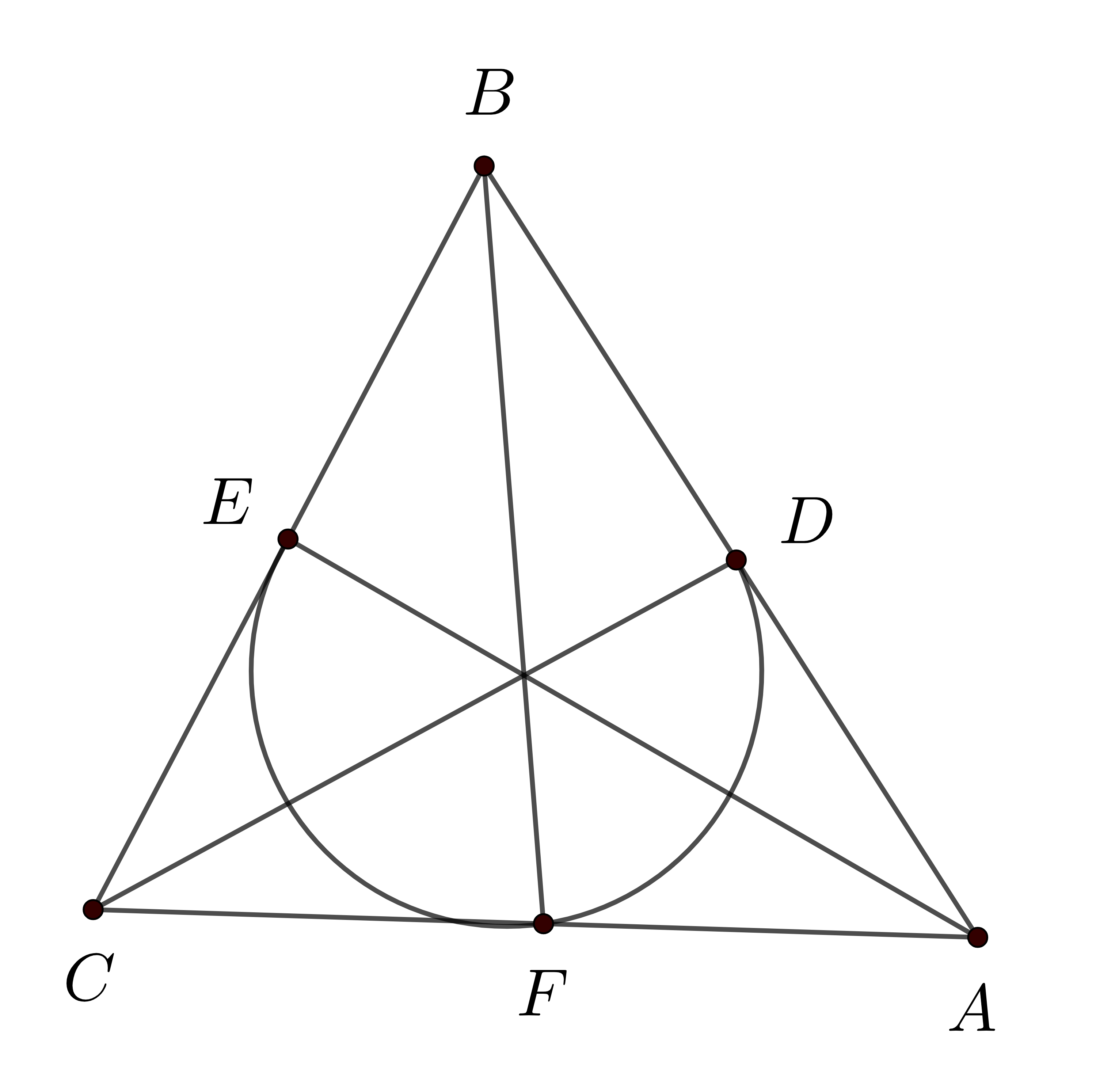}
    \includegraphics[width=0.62\textwidth]{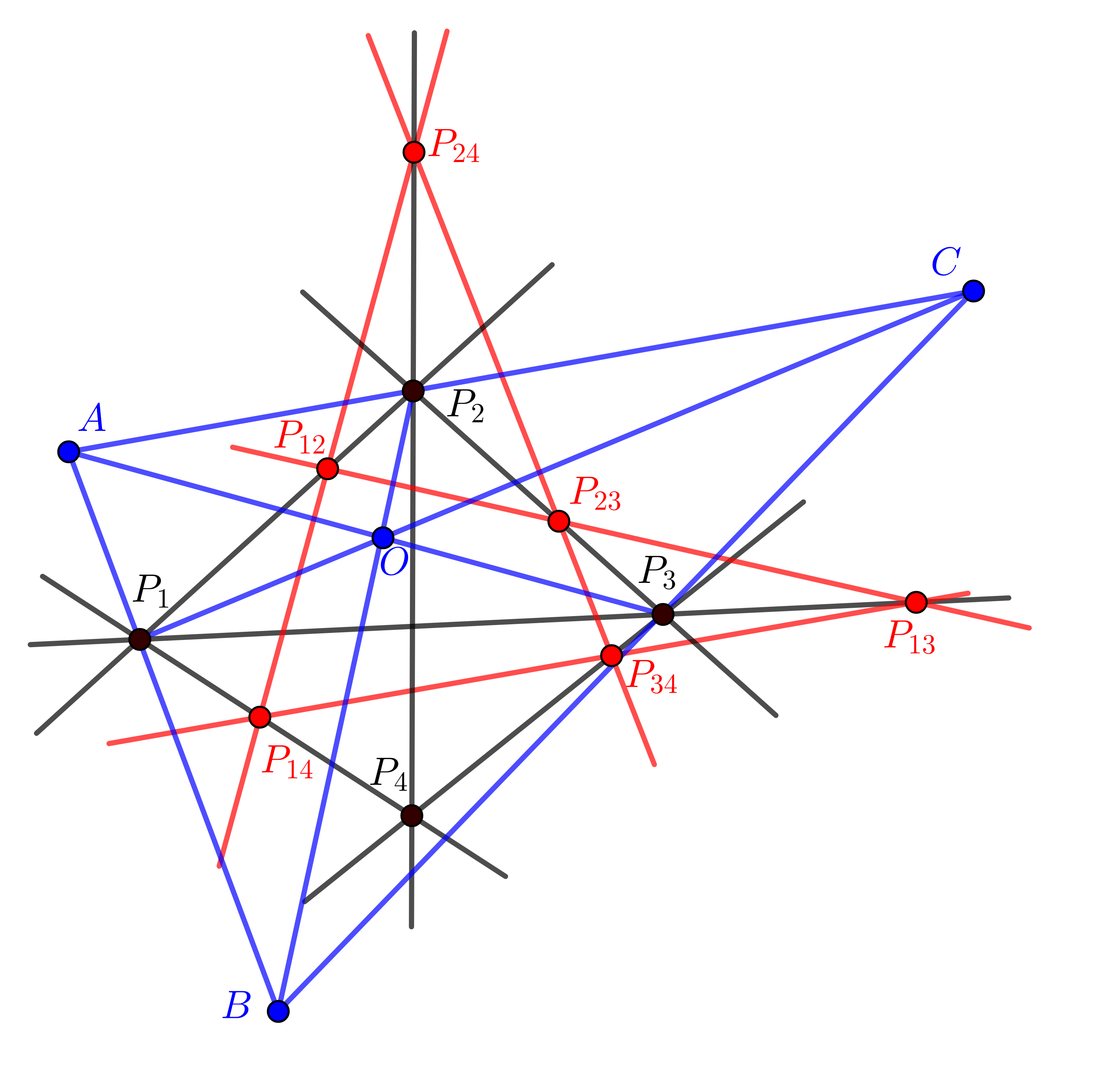}
    \caption{Fano configuration (left). It is realizable over a field of characteristic $2$. The coupled Fano axiom and Desargues' theorem (right).  
    The blue part corresponds to the Fano axiom, and the rest is Desargues' configuration from Figure~\ref{fig-desargues}. See Examples~\ref{ex-Fano} and~\ref{ex-Fano-Desargues}. 
    }
    \label{fig-fano}
\end{figure}

\begin{proposition}\label{p-Fano} The incidence theorem in Example~\ref{ex-Fano} is true over $\mathbb{C}$ 
but has no surface-tiling proof.
\end{proposition}

\begin{proof} 
This incidence theorem is true over $\mathbb{C}$, 
otherwise the cross-ratio of 
$A,C,F,AC\cap DE$ equals both $1$
(because $AC\cap DE=F$) and $-1$ (by the harmonic property of a quadrilateral).

If it had a surface-tiling proof, then it would be true over any infinite field by Theorem~\ref{th-master-theorem}. But over an infinite field of characteristic $2$, there is a counterexample: it is the \emph{Fano configuration}, that is, the projective plane over the field with $2$ elements, which is contained in the projective plane over any field of characteristic $2$. Thus, there is no surface-tiling proof.
\end{proof}

We emphasize that 
we have proved the absence of a surface-tiling proof with \emph{any} number of auxiliary constructions and case distinctions, not just an elementary surface-tiling proof. Thus, we need a counterexample over an \emph{infinite} field, not just
the field with $2$ elements (leaving no space for auxiliary constructions). 

Example~\ref{ex-Fano} is a variation of \emph{Fano axiom} used in some axiomatizations of geometry \cite{Hartshorne}. In our setup, it is an incidence theorem, not an axiom. Equivalently, the incidence theorem states that the projective plane over $\mathbb{F}_2$ 
does not embed into the projective plane over~$\mathbb{C}$. Clearly, Example~\ref{ex-Fano} remains true over any field of characteristic distinct from $2$, not just $\mathbb{C}$.

Example~\ref{ex-Fano} is 
quite degenerate. 
Traditionally, one prefers constructive incidence theorems. 
Informally, these are the ones for which the configuration can be built step by step (using the auxiliary constructions listed before Definition~\ref{def-surface-tiling-proof-with-auxiliary-constructions}) so that the last incidence is automatic. 
Formally, an incidence theorem with a matrix $M$ is \emph{constructive}, if there is a sequence of auxiliary constructions $M_0,\dots, M_k$ (see Definition~\ref{def-surface-tiling-proof-with-auxiliary-constructions}) such that 
$M_0$ has size $1\times 1$ and
$M_k$ is obtained from $M$ by swapping the first and last rows and also the first and last columns.

A subclass of constructive theorems is \emph{closure theorems}, which state that some construction always produces a periodic sequence of points. Technically, closure theorems are not incidence theorems in our sense, but can usually be restated as incidence theorems.

Let us give an example of a constructive theorem that does not follow from the Master Theorem. It is obtained from Desargues' theorem (Example~\ref{ex-desargues}) by replacing the assumption that points $P_1,P_2,P_3$ are non-collinear with the assumption that they are the feet of three cevians of a triangle. Surprisingly, this seemingly inessential modification is game-changing. 


\begin{example}[Coupled Fano axiom and Desargues theorem] \label{ex-Fano-Desargues} (See Figure~\ref{fig-fano} to the right.)
Let $A,B,C,O$ be four points in the plane such that no three of them are collinear. Let $P_1:=OA\cap BC$, $P_2:=OB\cap CA$, $P_3:=OC\cap AB$. Take a point $P_4$ not lying on the lines $P_1P_2, P_2P_3, P_3P_1$.

Pick 
non-collinear points $P_{14},P_{24},P_{34}\notin\{ P_1, P_2, P_3, P_4\}$ on 
$P_1P_4, P_2P_4, P_3P_4$ respectively. 
Take any $P_{23}\in P_2P_3\cap P_{24}P_{34}$, $P_{13}\in P_1P_3\cap P_{14}P_{34}$, $P_{12}\in P_1P_2\cap P_{13}P_{23}$. Then $P_{12}\in P_{14}P_{24}$.
\end{example}

\begin{proposition} \label{p-Fano-Desargues}
The 
theorem in Example~\ref{ex-Fano-Desargues} is true over $\mathbb{C}$ but has no surface-tiling proof.
\end{proposition}

\begin{proof}
This incidence theorem is true over $\mathbb{C}$ by the Desargues theorem (Example~\ref{ex-desargues}) because $P_1,P_2,P_3$ are non-collinear by the Fano axiom (Example~\ref{ex-Fano} and Proposition~\ref{p-Fano}).
If the theorem had a surface-tiling proof, then it would be true over an infinite field of characteristic $2$ by Theorem~\ref{th-master-theorem}. But there is a counterexample: let 
$A,B,C,O,P_1,P_2,P_3$ form a Fano configuration so that $P_1,P_2,P_3$ are collinear, let $P_{14},P_{24},P_{34}, P_{13}, P_{23}$ be as in the theorem,
and let $P_{12}$ be any point on $P_1P_2=P_{13}P_{23}$ but not 
$P_{14}P_{24}$.
Thus, there is no surface-tiling proof. 
\end{proof}

The construction in Example~\ref{ex-Fano-Desargues} is quite universal. On the one hand, a similar coupling with the Fano axiom can make almost any incidence theorem surface-tiling unprovable, just like we have done with Desargues' theorem. On the other hand, coupling with Desargues' theorem can decrease the degeneracy of our examples below, just like we have done with the Fano axiom. 


\subsection{\bluevarnew{Completeness} questions}

We see that the variety of all true incidence theorems is too large to be generated by
tilings. However, one can restrict oneself to a reasonable class of theorems to guarantee tiling proofs. A natural candidate is incidence theorems that are true over any field.

\begin{problem}
Does an incidence theorem that is true over any field have a surface-tiling proof?
\end{problem}

We conjecture that the answer is negative.
The following example is a reasonable candidate.

\begin{example}[Twin-Fano theorem] \label{ex-twin-fano} (See Figure~\ref{fig:twin-fano}.)
    Let points $D\ne B$ and $E\ne B,C$ lie on 
    the \bluenew{sides} $AB$ and $BC$ of a triangle $ABC$. Take $F\in AC$ such that $BF$ passes through $CD\cap AE$. Let $F\in DE$.

    Let points $D'\ne A',B'$ and $E'\ne B',C'$ lie on the 
    \bluenew{sides} $A'B'$ and $B'C'$ of a triangle $A'B'C'$. Take $F'\in A'C'$ such that $B'F'$ passes through $C'D'\cap A'E'$. Let $G'=D'E'\cap A'C'$.

    Let $A\ne F',G'$ and $D\ne G'$. 
    If $P\in DG'$ and $P\in AF'$ then $P\in AG'$.
\end{example}

\begin{proof} 
If the field characteristic is not $2$, then by the Fano axiom (Example~\ref{ex-Fano}) we get $A=D$.

If the field characteristic is $2$, then by the harmonic property of a quadrilateral and the equality $1=-1$ we get $F'=G'$. (Here we do not even need the construction in the first paragraph of Example~\ref{ex-twin-fano}.)

In either case, the incidence theorem follows.
\end{proof}

\begin{figure}
    \centering
\includegraphics[width=0.8\textwidth]{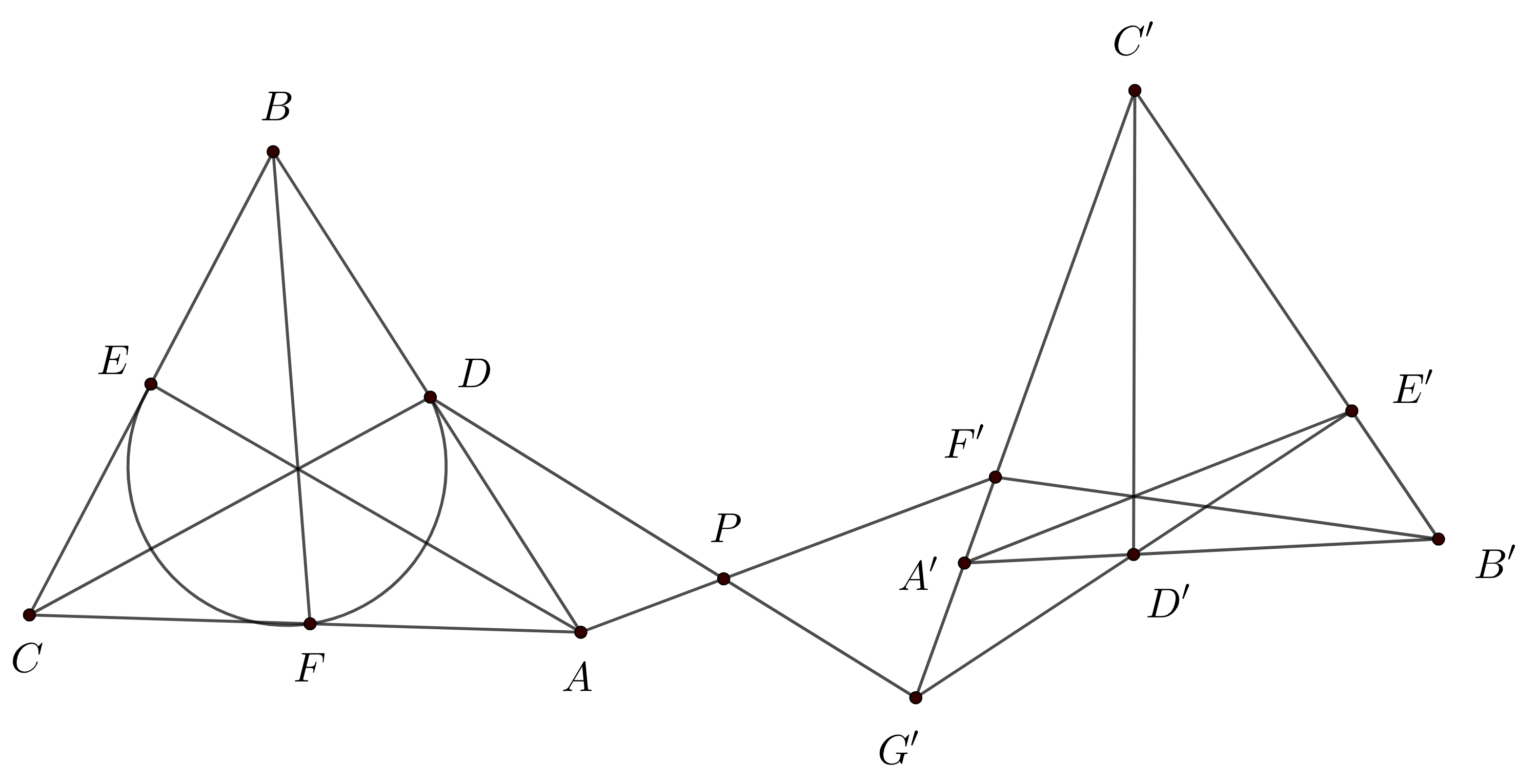}
    \caption{Twin-Fano theorem. See Example~\ref{ex-twin-fano}}
    \label{fig:twin-fano}
\end{figure}

It is intuitively clear that there is no surface-tiling proof because the configuration has two completely independent parts (formed by the points with and without primes, respectively).

\section{Real geometry}
\label{sec:real}

\subsection{Difference from complex geometry}

The incidence geometry over real numbers is essentially different from the one over complex numbers. This is already seen from the famous Sylvester--Gallai theorem; see, e.g., \cite{Green-Tao-13}. The theorem states that every finite set of points in the real projective plane has a line that passes through exactly two of the points or a line that passes through all of them. This is not true over complex numbers. A counterexample is the Hesse configuration, that is, the affine plane over the field with three elements. See Figure~\ref{fig-hesse}. It can be realized as the configuration of the nine inflection points of a smooth cubic curve on the complex projective plane, but not on the real one.

\begin{figure}[htbp]
    \centering
    \includegraphics[width=0.8\textwidth]{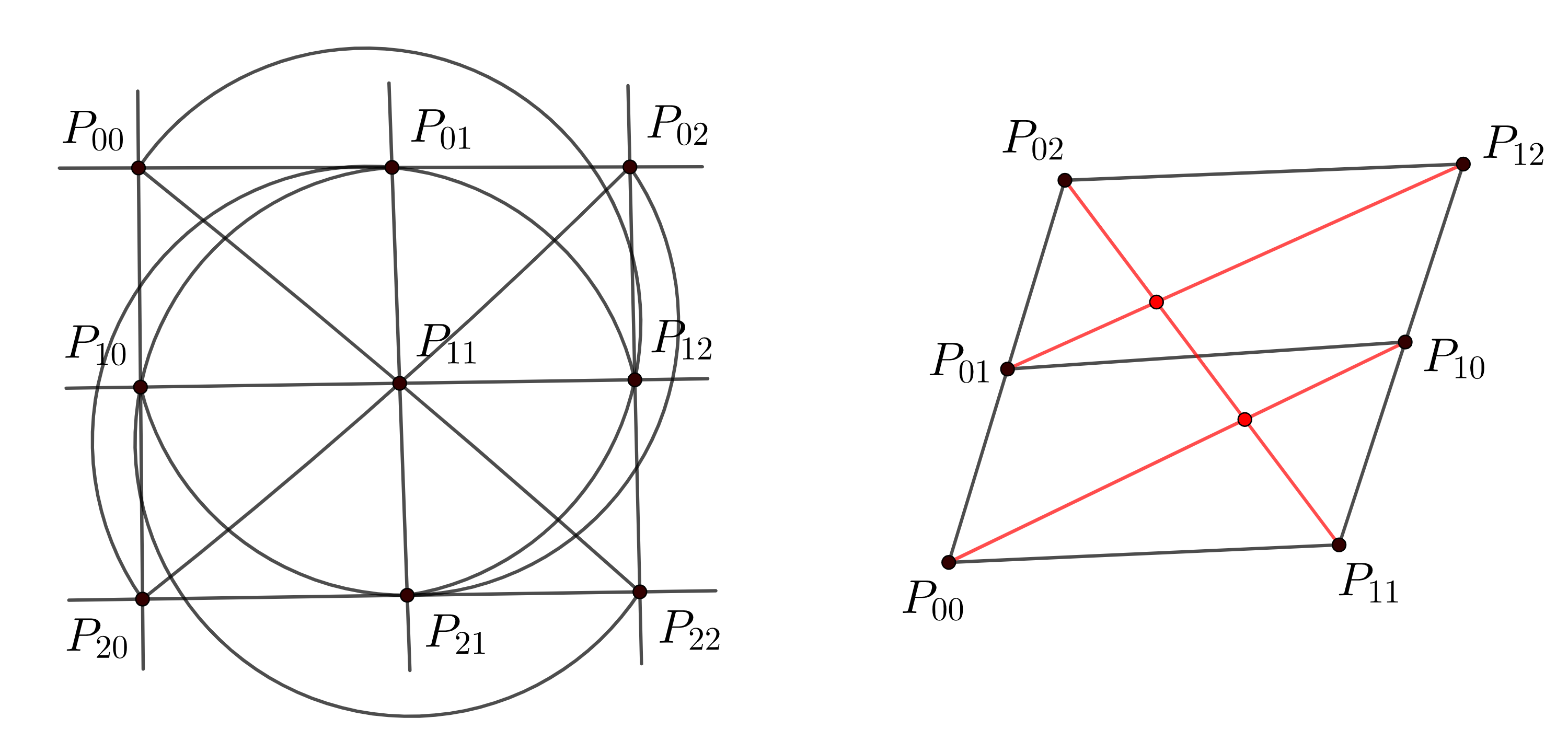}
    \caption{ Hesse configuration (left) and a projective transformation of its part (right). Since the red lines cannot be concurrent, the Hesse configuration cannot be realized in the real projective plane. See Example~\ref{ex-Hesse}.}
    \label{fig-hesse}
\end{figure}

Let us give a particular example of an incidence theorem that is true over $\mathbb{R}$ but not $\mathbb{C}$; it is the case $k=3$ of the following sequence of incidence theorems depending on a parameter $k$.

\begin{example}[$3k$-gon property] \label{ex-9-gon} (See Figure~\ref{fig-9-gon} to the left)
    Pick points $D,E,F\notin\{A,B,C\}$ on 
    the \bluenew{sides} $AB,BC,CA$ of a triangle $ABC$ and a point $O$ not on the \bluenew{sides}. 
    Starting with a point $P_1\in OA$ distinct from $O$ and $A$, construct $P_2:=OB\cap P_1D$, $P_3:=OC\cap P_2E$, $P_4:=OA\cap P_3F$, and so on. 
    If $P_1=P_{3k+1}$ then $D\in EF$. 
\end{example}

The conclusion of the theorem is equivalent to $P_1=P_4$, by the Desargues theorem. (For $k=1$, this incidence theorem is just the \bluenew{Desargues} theorem.) 

\begin{figure}[htbp]
    \centering
   \includegraphics[width=0.57\textwidth]{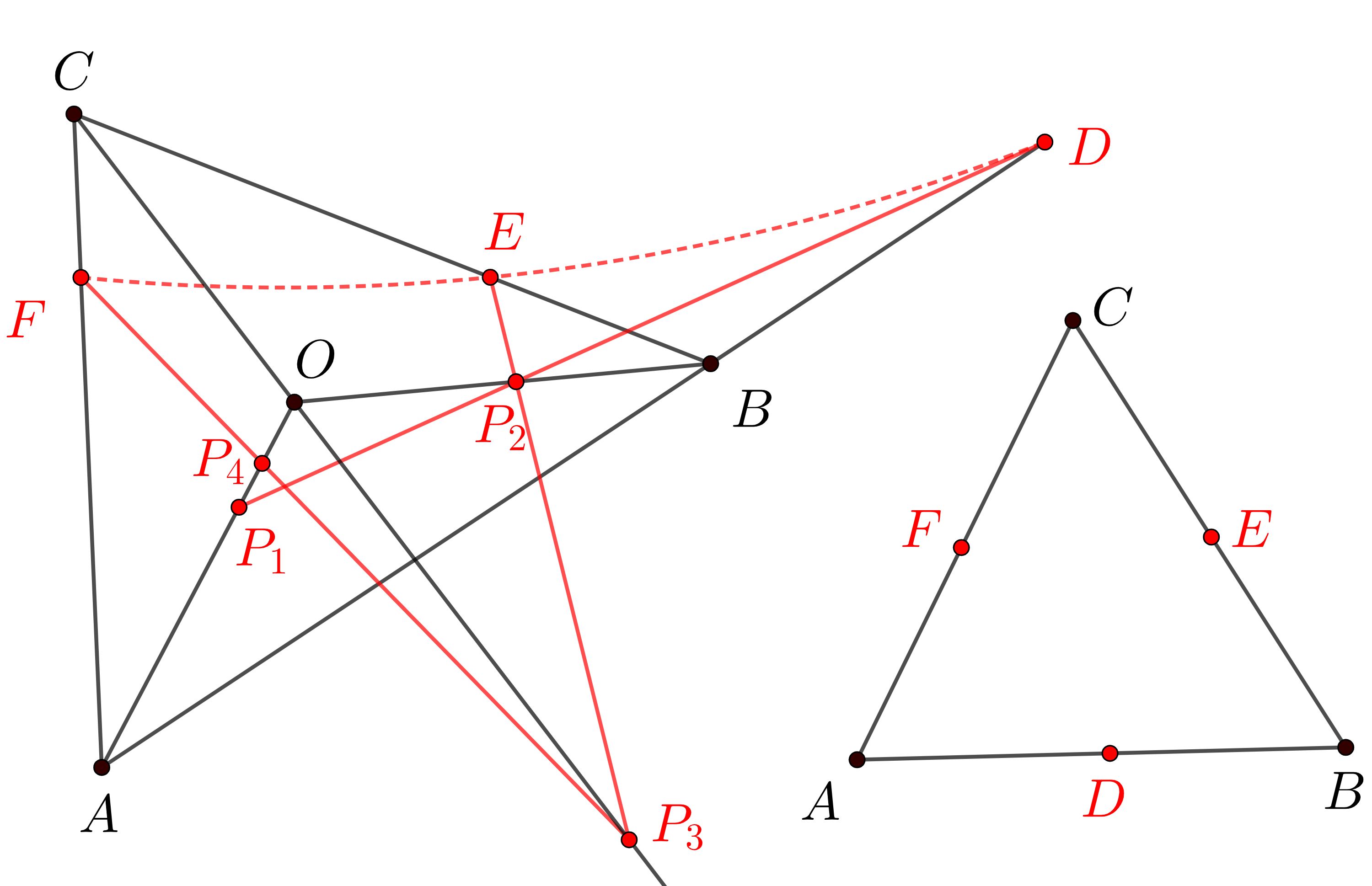}    
   \includegraphics[width=0.42\textwidth]{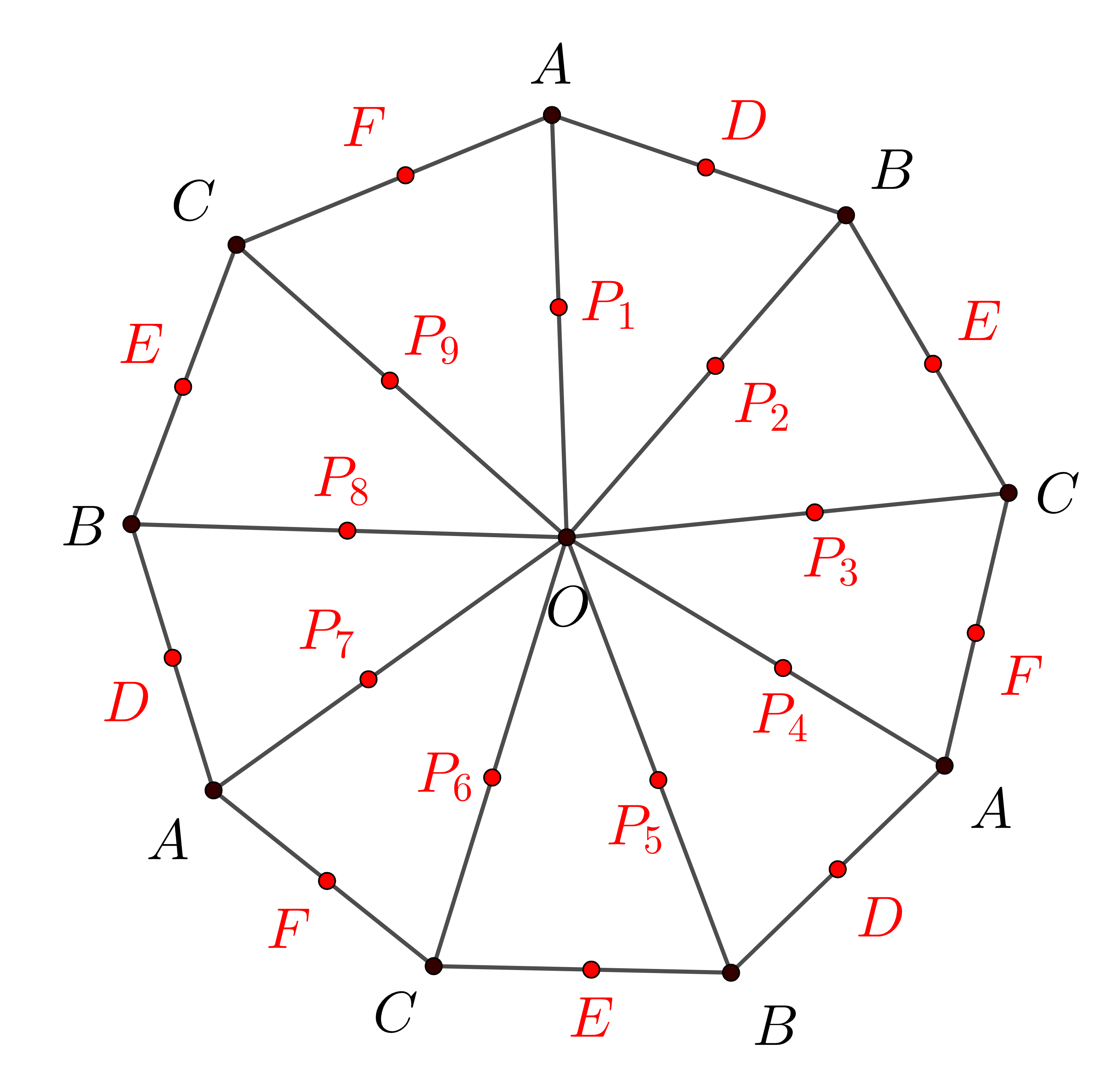}
    \caption{
    The 9-gon property (left) and a tiling of a 9-gon (right). Identifying the sides of the triangle $ABC$ and 
    of the 9-gon labeled by $AB$, $BC$, $CA$, we get a simplicial-complex proof of the 9-gon property over $\mathbb{R}$.
    See Example~\ref{ex-9-gon}, the proof of Proposition~\ref{p-9-gon}, and Definition~\ref{def-simplicial-complex-proof}.
    }
    \label{fig-9-gon}
\end{figure}

\begin{proposition}\label{p-9-gon} The incidence theorem in Example~\ref{ex-9-gon} for $k=3$ is true over $\mathbb{R}$ but false over $\mathbb{C}$ and hence has no surface-tiling proof.
\end{proposition}

\begin{proof}
Let us prove the incidence theorem over $\mathbb{R}$. The tiling in Figure~\ref{fig-9-gon} to the right explains the intuition beyond this proof. Write Menelaus's theorem for the triple of points $P_1,P_2,D$ on the 
\bluenew{sides} of the triangle $AOB$, the triple $P_2,P_3,E$ on $BOC$, etc. Multiplying the resulting $9$ equations and canceling common factors, we get 
\begin{equation}\label{eq-cube}
        \left(\left[\frac{AD}{BD}\right]\cdot
        \left[\frac{BE}{CE}\right]\cdot
        \left[\frac{CF}{AF}\right]\right)^3=1.
\end{equation}
Since the equation $x^3=1$ is equivalent to $x=1$ over $\mathbb{R}$, the cube can be removed in~\eqref{eq-cube}, and by Menelaus's theorem again, we conclude that $D,E,F$ are collinear.


Let us disprove the incidence theorem over $\mathbb{C}$. The same tiling suggests a counterexample.
Let $\tau\in\mathbb{C}$ be a primitive degree $9$ root of unity so that $\tau^9=1$ whereas $\tau^3\ne 1$. Given $A,B,C,O$ (no three being collinear), choose points $D,E,F,P_1,\dots,P_9$ so that
$$
\left[\frac{AD}{BD}\right]=\left[\frac{BE}{CE}\right]=\left[\frac{CF}{AF}\right]=\tau,
\left[\frac{OP_1}{AP_1}\right]=1, \left[\frac{OP_2}{BP_2}\right]=\tau, \left[\frac{OP_3}{CP_3}\right]=\tau^2 \text{ etc.}
$$
By Menelaus's theorem, all the assumptions of Example~\ref{ex-9-gon}
are satisfied
but the conclusion does not hold.
So, the 
theorem is false over $\mathbb{C}$, thus has no surface-tiling proof by Theorem~\ref{th-master-theorem}.
\end{proof}

\subsection{Simplicial-complex proofs}

This example suggests using topological spaces more general than surfaces in tiling proofs of incidence theorems over $\mathbb{R}$ (and other non-algebraically closed fields). The underlying property of those topological spaces 
is well-known in topology; 
see Remarks~\ref{rem-excision} and~\ref{rem-gauge} below.

\begin{definition}[Excision] \label{def-excised} Denote by $\mathbb{F}^*$ the multiplicative group of the field $\mathbb{F}$. For a finite two-dimensional simplicial complex, denote by $\vec E$ the set of its oriented edges, by 
$ab$ the edge oriented from $a$ to $b$, and by $abc$ the (non-oriented) face with the vertices $a,b,c$. The \emph{open face} is a face $abc$ without its boundary; it is denoted by $\mathrm{Int}\,abc$ or just $abc$ (if no confusion arises).

An open face $a_0b_0c_0$ 
\emph{can be excised} over $\mathbb{F}^*$ if for any function $U\colon \vec E\to \mathbb{F}^*$, the two properties
\begin{itemize}
    \item[(E)] for any oriented edge $ab$, we have $U(ab)=U(ba)^{-1}$; and
    \item[(F)] for any face $abc\ne a_0b_0c_0$, we have $U(ab)U(bc)U(ca)=1$;
\end{itemize}
imply that $U(a_0b_0)U(b_0c_0)U(c_0a_0)=1$.  
\end{definition}

\begin{remark} \label{rem-surface}
Any open face of any triangulated closed orientable surface can be excised over any commutative group. 
\end{remark}

\begin{proof}
Multiplying the equations from condition~(F) for all faces $abc\ne a_0b_0c_0$ (
with the vertices listed as prescribed by 
the surface orientation) and canceling common factors using
~(E), we get $U(a_0b_0)U(b_0c_0)U(c_0a_0)=1$. (This is 
what happened in the proof of Lemma~\ref{l-elementary-lemma}.) 
\end{proof}



Another example is obtained from the 9-gon and the triangle $ABC$ in Figure~\ref{fig-9-gon} to the right by gluing their boundaries along the obvious 3--1 map. (To obtain a genuine simplicial complex, we also need to subdivide the resulting $\Delta$-complex.) The above proof of Proposition~\ref{p-9-gon} was nothing but showing that the (open) face $ABC$ can be excised over $\mathbb{R}^*$ but not $\mathbb{C}^*$.
We observe that the possibility of excision depends on the underlying multiplicative group.


\begin{remark}\label{rem-subgroup}
If a face can be excised over a group then it can be excised over any its subgroup.
\end{remark}

Now we define an elementary simplicial-complex proof over $\mathbb{F}$ analogously to an elementary surface-tiling proof (see Definition~\ref{def-elementary-surface-tiling-proof}), only the 
surface is replaced with an arbitrary simplicial complex such that the 
marked face can be excised over $\mathbb{F}^*$. 

\begin{definition}[Simplicial-complex proof] \label{def-simplicial-complex-proof}
Consider an incidence theorem with an $m\times n$ matrix $M$ and 
a field $\mathbb{F}$. 
A \emph{labeled simplicial complex} is a finite two-dimensional simplicial complex equipped with two maps
    $$
    p\colon V\sqcup E\to \{1,\dots,m\}
    \qquad\text{and}\qquad
    l\colon F\sqcup E\to \{1,\dots,n\},
    $$
    where $V$, $E$, 
    $F$ are the sets of vertices, edges, 
    faces respectively.
An \emph{elementary simplicial-complex proof over $\mathbb{F}$} of the theorem is a labeled simplicial complex
 with the properties:
    \begin{enumerate}
        \item[(0)] there is a unique pair $i\in E,j\in F$ such that $i\subset j$ and $p(i)=l(j)=1$;
        \item[(*)] in this pair, the (open) face $j$ can be excised over $\mathbb{F}^*$;
        \item[($+1$)] for any other pair $i\!\in\! E, j\!\in\! F$ or $i\!\in\! V\sqcup E,j\!\in\! E$
        such that $i\subset j$ we have $M_{p(i)l(j)}=1$;
        \item[($-1$)] for any $i\!\in\! V\sqcup E,j\!\in\! E$  contained in one face and such that $i\not\subset j$, we have $M_{p(i)l(j)}=-1$.
    \end{enumerate}

A \emph{simplicial-complex proof over $\mathbb{F}$} is then defined analogously to \emph{surface-tiling proof}.

A labeled simplicial complex is \emph{bijectively labeled} if 
$p$ and $l$ are bijections. The \emph{incidence theorem generated by a bijectively labeled simplicial complex satisfying property}~(0) is the incidence theorem with the matrix $M$ determined by properties~($+1$) and~($-1$), and having zeroes at all the other entries. The (open) face $j$ from property~(0) is called the \emph{marked face}.
\end{definition}

%
%

\begin{remark} \label{rem-bijective} In an elementary simplicial-complex proof, one can always replace the maps $p$ and $l$ with bijections; this leads to a simplicial-complex proof of a more general theorem. 
\end{remark}

We get the following analog of Lemma~\ref{l-elementary-lemma}, this time, 
a necessary and sufficient condition.



\begin{proposition}[Elementary lemma over a given field] \label{p-excision}
The incidence theorem generated by a bijectively labeled simplicial complex satisfying property~(0) is true over an infinite field if and only if the (open) marked face 
can be excised over the multiplicative group of the field.
\end{proposition}

\begin{proof} This is analogous to Proposition~\ref{p-9-gon}.
Take a bijectively labeled simplicial complex having property~(0) and an infinite field $\mathbb{F}$. Let $M$ be the matrix determined by properties~($+1$) and~($-1$), and having
zeroes at all 
the other entries. Let $i=a_0b_0$ and $j=a_0b_0c_0$ be the pair from property~(0).

First, assume that the (open) marked face $j$ can be excised over $\mathbb{F}^*$. 
Take an arbitrary sequence of points $P_1,\dots,P_m\in P^2$ and lines $L_1,\dots,L_n\in P^{2*}$ having incidence matrix $M$. 
We may assume that $P_1,\dots,P_m$ lie in the affine plane $\mathbb{F}^2$.
Analogously to the proof of Lemma~\ref{l-elementary-lemma}, by properties~($+1$) and~($-1$) and Menelaus's theorem, we get
    \begin{equation*}
        \left[\frac{P_{p(a)}P_{p(ab)}}{P_{p(b)}P_{p(ab)}}\right]\cdot
        \left[\frac{P_{p(b)}P_{p(bc)}}{P_{p(c)}P_{p(bc)}}\right]\cdot
        \left[\frac{P_{p(c)}P_{p(ca)}}{P_{p(a)}P_{p(ca)}}\right]=1
    \end{equation*}
    for all faces $abc\ne j$. Define a function $U\colon\vec E\to \mathbb{F}^*$ by the formula $U(ab):=\left[{P_{p(a)}P_{p(ab)}}/{P_{p(b)}P_{p(ab)}}\right]$. Then conditions~(E) and~(F) of Definition~\ref{def-excised} hold. Since $j$ can be excised over $\mathbb{F}^*$, we get the same equation for the face~$j$. By Menelaus's theorem, the points 
    $P_{p(a_0b_0)}$, $P_{p(b_0c_0)}$, $P_{p(c_0a_0)}$ 
    are collinear. Thus $P_1\in L_1$, and the incidence theorem with the matrix $M$ is true over $\mathbb{F}$.

Now assume that the 
marked face $j$ \emph{cannot} be excised over $\mathbb{F}^*$. Then there is a function $U\colon\vec E\to \mathbb{F}^*$ satisfying 
conditions~(E) and~(F) of Definition~\ref{def-excised} such that $U(a_0b_0)U(b_0c_0)U(c_0a_0)\ne 1$. We may assume that $U(ab)\ne 1$ for each $ab\in\vec E$, otherwise perform the transformation $U(ab)\mapsto g(a)U(ab)g(b)^{-1}$ for a suitable function $g\colon V\to \mathbb{F}^*$.

Construct a counterexample to the incidence theorem with the matrix $M$ as follows. For each vertex $a\in V$, take a point $P_{p(a)}\in\mathbb{F}^2$ so that no three of them are collinear. This is possible because $\mathbb{F}$ is infinite and $p\colon V\sqcup E\to\{1,\dots,m\}$ is a bijection. For each edge $ab\in E$, set $L_{l(ab)}=P_{p(a)}P_{p(b)}$ and take 
$P_{p(ab)}\in L_{l(ab)}$ such that $[P_{p(a)}P_{p(ab)}/P_{p(b)}P_{p(ab)}]=U(ab)$. 
For each face $abc\ne a_0b_0c_0$, points $P_{p(ab)}$, $P_{p(bc)}$, and $P_{p(ca)}$ belong to one line by Menelaus's theorem and condition~(F); let this line be $L_{l(abc)}$. Finally, 
set
$L_{1}=P_{p(b_0c_0)}P_{p(c_0a_0)}$. Then $P_1\notin L_1$ by Menelaus's theorem
and condition $U(a_0b_0)U(b_0c_0)U(c_0a_0)\ne 1$. However, the constructed sequence of points and lines has incidence matrix $M$. Thus the incidence theorem is false.
\end{proof}

Proposition~\ref{p-excision} remains true over a \emph{finite} field if $p(V)$ has at most $4$ elements because then there is enough space to choose the points  $P_{p(a)}\in P^2$, $a\in V$, so that no three of them are collinear, and the same argument works ($U(ab)=1$ is now allowed and $P_{p(ab)}$ is 
the improper point of the line $P_{p(a)}P_{p(b)}$ in this case). As a direct consequence, we get the following result.   

\begin{proposition} \label{p-3n-gon} The incidence theorem in Example~\ref{ex-9-gon} is true over a field if and only if the polynomial $x^k-1$ has a unique root $x=1$ in the field. 
In particular, in the case $k=3$, the theorem is true over the field with $2$ or $8$ elements but false over the field with $4$ elements.  
\end{proposition}

The 
latter holds because the multiplicative group of a finite field with $q$ elements is cyclic of order $q-1$, hence $x^k-1$ has a unique root $x=1$ if and only if $q-1$ and $k$ are coprime. Recall that $\mathbb{F}_4$ is not a subfield of $\mathbb{F}_8$, so that there is no contradiction to Remark~\ref{rem-extension}.

We are ready to state the Master Theorem over an arbitrary given infinite field. It 
follows directly from Remark~\ref{rem-bijective} and Proposition~\ref{p-excision}. 

\begin{theorem}[Master theorem over a given field] \label{th-master-theorem-general}
If an incidence theorem with some matrix 
has a simplicial-complex proof over an infinite field, then it is true over the field.
\end{theorem}

By Remark~\ref{rem-subgroup}, we get the following corollary.

\begin{corollary}[Passing to a subgroup] \label{cor-passing-subgroup} If an incidence theorem with some matrix 
has a simplicial-complex proof over an infinite field $\mathbb{F}$, then it is true over any 
infinite field $\Bbbk$ such that the group $\Bbbk^*$ embeds into $\mathbb{F}^*$.
\end{corollary}

The introduced Master theorem over $\mathbb{R}$ generates many more incidence theorems than the Master Theorem~\ref{th-master-theorem} known before. However, for that, we need to generate simplicial complexes with excision property (*) over~$\mathbb{R}$. This is addressed in Section~\ref{ssec:gropes}.

\begin{remark}\label{rem-excision} Let us comment on the relation of Definition~\ref{def-excised} to the common terminology in topology \cite{Hatcher}. There, given a pair $(X,A)$ of topological spaces and a group $G$, we say that a subspace $B\subset A$ \emph{can be excised}, if the inclusion $i\colon B\to A$ induces an isomorphism $i^*\colon H^n(X,A;G)\cong H^n(X-B,A-B;G) $ of (singular) relative cohomology groups for all $n$. One of the Eilenberg--Steenrod axioms of cohomology theory states that a subspace $B$ whose closure is contained in the interior of $A$ can always be excised. We are interested in the case when $A=B$, 
hence get a nontrivial restriction on the pair $(X,A)$.

Namely, in the 
case when $X$ is a 2-dimensional simplicial complex, $B=A=\mathrm{Int}\,a_0b_0c_0$ is an open face, 
$G=\mathbb{F}^*$,  
this condition reduces to $i^*\colon H^1(X, \mathrm{Int}\,a_0b_0c_0; \mathbb{F}^*)\cong H^1(X-\mathrm{Int}\,a_0b_0c_0; \mathbb{F}^*)$. Computing $H^1(X, \mathrm{Int}\,a_0b_0c_0; \mathbb{F}^*)\cong H^1(X\cup C\mathrm{Int}\,a_0b_0c_0; \mathbb{F}^*)\cong H^1(X; \mathbb{F}^*) $ using the excision property and a deformation retraction of $X\cup C\mathrm{Int}\,a_0b_0c_0$ to $X$, we further simplify this condition to $i^*\colon H^1(X; \mathbb{F}^*)\cong H^1(X-\mathrm{Int}\,a_0b_0c_0; \mathbb{F}^*) $ (where the injectivity is automatic).

This is exactly the condition in Definition~\ref{def-excised} restated in terms of simplicial cohomology. Indeed, condition~(E) means that $U\colon\vec E\to\mathbb{F}^*$ is a simplicial cochain, and condition~(F) means that its coboundary in $X-\mathrm{Int}\,a_0b_0c_0$ vanishes. Thus, $U$ is a cocycle in $X-\mathrm{Int}\,a_0b_0c_0$, which represents some cohomology class $[U]\in H^1(X-\mathrm{Int}\,a_0b_0c_0; \mathbb{F}^*)$. Consider the map $i^*\colon H^1(X; \mathbb{F}^*)\to H^1(X-\mathrm{Int}\,a_0b_0c_0; \mathbb{F}^*) $ induced by the inclusion $i\colon X-\mathrm{Int}\,a_0b_0c_0\to X$. The class $[U]$ lies in the image of $i^*$ 
if and only if $U\colon\vec E\to\mathbb{F}^*$ is also a cocycle in $X$, i.e., satisfies the additional condition $U(a_0b_0)U(b_0c_0)U(c_0a_0)=1$. This is exactly what Definition~\ref{def-excised} says.
\end{remark}




\begin{remark}
\label{rem-Delta-complex}
    All the results of this subsection remain true if one replaces a simplicial complex with a $\Delta$-complex (that is, if one does not require the intersection of distinct simplices to be a single simplex). This is proved analogously to Remark~\ref {rem-triangulation} using an octahedral subdivision.
\end{remark}


\subsection{Untilable theorems}

Let us show that Fano's axiom remains unprovable even using simplicial complexes. 

\begin{proposition}\label{p-Fano-simplicial} The incidence theorem in Example~\ref{ex-Fano} is true over both $\mathbb{C}$ and $\mathbb{R}$ but has no simplicial-complex proof over neither $\mathbb{C}$ nor $\mathbb{R}$.
\end{proposition}

\begin{proof} The incidence theorem is true over both $\mathbb{C}$ and $\mathbb{R}$ by Proposition~\ref{p-Fano} and Remark~\ref{rem-extension}.

Assume that there is a simplicial-complex proof over $\mathbb{C}$ or $\mathbb{R}$.
Consider the field $\mathbb{Z}/2\mathbb{Z}(X)$ of rational functions with the coefficients in the field $\mathbb{Z}/2\mathbb{Z}$ with two elements. Since the polynomial ring $\mathbb{Z}/2\mathbb{Z}[X]$ is a unique factorization domain, 
it follows that $\mathbb{Z}/2\mathbb{Z}(X)^*$ is a free Abelian group with countably many generators (irreducible polynomials). Thus $\mathbb{Z}/2\mathbb{Z}(X)^*$ is isomorphic to the group $\mathbb{Q}^*_{>0}$ generated by prime numbers. Therefore, $\mathbb{Z}/2\mathbb{Z}(X)^*$ embeds into both $\mathbb{C}^*$ and $\mathbb{R}^*$.

By Corollary~\ref{cor-passing-subgroup}, the incidence theorem must be true over $\mathbb{Z}/2\mathbb{Z}(X)$. However, there is a counterexample: the Fano configuration, 
which is contained in the projective plane over any field of characteristic $2$. Thus, there is no simplicial-complex proof over either $\mathbb{C}$ or $\mathbb{R}$.    
\end{proof}

We emphasize once again that 
we have proved the absence of a simplicial-complex proof with \emph{any} number of auxiliary constructions and case distinctions, not just an elementary simplicial-complex proof. Thus, we need a counterexample for an \emph{infinite} field, not just
the field with $2$ elements (leaving no space for auxiliary constructions). 
An arbitrary infinite field  $\mathbb{F}$ of characteristic $2$ will \emph{neither} do the job, because we need an embedding $\mathbb{F}^*\subset\mathbb{R}^*$. 

Let us also give an example specific to the field $\mathbb{R}$, not true over $\mathbb{C}$.

\begin{example}[Hesse configuration] \label{ex-Hesse} (See Figure~\ref{fig-hesse} to the left.)
     Let $P_{ij}$, where $0\le i,j\le 2$, be nine distinct points. Let any three points $P_{ij},P_{kl},P_{mn}$ with both $i+k+m$ and $j+l+n$ divisible by $3$ be collinear.
     Then $P_{00}\in P_{01}P_{10}$.
\end{example}

The conclusion of this incidence theorem is equivalent to the collinearity of all nine points.

\begin{proposition}\label{p-Hesse} The incidence theorem in Example~\ref{ex-Hesse} is false over $\mathbb{C}$, true over $\mathbb{R}$, but has no simplicial-complex proof over $\mathbb{R}$.
\end{proposition}

\begin{proof} 
This incidence theorem is false over $\mathbb{C}$: the inflection points of a smooth cubic curve give a counterexample.

This incidence theorem holds over $\mathbb{R}$ by the Sylvester--Gallai theorem, but
we give an elementary proof. See Figure~\ref{fig-hesse} to the right. Assume that $P_{00}\notin P_{01}P_{10}$.
We may also assume that both $P_{22}$ and $P_{00}P_{01}\cap P_{10}P_{11}$ are improper points of $P^2$. Then $P_{00}P_{01}P_{10}P_{11}$, $P_{01}P_{02}P_{12}P_{10}$, and $P_{02}P_{00}P_{11}P_{12}$ are parallelograms. 
The extensions of their diagonals $P_{00}P_{10}$, $P_{01}P_{12}$, and $P_{02}P_{11}$ pass through $P_{20}$. This is impossible in $\mathbb{R}^2$ because one of the parallelograms contains the other two, hence one of the diagonals crosses the other two twice, by topological reasons. 


Assume that there is a simplicial-complex proof over $\mathbb{R}$.
Consider the field $\mathbb{Z}/3\mathbb{Z}(X)$ of rational functions with the coefficients in the field $\mathbb{Z}/3\mathbb{Z}$. 
The group $\mathbb{Z}/3\mathbb{Z}(X)^*$ is the direct sum of the free Abelian group with countably many generators (irreducible polynomials) and the group with two elements (constants $\pm 1$). Thus $\mathbb{Z}/3\mathbb{Z}(X)^*$ is isomorphic to the group $\mathbb{Q}^*$ generated by prime numbers and the number $-1$. Therefore, $\mathbb{Z}/3\mathbb{Z}(X)^*$ embeds into $\mathbb{R}^*$.

By Corollary~\ref{cor-passing-subgroup}, the incidence theorem must be true over $\mathbb{Z}/3\mathbb{Z}(X)$. However, there is a counterexample: the \emph{Hesse configuration}, that is, the affine plane over the field with $3$ elements, which is contained in the projective plane over any field characteristic $3$. Thus, there is no simplicial-complex proof over $\mathbb{R}$.
\end{proof}



One can extract a smaller configuration from 
Example~\ref{ex-Hesse}, which still gives an untilable theorem over $\mathbb{R}$:


\begin{example}[Impossible 6-gon theorem] \label{ex-6-gon} (See Figure~\ref{fig-unreal-6-gon} to the left). 
Let $P_1,\dots, P_6$ be distinct points. Suppose that $P_1,P_3,P_5$ are collinear, $P_2,P_4,P_6$ are collinear, $P_1P_2,P_3P_4,P_5P_6$ are concurrent, and $P_2P_3,P_4P_5,P_6P_1$ are concurrent. Then $P_{1}\in P_{2}P_{3}$. 
\end{example}

Conversely, applying the Pappus theorem twice, one can see that the configuration in Example~\ref{ex-6-gon} always embeds into a Hesse configuration. The two configurations are close cousins.

\begin{figure}[htbp]
    \centering
   \includegraphics[width=0.35\textwidth]{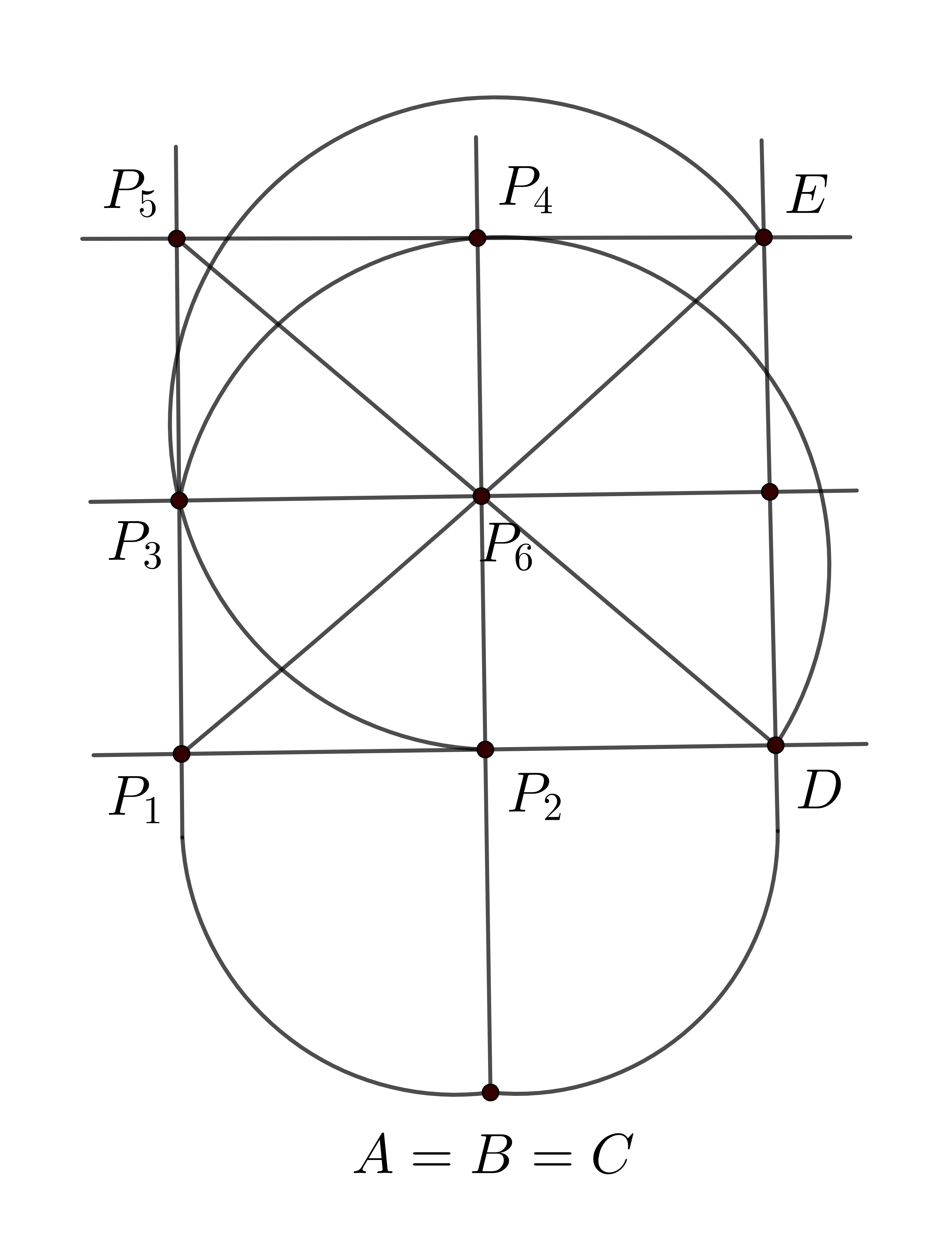}   
   \includegraphics[width=0.6\textwidth]{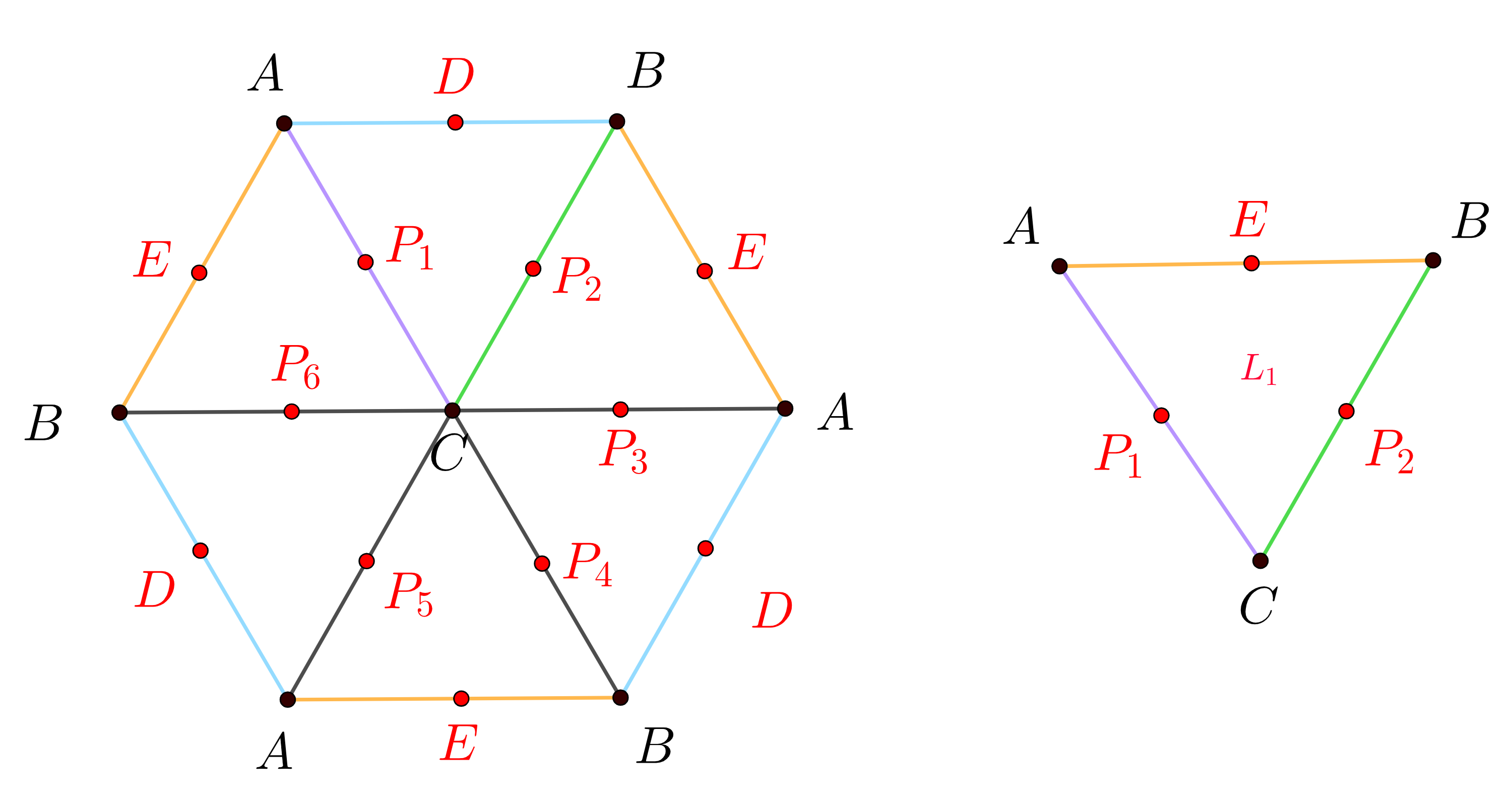}   
    \caption{
    The impossible 6-gon theorem (left) and its would-be tiling proof (right).
    Here 
    $C=P_1P_3\cap P_2P_4$, $D=P_1P_2\cap P_3P_4$, $E=P_2P_3\cap P_4P_5$, $A=P_1P_3\cap DE$, $B=P_2P_4\cap DE$ are auxiliary points. The edges of the same color (except for black) are identified. The would-be tiling proof is incorrect because of the possibility $A=B=C$, and the theorem actually has no tiling proof. See Example~\ref{ex-6-gon}.
    }
    \label{fig-unreal-6-gon}
\end{figure}

At first sight, Figure~\ref{fig-unreal-6-gon} to the right shows a tiling proof of Example~\ref{ex-6-gon} because the face labeled by $L_1$ can be excised over $\mathbb{R}^*$. However, this is not a tiling proof because the points assigned to the vertices can coincide, violating property~($-1$) in Definition~\ref{def-simplicial-complex-proof}.
Actually, Proposition~\ref{p-Hesse} and its proof remain true, if Example~\ref{ex-Hesse} is replaced with Example~\ref{ex-6-gon} (over $\mathbb{Z}/3\mathbb{Z}$, a counterexample is the 
points $(0,0)$, $(1,0)$, $(0,1)$, $(1,2)$, $(0,2)$, $(1,1)$; see Figure~\ref{fig-unreal-6-gon} to the left).
This demonstrates how accurate one should be with general-position arguments and how one can use tilings to generate untilable theorems.

The proof of Propositions~\ref{p-Fano-simplicial} and~\ref{p-Hesse} will \emph{not} work if we replace the Fano and Hesse configurations 
with the affine or projective plane over any finite field other than $\mathbb{Z}/2\mathbb{Z}$ and $\mathbb{Z}/3\mathbb{Z}$, because the multiplicative group of the field would have too much torsion to be embedded into $\mathbb{R}^*\cong\mathbb{Z}/2\mathbb{Z}\oplus\mathbb{R}$. So, 
the examples in this section are quite unique for our proof to work.

\subsection{A paradoxical example}

We proceed with 
a paradoxical example. From Proposition~\ref{p-excision} one is likely to guess that a tiling proof over $\mathbb{R}$ will work over $\mathbb{C}$ as well, once the incidence theorem is true over $\mathbb{C}$. However, this is not the case in general. A counter-example is constructed by coupling the Fano axiom and the 9-gon property (Examples~\ref{ex-Fano} and~\ref{ex-9-gon}). 

\begin{example}[Coupled Fano and 9-gon configuration] \label{ex-F4} (See Figure~\ref{fig-F4})
Let points $D\ne B$ and $E\ne B,C$ lie on 
the \bluenew{sides} $AB$ and $BC$ of a triangle $ABC$. Take $F\in AC$ such that $BF$ passes through $CD\cap AE$. Let $F\in DE$.

Pick points $D',E',F'\notin\{A',B',C'\}$ on 
the \bluenew{sides} $A'B',B'C',C'A'$ of a triangle $A'B'C'$ and a point $O'$ not on the 
\bluenew{sides}.
Starting with a point $P'_1\in O'A'$ distinct from $O'$ and $A'$, set 
$P'_2:=O'B'\cap P'_1D'$, $P'_3:=O'C'\cap P'_2E'$, $P'_4:=O'A'\cap P'_3F'$ etc. 
Let $P'_1=P'_{10}$.

If $A,D\ne E'$, $A\in D’E’$, and $D\in E’F’$, then $D’\in E’F’$.
%
%
\end{example}

\begin{proposition} \label{p-F4}
The incidence theorem in Example~\ref{ex-F4} is true over both $\mathbb{C}$ and $\mathbb{R}$, has a simplicial-complex proof over $\mathbb{R}$ but has no simplicial-complex proof over $\mathbb{C}$.
\end{proposition}


\begin{figure}[htbp]
    \centering
       \includegraphics[width=0.5\textwidth]{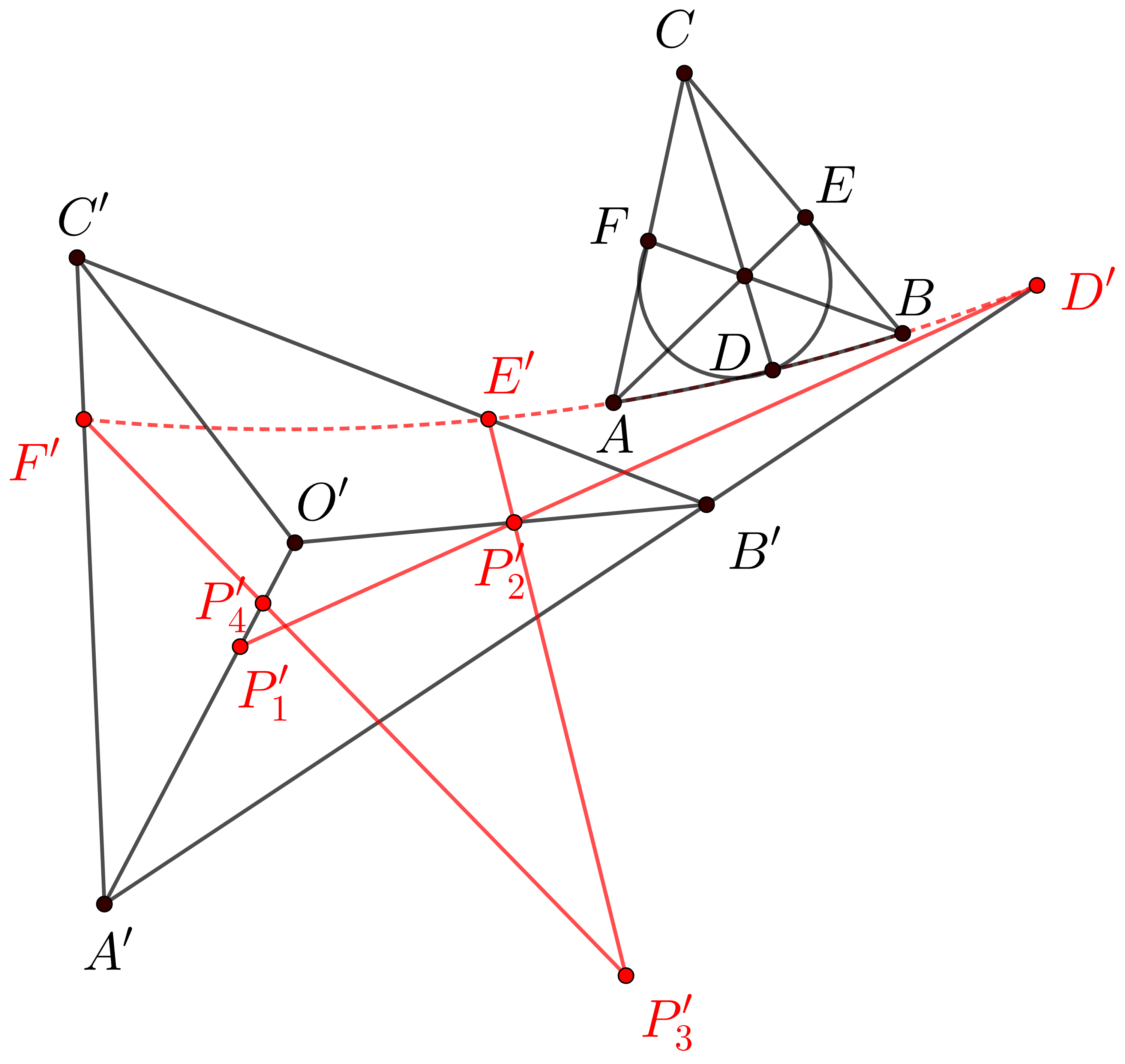} 
    \caption{
    Coupled Fano and 9-gon configuration.
    See Example~\ref{ex-F4}. 
    }
    \label{fig-F4}
\end{figure}

\begin{proof}[Proof of Proposition~\ref{p-F4}] 
    The incidence theorem is true over both $\mathbb{C}$ and $\mathbb{R}$, because 
    $A=D$ by Proposition~\ref{p-Fano}, so that the conditions $A,D’,F’\ne E'$, $A\in D’E’$, and $D\in E’F’$ imply $D’\in E’F’$.
    (Here we do not even need the construction in the second paragraph of Example~\ref{ex-F4}.)

    The incidence theorem has a simplicial-complex proof over $\mathbb{R}$ because the implication $P_1'=P_{10}'\implies D'\in E'F'$ does, by Proposition~\ref{p-3n-gon}. (Here we do not even need the constructions in the first and the third paragraph of Example~\ref{ex-F4}.)
    
    Let us prove that the incidence theorem has no simplicial-complex proof over $\mathbb{C}$. Assume the converse. Consider the field $\mathbb{F}_4(X)$ of rational functions with the coefficients in 
    $\mathbb{F}_4$. 
    Since the polynomial ring $\mathbb{F}_4[X]$ is a unique factorization domain, 
    it follows that $\mathbb{F}_4(X)^*$ is the direct sum of the free Abelian group with countably many generators (irreducible polynomials) and the group with three elements ($\mathbb{F}_4^*$). Thus $\mathbb{F}_4(X)^*$ is isomorphic to
    the group generated by prime numbers and the cubic root of unity $(i\sqrt{3}-1)/2$, 
    hence embeds into $\mathbb{C}^*$.

    By Corollary~\ref{cor-passing-subgroup}, the incidence theorem must be true over $\mathbb{F}_4(X)$, hence over $\mathbb{F}_4$ by Remark~\ref{rem-extension}. However, over $\mathbb{F}_4$, there is a counterexample: Indeed, by Proposition~\ref{p-3n-gon}, there \bluenew{is} a counterexample to the incidence theorem in Example~\ref{ex-9-gon} for $k=3$. It remains to pick points $A$ and $D$ (distinct from $E'$) on the lines $D'E'$ and $E'F'$ respectively and construct a Fano configuration with the points $A$ and $D$. Then all the assumptions of Example~\ref{ex-F4} hold, but the conclusion does not. Thus, there is no simplicial-complex proof over $\mathbb{C}$.
\end{proof}

Here, using the field with at least four elements is essential: $\mathbb{F}_2$ and $\mathbb{F}_3$ would never do the job because $\mathbb{F}_2^*,\mathbb{F}_3^*\subset \mathbb{R}^*$.



\subsection{Generalized gropes}
\label{ssec:gropes}

Now we demonstrate that the Master Theorem~\ref{th-master-theorem-general} over a given field produces many more incidence theorems than the Master Theorem~\ref{th-master-theorem} known before.
We focus on real geometry, although the same applies to any field. The underlying simplicial (and $\Delta$-) complexes \bluenew{are glued from orientable surfaces one by one, each attached along its boundary}. They bear some similarity to the gropes used in the proof of the four-dimensional generalized Poincar\'e conjecture~\cite{Behrens-etal-21, Freedman-Quinn-90}; 
see \cite{Teichner} for a concise introduction. \bluenew{For this reason, these} complexes are called generalized gropes.
The 
definition 
is a straightforward generalization of the construction in Figure~\ref{fig-9-gon} to the right.

\begin{definition}[Generalized grope] 
    \label{def-grope}
    An integer $k>1$ is \emph{torsion-coprime} over $\mathbb{F}^*$ if the equation $x^k=1$ has a unique solution $x=1$ in $\mathbb{F}^*$. (E.g., the torsion-coprime integers over $\mathbb{R}^*$ are exactly odd numbers $k>1$.)
    
    A generalized grope of complexity $q$ over $\mathbb{F}^*$ is a $\Delta$-complex defined inductively as follows. 
    
    A \emph{generalized grope of complexity $0$ over $\mathbb{F}^*$} is a $\Delta$-triangulated closed orientable surface.

    A \emph{generalized grope of complexity $q+1$ over $\mathbb{F}^*$} is obtained from a generalized grope of complexity $q$ over $\mathbb{F}^*$ by removing an open face $f$ and gluing in a $\Delta$-triangulated compact orientable surface with the boundary $\partial S$ having one component and containing $3k$ vertices, where $k>1$ is torsion-coprime over $\mathbb{F}^*$,
    by identifying the boundaries using a simplicial covering~$\partial S\to\partial f$.

    A \emph{generalized grope} over $\mathbb{F}^*$ is a $\Delta$-complex that is a generalized grope over $\mathbb{F}^*$ of complexity $q$ for some~$q$. The minimal $q$ with this property is the \emph{complexity} of the generalized grope. (We leave aside whether a number $q$ with this property is uniquely determined by the $\Delta$-complex.)
    %
\end{definition}

Examples of generalized gropes of complexity $1$ and $2$ over $\mathbb{R}^*$ are shown in Figure~\ref{fig-9-gon} to the right and Figure~\ref{fig-grope}.
\bluenew{The former is obtained by gluing the boundary of a 9-gon to the boundary of a triangle by an obvious 3--1 map.}
There is no generalized grope of positive complexity over $\mathbb{C}^*$ or the multiplicative group of any other algebraically closed field. For such fields, generalized gropes generate no new incidence theorems compared to surfaces.


\begin{figure}[htbp]
    \centering
    \includegraphics[width=0.95\textwidth]{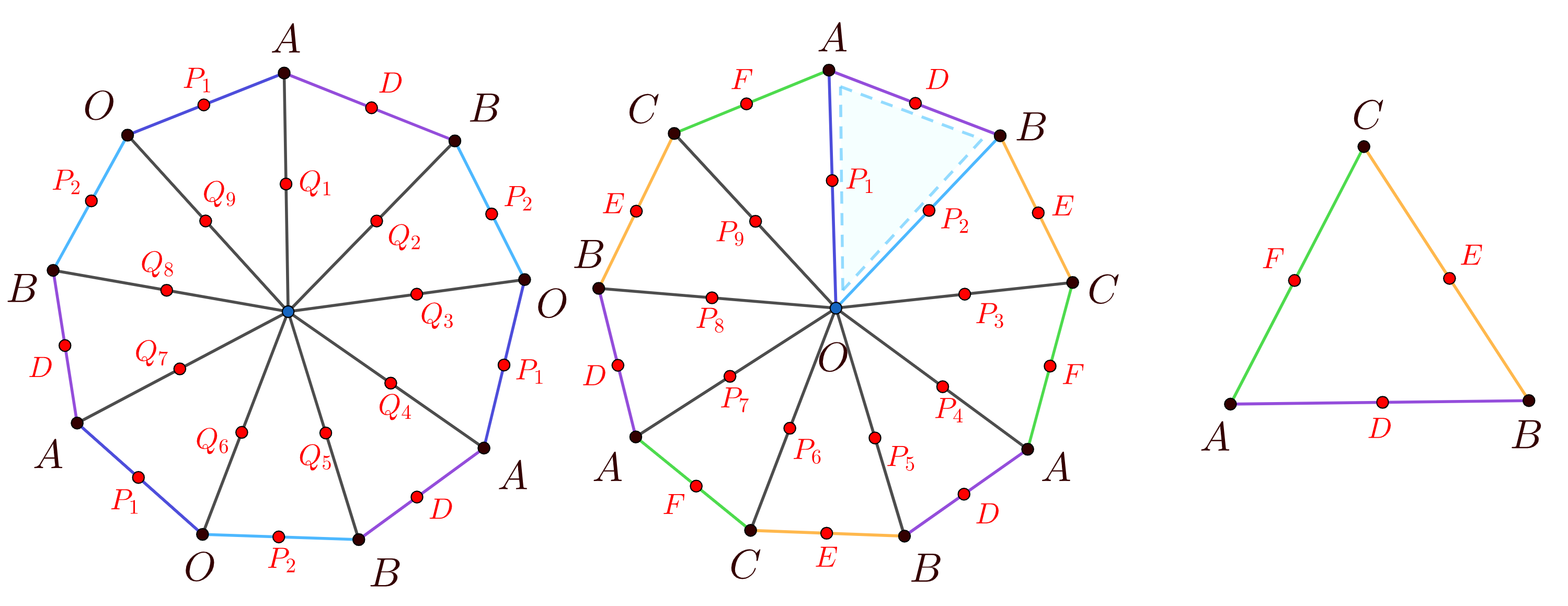}
    \caption{A (labeled) generalized grope over $\mathbb{R}^*$. The edges of the same color (except for black) are identified. The blue face is absent. The generalized grope can be constructed in two steps: First, the boundary of the 9-gon in the middle is glued to the triangle 
    to the right. Second, the interior of the blue face 
    is removed, and the boundary of the 9-gon to the left is glued to the boundary of the face. The labeled generalized grope generates an incidence theorem true over~$\mathbb{R}$. See Definition~\ref{def-grope}.}
    \label{fig-grope} 
\end{figure}

\begin{proposition}\label{p-grope}
    Any open face of a generalized grope over $\mathbb{F}^*$ can be excised over $\mathbb{F}^*$. 
\end{proposition}

\begin{corollary} \label{cor-grope}
The incidence theorem generated by any bijectively labeled generalized grope over $\mathbb{F}^*$ satisfying property~(0) is true over $\mathbb{F}$.
\end{corollary}

\begin{proof}[Proof of Proposition~\ref{p-grope}]
    The proof is by induction on the grope complexity. The base (zero complexity) is Remark~\ref{rem-surface}. To perform the induction step, assume that the proposition holds for all generalized gropes of complexity $q$ over $\mathbb{F}^*$. Take a generalized grope $G$ of complexity $q+1$ over $\mathbb{F}^*$, its face $a_0b_0c_0$, and an arbitrary function $U\colon\vec E\to\mathbb{F}^*$ satisfying conditions~(E) and~(F) of Definition~\ref{def-excised}. Let $G$ be obtained from a generalized grope $G'$ of complexity $q$ and a surface $S$ as described in Definition~\ref{def-grope}. Let $\partial S=a_1b_1c_1\dots a_kb_kc_ka_1$ where all edges $a_ib_i$ are glued to the same edge $a'b'$ of $G'$,
    and similarly for $b_ic_i$ and $c_ia_{i+1}$. Denote $U(a_ib_i):=U(a'b')$, $U(b_ic_i):=U(b'c')$, and $U(c_ia_{i+1}):=U(c'a')$. Consider two cases.

    Case 1: $a_0b_0c_0\subset G'$. Analogously to Remark~\ref{rem-surface},
    we get $U(a_1b_1)U(b_1c_1)
    \dots U(c_ka_1)=1$. Thus 
    $\left(U(a'b')U(b'c')U(c'a')\right)^k=1$. Since the number $k$ is torsion-coprime over $\mathbb{F}^*$, we get $U(a'b')U(b'c')U(c'a')=1$. Then the restriction of $U$ to the set of oriented edges of $G'$ satisfies~(E) and~(F). By the inductive hypothesis, we get $U(a_0b_0)U(b_0c_0)U(c_0a_0)=1$. 

    Case 2: $a_0b_0c_0\subset S$. By the inductive hypothesis, we get 
$U(a_1b_1)U(b_1c_1)
\dots U(c_ka_1)=\left(U(a'b')U(b'c')U(c'a')\right)^k=1^k=1$. Then $U(a_0b_0)U(b_0c_0)U(c_0a_0)=1$.

    In both cases, the open face $a_0b_0c_0$ can be excised over $\mathbb{F}^*$. 
\end{proof}



Proposition~\ref{p-grope} does \emph{not} characterize generalized gropes. 
As a dummy example, any open face of any simplicial complex can be excised over $\mathbb{F}_2^*$ (because $\mathbb{F}_2^*=\{1\}$), although not every simplicial complex is a generalized grope over $\mathbb{F}_2^*$ (because the Euler characteristic of the latter is always even).
As another example, 
any open face of a triangulated closed \emph{non-orientable} surface can be excised over $\mathbb{F}_3^*$, although the latter is not a generalized grope over $\mathbb{F}_3^*$. 
This is shown analogously to Remark~\ref{rem-surface}
because $\mathbb{F}_3^*=\{+1,-1\}$ 
thus $U(ab)^2=1$ for any edge $ab$. This idea will be used in a moment to construct a counterexample
over $\mathbb{R}^*$.

By an \emph{elementary grope-tiling proof} over $\mathbb{F}$ we mean a particular case of an elementary simplicial-complex proof over $\mathbb{F}$, when the simplicial-complex in Definition~\ref{def-simplicial-complex-proof} is a generalized grope over~$\mathbb{F}^*$. A \emph{grope-tiling proof} over $\mathbb{F}$ is then defined analogously to a surface-tiling proof.

For instance, Example~\ref{ex-9-gon} has a grope-tiling proof (see Figure~\ref{fig-9-gon} to the right).

Now an incidence theorem having a simplicial-complex 
but not a grope-tiling proof over~$\mathbb{R}^*$: 

\begin{example}[Coupled complete quadrilateral and 6-gon configuration] \label{ex-non-grope} (See Figure~\ref{fig-non-grope} to the top)
    Let $ABC$ be a triangle. Let points $D,E,F,G,H,I\notin\{A,B,C\}$ on the lines $AB,BC,CA,AB,BC,CA$ respectively satisfy $D\in HI$, $E\in IG$, and $F\in GH$. 
    
    Take a point $O\notin AB,BC,CA$. Starting with 
    $P_1\in OA$ distinct from $O$ and $A$, construct $P_2:=OB\cap P_1D$, $P_3:=OC\cap P_2E$, $P_4:=OA\cap P_3F$ etc. 
    If $P_7=P_1$, then $D\in EF$. 
\end{example}

\begin{figure}[htbp]
    \centering
    \includegraphics[width=0.9\textwidth] {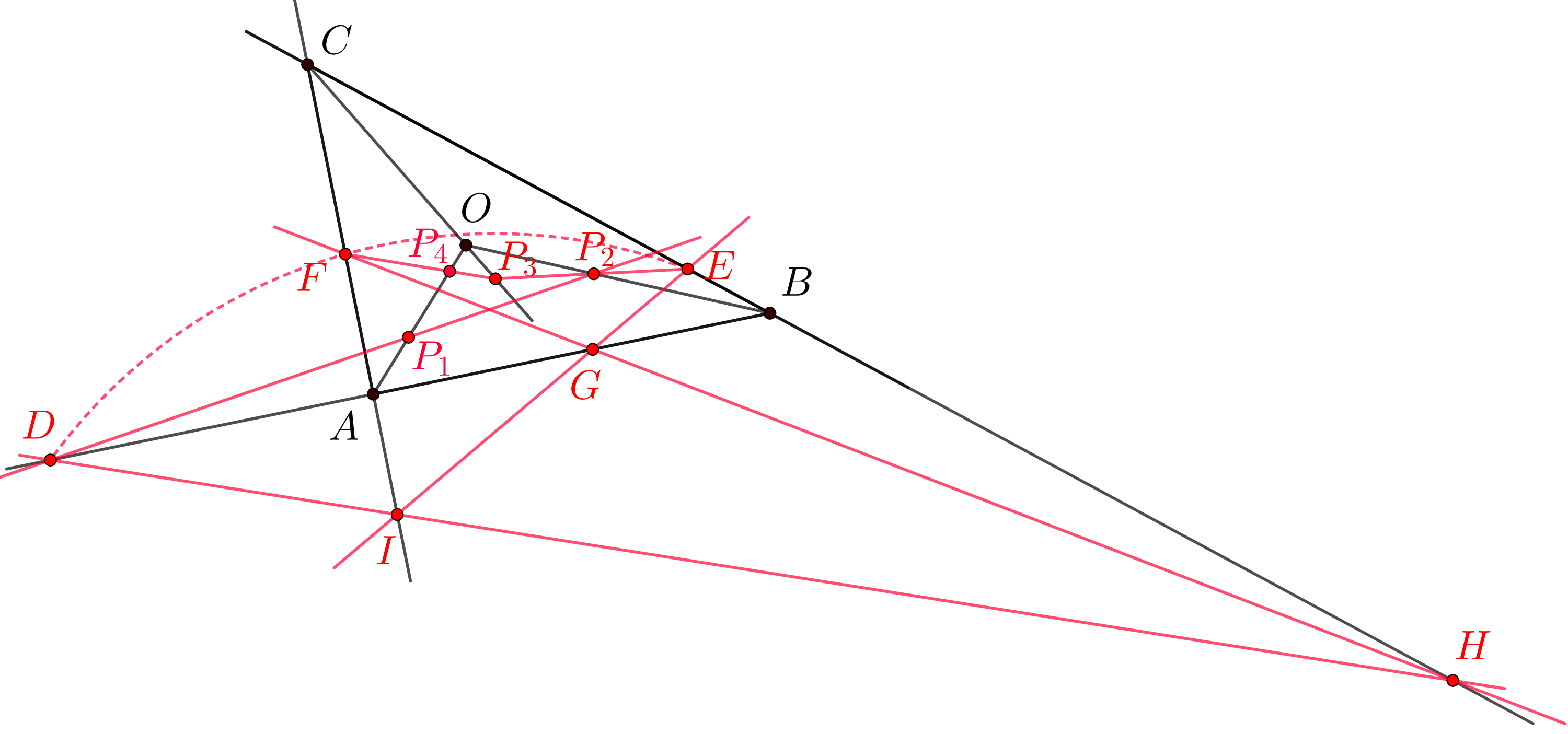} 
    \includegraphics[width=0.7\textwidth] {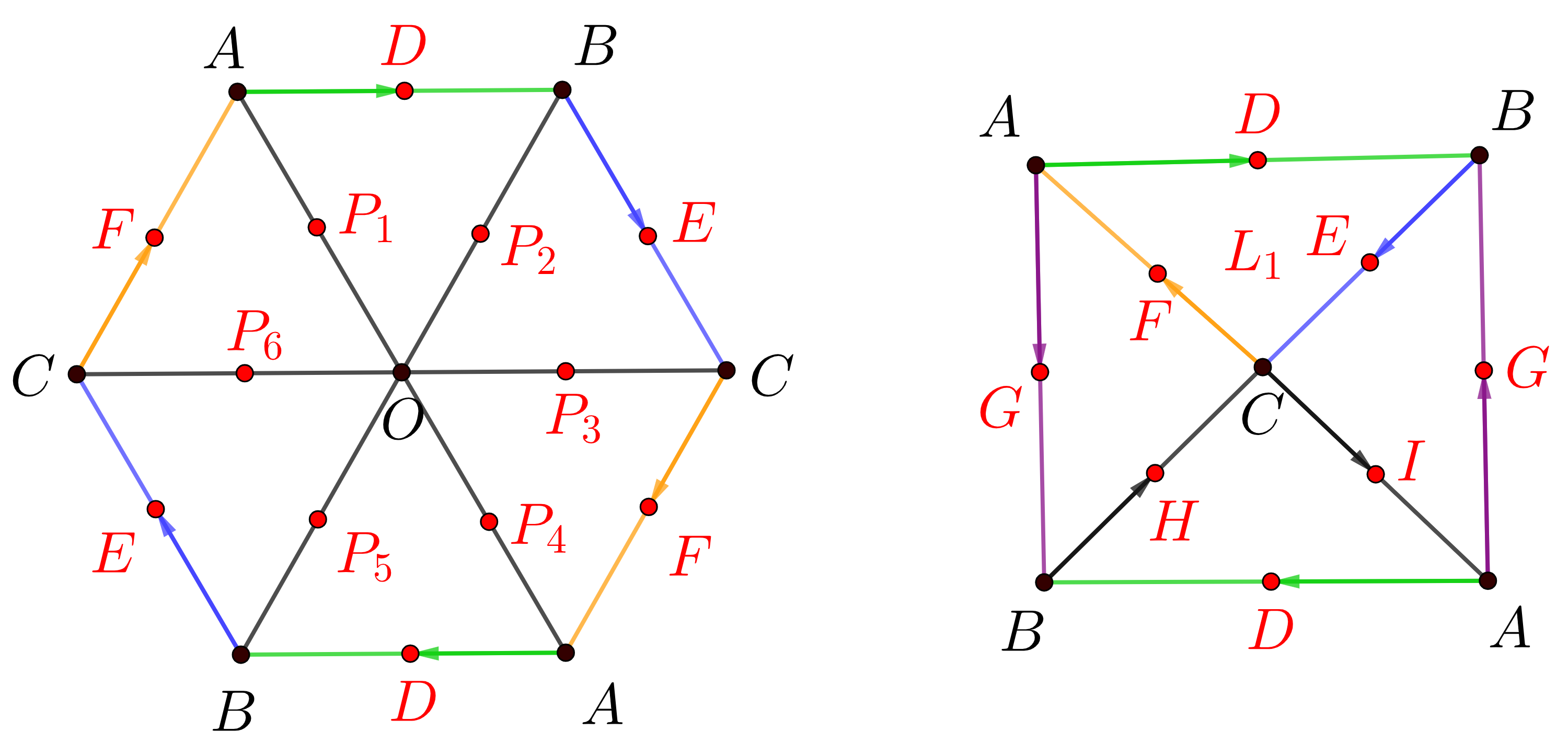} 
    \includegraphics[width=0.7\textwidth] {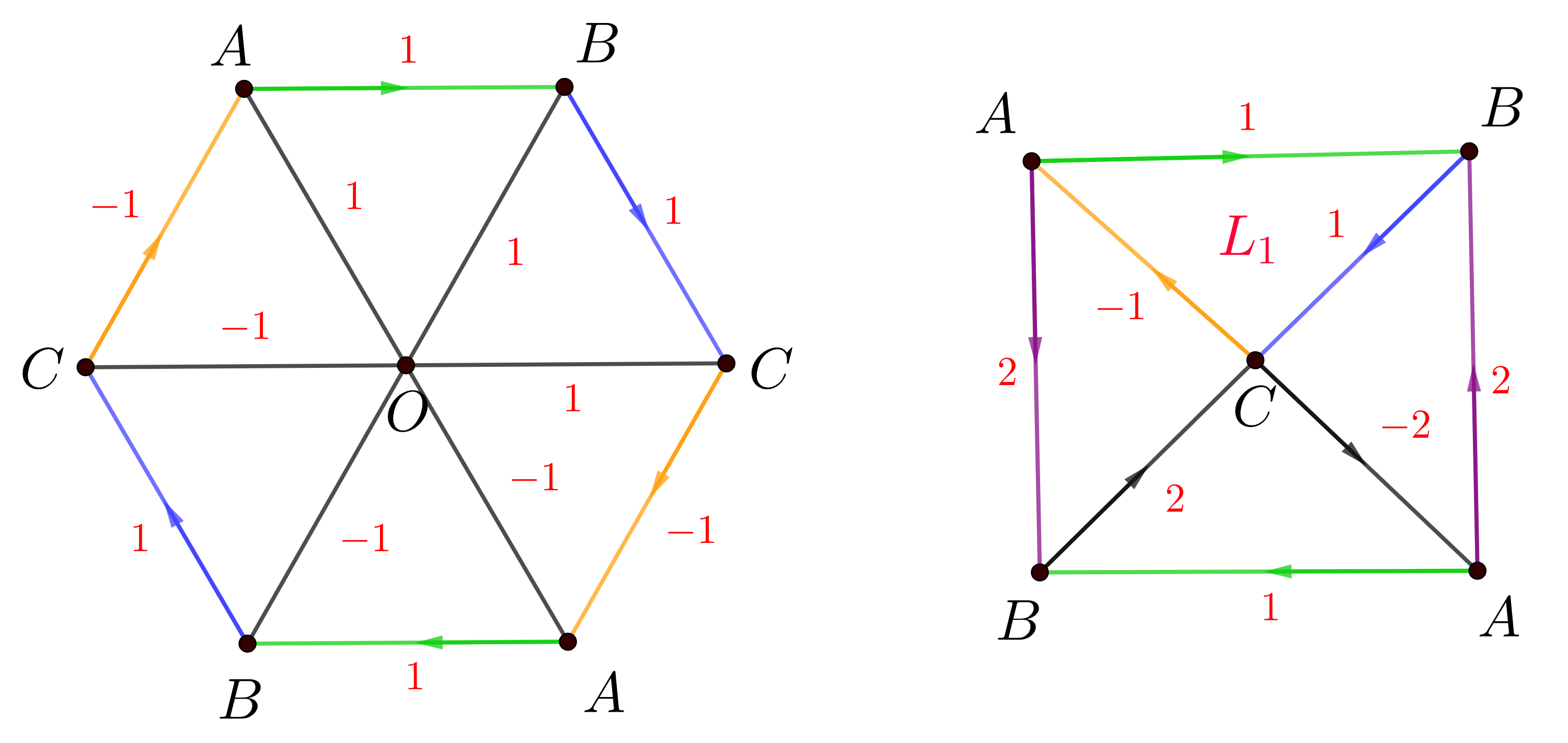} 
    \caption{
    Top: coupled complete quadrilateral and 6-gon configuration.
    Middle: a labeled $\Delta$-complex. The edges of the same color (except for black) are identified.
    Bottom: Elements of the field with $5$ elements assigned to oriented edges to disprove the incidence theorem over this field.
    See Example~\ref{ex-non-grope} and the proof of Proposition~\ref{p-non-grope}.
    }
    \label{fig-non-grope}
\end{figure}

\begin{proposition}\label{p-non-grope} The incidence theorem in Example~\ref{ex-non-grope} is true over $\mathbb{R}$, has a simplicial-complex proof over $\mathbb{R}$, but does not have a grope-tiling proof over $\mathbb{R}$.
\end{proposition}


\begin{proof}
This incidence theorem is generated by the labeled $\Delta$-complex shown
in Figure~\ref{fig-non-grope} to the middle, where the marked face is labeled by $L_1$. 
To get a simplicial-complex proof over $\mathbb{R}$, it suffices to  
check property~(*) in Definition~\ref{def-simplicial-complex-proof}
and apply Remark~\ref{rem-Delta-complex}. 

So, let us show that the marked face can be excised over $\mathbb{R}^*$.
Take a 
function $U\colon\vec E\to \mathbb{R}^*$ satisfying conditions~(E) and~(F) of Definition~\ref{def-excised}.
Denote by $U(aXb)$ its value at an oriented edge $ab$ labeled with $X$ to avoid ambiguity. Multiplying the equations from condition~(F) for six oriented faces of the hexagon $ABCABC$, we get $(U(ADB)U(BEC)U(CFA))^2=1$. Hence $|U(ADB)U(BEC)U(CFA)|=1$.  Multiplying the equations from condition~(F) for three non-marked oriented faces of the square $ABAB$, we get
$U(ADB)^{-1}U(BEC)U(CFA)U(AGB)^2=1$. Hence $U(ADB)U(BEC)U(CFA)>0$. Therefore, $U(ADB)U(BEC)U(CFA)=1$. Thus the marked face can be excised over $\mathbb{R}^*$, and the incidence theorem is true over $\mathbb{R}$ by Theorem~\ref{th-master-theorem-general}.

Assume that there is a grope-tiling proof over $\mathbb{R}$. Consider the field $\mathbb{F}_5(X)$. 
    The group $\mathbb{F}_5(X)^*$ is the direct sum of the free Abelian group with countably many generators 
    and a cyclic group $\mathbb{F}_5^*$ with four elements generated by $2$. Hence for each odd $k$ the equation $x^k=1$ has a unique solution $x=1$ in $\mathbb{F}_5(X)^*$. Thus, any odd number \bluenew{$k>1$} is torsion-coprime over $\mathbb{F}_5(X)^*$, just like in $\mathbb{R}^*$. Hence any generalized grope over $\mathbb{R}^*$ is also a generalized grope over $\mathbb{F}_5(X)^*$.
By Corollary~\ref{cor-grope} and Remark~\ref{rem-bijective}, the incidence theorem must be true over~$\mathbb{F}_5(X)$. 

This contradicts Proposition~\ref{p-excision} because the marked face in Figure~\ref{fig-non-grope} to the middle \emph{cannot} be excised over $\mathbb{F}_5(X)^*$.
Indeed, to each oriented edge $ab$ assign an element $U(ab)\in\mathbb{F}_5^*$ as shown in Figure~\ref{fig-non-grope} to the bottom (and set $U(ba):=U(ab)^{-1}$ to fit condition~(E)). Then condition~(F) holds for each 
face but 
one. Thus, there is no grope-tiling proof over $\mathbb{R}$.
\end{proof}

In this proof, the field $\mathbb{F}_5(X)$ can be replaced with $\mathbb{Q}[i]$ but not $\mathbb{F}_2(X)$, $\mathbb{F}_3(X)$, nor $\mathbb{F}_4(X)$. The reason is that the multiplicative group of the field must contain an element of order $4$ but no elements of odd order (this can be seen from a careful analysis of the argument). 


\begin{problem} \label{problem-c-excision} Prove that if a face $a_0b_0c_0$ of a 2-dimensional simplicial complex can be excised over $\mathbb{C}^*$, then there is a simplicial mapping of a triangulated closed orientable surface to the complex such that the preimage of $a_0b_0c_0$ consists of a single face. Consequently, if an incidence theorem has a simplicial-complex proof over $\mathbb{C}$, then it has a surface-tiling proof.
\end{problem}


\begin{problem} \label{problem-excision} Characterize all simplicial complexes and all their faces that can be excised over the multiplicative group of a given field.
\end{problem}

\section{Noncommutative geometry}
\label{sec:skew}

Over a \emph{skew field}, the Master Theorem remains true if the underlying surface is a sphere, but becomes false for surfaces of higher genus. 

\bluenew{Before making this precise, let us discuss this subject informally. In a skew field, the addition is commutative, but multiplication is not anymore. A typical example of a skew field are the quaternions. In a noncommutative setup, all the definitions remain almost literally the same, only we 
need to specify the multiplication order. As a result, one ends up with the notions of left or right projective lines and planes. We stick to the rule to multiply by scalars on the left and by matrices on the right. This determines all lines and planes to be left ones. }


\bluenew{More precisely, we}
fix a skew field $\mathbb{F}$, that is, a noncommutative division ring. Introduce the space $\mathbb{F}^3:=\{(x,y,z):x,y,z\in\mathbb{F}\}$.  Given $(x,y,z)\in\mathbb{F}^3\setminus\{(0,0,0)\}$, the subset $\{(tx,ty,tz):t\in\mathbb{F}\}$ is called a \emph{one-dimensional left subspace} (or 
\emph{sub-module}) of $\mathbb{F}^3$. The set of all 
such subspaces is called the \emph{(left) projective plane $P^2$ over~$\mathbb{F}$}, and each such subspace is viewed as a \emph{point} of $P^2$. Given $(a,b,c)\in\mathbb{F}^3\setminus\{(0,0,0)\}$, the subset $\{(x,y,z)\in\mathbb{F}^3:xa+yb+zc=0\}$ is called a \emph{two-dimensional left subspace} of $\mathbb{F}^3$, and also a \emph{(left) line} on $P^2$. 
The \emph{affine plane} $\mathbb{F}^2$ over $\mathbb{F}$ embeds into $P^2$ in the usual way. 
Incidence theorems and elementary surface-tiling proofs are then defined analogously to Section~\ref{sec:preliminaries}.


By an \emph{elementary sphere-tiling proof} we mean a particular case of an elementary surface-tiling proof, when the closed orientable surface in Definition~\ref{def-elementary-surface-tiling-proof} is a sphere. A \emph{sphere-tiling proof} is then defined analogously to a surface-tiling proof (notice that auxiliary constructions are well-defined because any skew field is infinite). 
Analogously, one defines a torus-tiling proof.

\begin{theorem}[Noncommutative Master Theorem] 
    \label{th-non-commutative-master-theorem}
If an incidence theorem with some matrix 
has a sphere-tiling proof, then it is true over any skew field. 
\end{theorem}

\begin{proposition}\label{p-noncommutative-Pappus} 
    Pappus' theorem 
    has a torus-tiling proof but is false over any skew field. 
\end{proposition}

\begin{corollary} Pappus' theorem (Example~\ref{ex-pappus}) has no sphere-tiling proof.
\end{corollary}
    
These results are not surprising. It is well-known that Desargues' theorem (generated by the simplest triangulation of the sphere) reflects the associativity of the ground ring, and Pappus' theorem (generated by a torus) reflects the commutativity \cite[Theorem~6.1]{Hartshorne}. Thus the former theorem is true over any skew field, whereas the latter is not.
Any triangulation of the sphere can be obtained from the simplest one by so-called bistellar moves. One can see that they correspond to applications of Desargues' theorem; cf.~\cite[Definition 9.12]{FP22}. Thus any incidence theorem generated by a triangulated sphere should be true over a skew field. However, the applications of Desargues' theorem here require additional general position arguments, which are hard to make rigorous. Thus we prefer a direct combinatorial proof based on the following well-known lemmas. Their short proofs are presented in Appendix~\ref{sec:auxiliary}.

\begin{lemma} [Noncommutative Menelaus's theorem] \textup{(See~\cite[Theorem~4.12]{RRS})}\label{l-menelaus} \bluenew{Assume that} points $A,B,C$ of the affine plane over a skew field do not lie on one 
line.
Let other points $D,E,F$ lie on the 
lines $AB,BC,CA$ respectively. Then $D,E,F$ lie on one 
line if and only if
\begin{equation}\label{eq-non-commutative-menelaus}
        \left[\frac{AD}{BD}\right]\cdot
        \left[\frac{BE}{CE}\right]\cdot
        \left[\frac{CF}{AF}\right]=1,
\end{equation}
where $[YX/ZX]$ denotes the unique element $k$ of the skew field such that $Y-X=k(Z-X)$.
\end{lemma}


\begin{lemma}[Van Kampen lemma, easy part] \textup{(See~\cite[Lemma~11.1]{Olshanskii-89})} 
\label{l-plus} 
Let $U$ be a map of the set of oriented edges of a triangulated disc with the boundary $p_1p_2\dots p_kp_1$ to a group such that 
\begin{itemize}
    \item[\textup{(E)}] for any oriented edge $ab$, we have $U(ab)=U(ba)^{-1}$; and
    \item[\textup{(F)}] for any face $abc$, we have $U(ab)U(bc)U(ca)=1$.
\end{itemize}
Then $U(p_1p_2)U(p_2p_3)\dots U(p_kp_1)=1$. 
\end{lemma}


\begin{proof}[Proof of Theorem~\ref{th-non-commutative-master-theorem}]
    It suffices to consider an elementary sphere-tiling proof. 
    Let the incidence theorem with a matrix $M$ have such a proof. 
    Take a 
    sequence of points $P_1,\dots,P_m$ 
    and lines $L_1,\dots,L_n$ 
    having incidence matrix $M$. Since a skew field is infinite by Wedderburn's little theorem, 
    we may assume that $P_1,\dots,P_m$ lie in the affine plane over the skew field. 
    
    Define a 
    function 
    on the set of oriented edges by the formula $U(ab):=[P_{p(a)}P_{p(ab)}/P_{p(b)}P_{p(ab)}]$. Clearly, $U(ba)=U(ab)^{-1}$. By properties~($+1$) and ($-1$) from Definition~\ref{def-elementary-surface-tiling-proof} and Lemma~\ref{l-menelaus}, we get $U(ab)U(bc)U(ca)=1$ for each face $abc$ distinct from the marked face $a_0b_0c_0$. 
    
    Remove the marked face $a_0b_0c_0$ from the triangulation. We get a triangulation of a disc with three boundary vertices $a_0,b_0,c_0$. By Lemma~\ref{l-plus}, we get  
    $U(a_0b_0)U(b_0c_0)U(c_0a_0)=1$. Again by Lemma~\ref{l-menelaus}, the points $P_{p(a_0b_0)}$, $P_{p(b_0c_0)}$, and $P_{p(c_0a_0)}$ lie on one line. Hence $P_1\in L_1$.
\end{proof}

\begin{proof}[Proof of Proposition~\ref{p-noncommutative-Pappus}]
    A torus-tiling proof of Example~\ref{ex-pappus} was given in Section~\ref{ssec:examples}. 
    
    It is well-known that Pappus' theorem is false over any skew field \cite[Theorem~6.1]{Hartshorne}, but let us give a short tiling disproof. 
    Take two elements $u$ and $v$ such that $uv\ne vu$. Consider the tiling in Figure~\ref{fig-pappus} to the top right. To each oriented edge $ab$ assign an element $U(ab)$ of the skew field as shown in Figure~\ref{fig-quaterionic-pappus} (and set $U(ba):=U(ab)^{-1}$ to fit condition~(E) from Lemma~\ref{l-plus}). Then condition~(F) holds for each face but one. Take three lines  $L_2$, $L_3$, and $L_4$ forming a triangle with vertices $P_{10}$, $P_{11}$, $P_{12}$. 
    Take points $P_1,\dots,P_9$ such that 
    $[P_{p(a)}P_{p(ab)}/P_{p(b)}P_{p(ab)}]=U(ab)$ for each oriented edge $ab$ (if $U(ab)=1$ then set $P_{p(ab)}$ to be the improper point of the line $P_{p(a)}P_{p(b)}$). By Lemma~\ref{l-menelaus}, we get a counterexample to Pappus' theorem.
\end{proof}

\begin{figure}[htbp]
    \centering
       \includegraphics[width=0.4\textwidth]{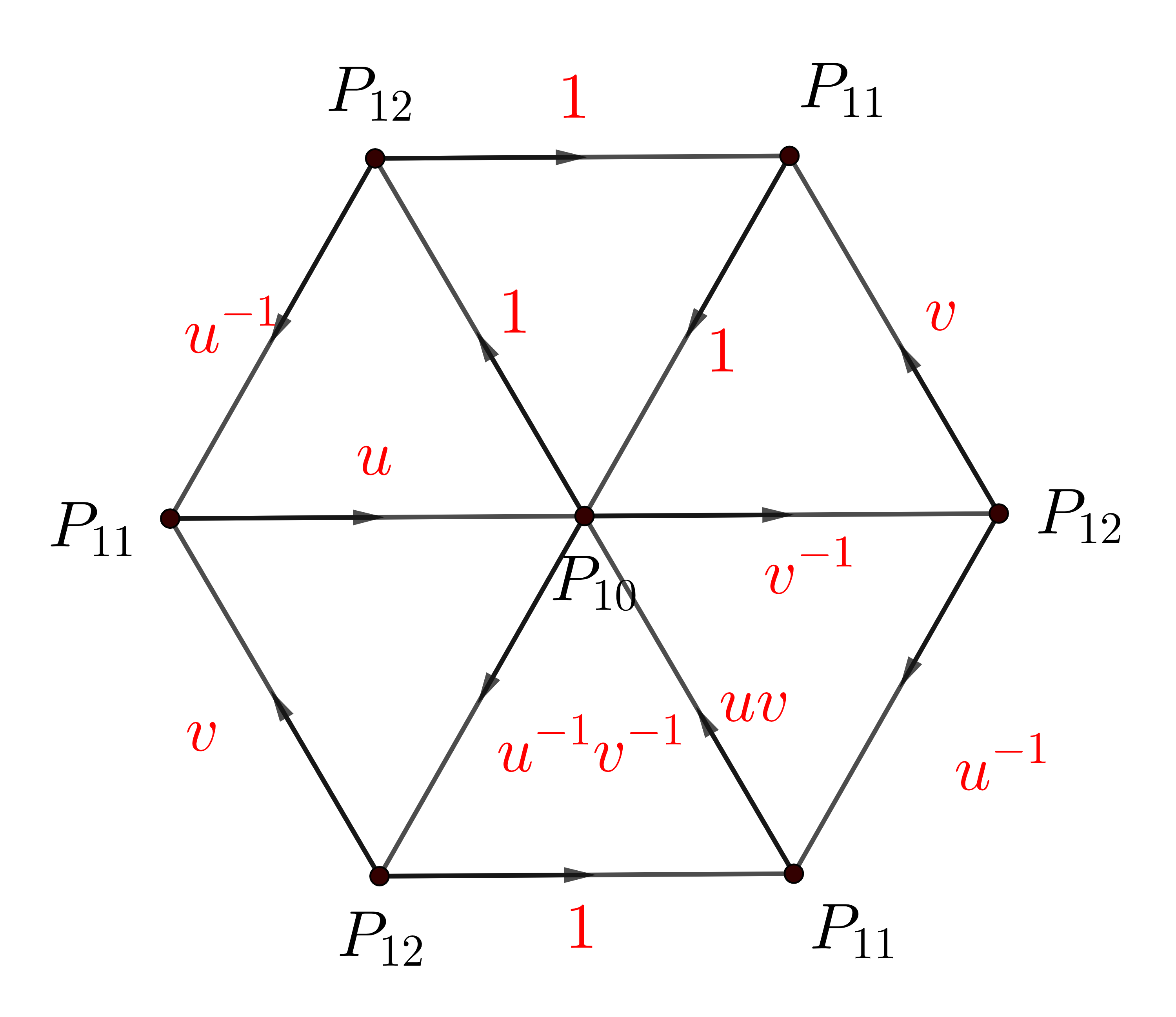}
    \caption{Tiling of a torus and the elements of the skew field assigned to the oriented edges. 
    This leads to a counterexample to Pappus' theorem if $uv\ne vu$. See the proof of Proposition~\ref{p-noncommutative-Pappus}.}
    \label{fig-quaterionic-pappus}
\end{figure}

The definition of a simplicial-complex proof (Definition~\ref{def-simplicial-complex-proof}) and the Master Theorem over a given field (Theorem~\ref{th-master-theorem-general}) remain literally the same over a skew field instead of a field. What is new is that the faces of an orientable surface of a positive genus $g$ cannot be excised anymore. (See Figure~\ref{fig-quaterionic-pappus}.) One way to restore this excision is to attach discs along $g$ disjoint simple closed curves that do not split the surface.
The gropes defined in \cite{Behrens-etal-21} will also do the job.  
It is interesting to find counter-examples analogous to the ones in Section~\ref{sec:real}.

\begin{problem}\label{pr-noncommutative}
    Do the following classes of incidence theorems coincide over a given skew field:
    \begin{itemize}
        \item the ones that are true over the skew field;
        \item the ones that have a simplicial-complex proof over the skew field;
        \item the ones that have a sphere-tiling proof?
    \end{itemize}
\end{problem}


\begin{remark}\label{rem-gauge}
Conditions~(E) and~(F) from Lemma~\ref{l-plus} arise 
in lattice gauge theory; see \cite[Section~1]{S22} for a concise elementary introduction. The value $U(ab)$ satisfying~(E) is interpreted as a \emph{parallel-transport operator} along an oriented edge $ab$, and condition~(F) means vanishing \emph{curvature} at the face $abc$. Thus Definition~\ref{def-excised} is interpreted as a flatness criterion: a face $a_0b_0c_0$ can be excised if vanishing curvature at all other faces implies vanishing curvature at $a_0b_0c_0$. 
\end{remark}

\appendix

\section{Auxiliary results from noncommutative algebra}
\label{sec:auxiliary}

Here we prove some standard results from noncommutative algebra used above. We prefer to include these concise proofs instead of references to much more general results in the literature. Notice that these proofs lead to different generalizations; 
see \cite[Lemma~20]{S22} and \cite[p.~235]{Penrose-71}.

We prove Lemma~\ref{l-menelaus} by showing that the well-known proof of Menelaus's theorem using homotheties remains true over a skew field $\mathbb{F}$ (there is also a proof by direct computation \cite[Appendix~A]{arxiv-version-1}). The \emph{(left) homothety with the center $C\in\mathbb{F}^2$ and the ratio $k\in\mathbb{F}$} is the map $
\mathbb{F}^2\to\mathbb{F}^2$ given by $
A\mapsto k(A-C)+C$ 
for each $A\in \mathbb{F}^2$. Clearly, 
the point $A\in\mathbb{F}^2$,  its image $B
$, and the center $C$ lie on one 
line and $[BC/AC]=k$. A line is fixed by a homothety if and only if the line passes through the center or the ratio is $1$ (this is sufficient to show when the center is the origin, in which case a line $xa+yb+c=0$ is taken to the line $xa+yb+kc=0$).
The composition of homotheties is again a homothety unless the product of their ratios is $1$, when it is a translation (because a composition of 
maps of the form $X\mapsto kX+b$ has the same form). 

\begin{proof}[Proof of Lemma~\ref{l-menelaus}]
Consider three homotheties with centers $D$, $E$, $F$ that respectively send $A$ to $B$, $B$ to $C$, and $C$ to $A$. Their composition in order is a homothety or a translation that fixes $A$, hence it is a homothety with center $A$, possibly with ratio $1$ (in which case it is the identity). This composition fixes the line $DE$ if and only if $F$ belongs to $DE$ (since the first two homotheties certainly fix $DE$, and the third does so only if $F$ lies on $DE$). Therefore $D$, $E$, $F$ lie on one line if and only if this composition is the identity, which means that the product of the three ratios is $1$. The latter is equivalent to~\eqref{eq-non-commutative-menelaus} with the left side inverted.
\end{proof}

For the proof of Lemma~\ref{l-plus}, we need an auxiliary notion and a lemma. We say that a face of a triangulated disc is \emph{free} if it contains either two boundary edges or one boundary edge and one nonboundary vertex. (The informal meaning of this condition is that the simplicial complex remains a triangulated disc after removing the face.)

\begin{lemma}[Shellability] \label{l-free} \textup{(See \cite[Lemma~19]{S22} or \cite[Theorem~VI.6.A.]{Bing-83})} If a triangulated disc has more than one face, then it has at least two free faces.
\end{lemma}

\begin{proof}
\cite[Proof of Lemma~19]{S22}
Assume the converse and take a counterexample with a minimal number of faces. The counterexample has more than one face adjacent to the boundary, hence it has a nonfree face $abc$ containing a boundary edge $ab$. Then $bc$ and $ca$ are nonboundary edges and $c$ is a boundary vertex. Then $abc$ splits the disc into two non-empty discs with fewer faces. Since our counterexample is minimal, it follows that each of the two smaller discs has either a unique face or at least two free faces. If one of the smaller discs has a unique face, then the face is free in the original disc as well. If one of the smaller discs has two free faces, then at least one of them contains neither $bc$ nor $ca$, hence remains free in the original disc. We have found two free faces. This contradiction proves the lemma.
\end{proof}

\begin{proof}[Proof of Lemma~\ref{l-plus}] 
  Use induction over the number of faces. If there is a single face, then there is nothing to prove. Otherwise, let $abc$ be a free face given by Lemma~\ref{l-free}. Assume that the edge $ab$ is on the boundary and $ca$ is not. 
  Since the conclusion of the lemma is invariant under a cyclic permutation or reversal 
  of indices, we may assume $a=p_1$ and $b=p_{2}$. Then either $c=p_3$ or $c$ is not on the boundary because $abc$ is free and $ca$ is not on the boundary. 
  
  Remove the face $abc$. 
  We get a triangulated disc with one boundary vertex less ($b=p_{2}$ is deleted if $c=p_{3}$) or one boundary vertex more ($c$ is inserted between $a=p_1$ and $b=p_{2}$ if $c\ne p_{3}$). Applying the inductive hypothesis to the resulting disc and conditions~(E) and~(F), we arrive at the desired equation 
  $$
  U(p_{1}p_{2})U(p_2p_{3})
  \dots U(p_{k}p_{1})
  =
  \begin{cases}
      U(p_1p_3)U(p_{3}p_4)\dots U(p_{k}p_{1}), &\text{if }c=p_3;\\
      U(p_{1}c)U(cp_2)U(p_2p_{3})
      \dots U(p_{k}p_{1}), &\text{if }c\ne p_3;
  \end{cases}
  \qquad=1.\\[-0.8cm]
  $$
\end{proof}

\subsection*{Acknowledgments} The authors are grateful to L.~Beklemishev, J.~Buhler, S.~Fomin, A.~Izosimov, A.~Klyachko, M.~Larson, N.~Lubbes, S.~Melikhov, and A.~Talambutsa for useful discussions. 

\bibliographystyle{elsarticle-num}

\small


\noindent
\textsc{Pavlo Pylyavskyy\\
University of Minnesota, USA}
\\
\texttt{ppylyavs\,@\,umn$\cdot $edu} \quad 
\\[0.2cm]
\textsc{Mikhail Skopenkov\\
King Abdullah University of Science and Technology, Thuwal, Saudi Arabia}
\\
\texttt{mikhail.skopenkov\,@\,gmail$\cdot $com} \quad \url{https://users.mccme.ru/mskopenkov/}

\end{document}